\documentclass{article}
\usepackage[a4paper, total={6in, 8in}]{geometry}
\usepackage[utf8]{inputenc}
\usepackage{amssymb}
\usepackage{amsmath}
\usepackage{amsthm}
\usepackage{standalone}
\usepackage{subfiles}
\usepackage{lineno}
\usepackage{caption}
\usepackage{subcaption}
\usepackage{graphicx}
\usepackage{grffile}
\usepackage{hyperref}
\usepackage{cleveref}
\usepackage{nicefrac}
\usepackage{subcaption}
\usepackage{epstopdf}
\usepackage{mathtools}
\usepackage{tikz}
\usepackage{pgfplots}
\usepackage{adjustbox}
\usepackage{xcolor}
\usepackage{lscape}
\usepackage{booktabs}

\usepackage[framemethod=default]{mdframed}
\newmdenv[linecolor=red,backgroundcolor=gray!10]{myframe}
\pgfplotsset{compat=1.13}
\usetikzlibrary{math}
\usetikzlibrary{intersections}
\usetikzlibrary{calc}
\newtheorem{theorem}{Theorem}[section]
\newtheorem{corollary}[theorem]{Corollary}
\newtheorem{lemma}[theorem]{Lemma}
\newtheorem{definition}[theorem]{Definition}
\newtheorem{remark}[theorem]{Remark}
\newtheorem{proposition}[theorem]{Proposition}
\newcommand{\ie}{\textit{i.e.}\,}
\newcommand{\eg}{\textit{e.g.}\,}
\newcommand{\R}{\mathbb R}

\newcommand{\eps}{\varepsilon}

\newcommand{\black}{\color{black}}

\newcommand{\dd}{\;\mathrm{d} }
\newcommand{\ddx}{\;\mathrm{d} {\mathbf x}}

\newcommand{\mesh}{\mathcal{T}} 
\newcommand{\pos}{{\mathbf x}}
\newcommand{\levelset}{\phi}
\newcommand{\setIk}{I_{\pos_k}}

\newcommand{\setRk}{R_{\pos_k}}
\newcommand{\detJl}{\left|{\textup{\text{det}}}J_l\right|}

\newcommand{\targetU}{\hat u}

\newcommand{\objectiveFunction}{g}
\newcommand{\phiObjectiveFunction}{J}
\newcommand{\redObjectiveFunction}{\mathfrak g} 
\newcommand{\redPhiObjectiveFunction}{\mathcal J} 
\newcommand{\Tplus}{T^{-\rightarrow+}_{k,\varepsilon}}
\newcommand{\Tminus}{T^{+\rightarrow-}_{k,\varepsilon}}
\renewcommand{\S}{S_{k,\varepsilon}}
\newcommand{\SComplex}{ S_{k,ih}}
\newcommand{\TplusComplex}{T^{-\rightarrow+}_{k, ih}}
\newcommand{\TminusComplex}{T^{+\rightarrow-}_{k,ih}}

\newcommand{\operator}{O} 
\newcommand{\discreteOperator}{O_{k,\varepsilon}} %
\newcommand{\setTplus}{\mathfrak T^+}
\newcommand{\setTminus}{\mathfrak T^-}
\newcommand{\setS}{\mathfrak S}

\newcommand{\symmDiff}{\triangle}
\newcommand{\defC}{c_t}

\newcommand{\changeArea}{\delta_{k,\eps} a}
\newcommand{\topshapeDerivative}{{topological-shape }}

\newcommand{\mylim}{\underset{t \searrow 0}{\mbox{lim }}}

\newcommand{\stuff}[1]{{\footnotesize \color{gray}}}

\newcommand{\BL}{\dot{BL}(\mathbb R^2)}
\usepackage{ulem}
\usepackage{comment}
\usepackage{xargs}
\usepackage{xcolor}
\usepackage[colorinlistoftodos,prependcaption,textsize=tiny]{todonotes}
\usepackage{pdfpages} 

\newcommandx{\pcomment}[2][1=]{\todo[linecolor=red,backgroundcolor=red!25,bordercolor=red,#1]{#2}}
\newcommandx{\mcomment}[2][1=]{\todo[linecolor=olive,backgroundcolor=olive!25,bordercolor=olive,#1]{#2}}

\begin{document}



\title{A unified approach to shape and topological sensitivity analysis of discretized optimal design problems}


\author{P. Gangl$^{1,2,3}$ and M.H. Gfrerer$^{4}$}

\date{$^1$Johann Radon Institute for Computational and Applied Mathematics, Altenberger Stra\ss{}e 69, 4040 Linz, Austria\\
$^2$Institute of Applied Mathematics, Graz University of Technology, Steyrergasse 30/III, 8010 Graz, Austria \\
$^3$Chair of Applied Mathematics (Continuous Optimization), Friedrich Alexander University Erlangen-N\"urnberg, Cauerstra\ss{}e 11, 91058 Erlangen, Germany\\
$^4$Institute of Applied Mechanics, Graz University of Technology, Technikerstrasse 4, 8010 Graz, Austria}

%

\maketitle

\begin{abstract}
We introduce a unified sensitivity concept for shape and topological perturbations and perform the sensitivity analysis for a discretized PDE-constrained design optimization problem in two space dimensions. We assume that the design is represented by a piecewise linear and globally continuous level set function on a fixed finite element mesh and relate perturbations of the level set function to perturbations of the shape or topology of the corresponding design. We illustrate the sensitivity analysis for a problem that is constrained by a reaction-diffusion equation and draw connections between our discrete sensitivities and the well-established continuous concepts of shape and topological derivatives. Finally, we verify our sensitivities and illustrate their application in a level-set-based design optimization algorithm where no distinction between shape and topological updates has to be made.

\end{abstract}

\tableofcontents

%
%
%
%



%
%

\section{Introduction}
Numerical methods for the design optimization of technical systems are of great interest in science and engineering. Applications include the optimization of mechanical structures \cite{sigmund2013topology,allaire2004structural}, electromagnetic devices \cite{GanglLangerLaurainMeftahiSturm2015,AmstutzGangl2019}, fluid flow \cite{HaubnerSiebenbornUlbrich2021}, heat dissipation \cite{Hagg2018} and many more.
There exist several different approaches to computational design optimization. On the one hand, shape optimization techniques based on the mathematical concept of shape derivatives \cite{DZ2} can modify boundaries and material interfaces in a smooth way, but typically cannot alter the topology of a design. An exception being the level set method for shape optimization \cite{allaire2004structural} where the design is represented by the zero level set of a design function whose evolution is guided by shape sensitivities via a transport equation. While this approach allows for merging of components, it lacks a nucleation mechanism and is often coupled with the topological derivative concept \cite{sokolowski1999topological,b_NovoSoko}, see e.g. \cite{burger2004incorporating,AllaireJouve:2006a}.
In the class of density-based topology optimization methods \cite{bendsoe2003topology}, a design is represented by a density function $\rho(x)$ that is allowed to attain any value in the interval $[0,1]$. Then, regions with $\rho(x) = 0$ and $\rho(x) = 1$ are interpreted as occupied by material 1 and 2, respectively, while intermediate density values $0<\rho(x) < 1$ are penalized in order to obtain designs that are almost ``black-and-white''. One advantage of density based methods is that the system response depends continuously on $\rho$ and the standard notions of derivatives in vector spaces can be applied. Here, interfaces are typically not crisp and there is no measure of optimality with respect to shape variations at the interface.
Finally we mention the level-set algorithm for topology optimization introduced in \cite{amstutz2006new}, where the design is guided solely by the topological derivative, which however is not defined on the material interfaces. As a consequence, the final designs cannot be shown to be optimal with respect to shape variations at the interface. This aspect has been thorougly analyzed in \cite{amstutz2018consistent} where the authors draw a connection to density-based methods and, for two particular problem classes, propose an interpolation scheme which relates the derivative with respect to the density function, to topological and shape derivatives in the interior and on the interface, respectively.

The goal of this paper is to unify the concepts of topological and shape perturbations and to treat design optimization problems by a unified sensitivity, called the \textit{\topshapeDerivative derivative}. In this way, we aim at combining topological sensitivity information (related to the topological derivative) in the interior of each subdomain and shape sensitivity information (related to the shape derivative) at the material interface. While the topological derivative is defined as the sensitivity of a design-dependent cost function with respect to the introduction of a small hole or inclusion of different material, the shape derivative is defined as the cost function's sensitivity with respect to a transformation of the domain. In order to unify these two concepts, we consider a domain description by means of a continuous level set function which attains positive values in one of two subdomains and negative values in the other. Then a perturbation of the level set function in the interior of a subdomain can be related to a topological perturbation, and a perturbation close to the material interface can be seen as a perturbation of the shape of the domain. We remark that this point of view is in alignment with the concept of dilations of points and curves as introduced in \cite{delfour2018topological}, see also \cite{Delfour_engcomp_2022}. In principle, this idea was already followed in \cite{bernland2018acoustic}, however only for the case of shape optimization and not in combination with topology optimization. In \cite{laurain2018analyzing} the author represents domains by level set functions and relates shape and topological derivatives of shape functionals to derivatives with respect to the level set function in a continuous setting without PDE constraints.
In contrast to this, we consider PDE-constrained problems, but our analysis is performed on the discrete level, i.e. we follow the paradigm ``discretize-then-optimize'' for our sensitivity analysis with respect to a level set function.

The rest of this paper is organized as follows: In Section \ref{s_contiSetting}, we introduce the model problem and the classical concepts of topological and shape derivative in the continuous setting. After presenting the discretized setting in Section \ref{sec::discretization}, we proceed to compute the numerical \topshapeDerivative derivative of our discretized model problem in Section \ref{sec::numTopShapeDer}. In Section \ref{sec::connection} we compare the computed sensitivities with the sensitivities obtained by discretizing the continuous formulas. Finally we verify our computed formulas and present optimization results in Section \ref{sec::numericalExamples} before giving a conclusion in Section \ref{sec::conclusion}.

\section{Model problem and continuous setting} \label{s_contiSetting}
Let $D$ be a given, fixed, open and bounded hold-all domain and $\Omega \subset D$ an open and measurable subset. Let the boundary of $D$ be decomposed into $\Gamma_D, \Gamma_N \subset D$ with $\overline \Gamma_D \cup \overline \Gamma_N = \partial D$ and $\Gamma_D \cap \Gamma_N = \emptyset$.
In the present paper, we consider a topology optimisation problem  with a tracking type cost function
\begin{equation}\label{eq::ContinuousObjective}
\begin{aligned}
\objectiveFunction(\Omega,u) = c_1 |\Omega| + c_2
\int_D \tilde\alpha_\Omega |u-\targetU|^2 \dd x 
\end{aligned}
\end{equation}
where $\targetU \in H^1(D)$ is a given desired state, and $c_1$, $c_2 \in \R$ are given constants.
\black
The continuous topology optimization problem reads
\begin{subequations}\label{eq::ContinuousProblem}
\begin{alignat}{2}
\min_{\Omega \in \mathcal A} \; \objectiveFunction(\Omega,u), & \\
\text{ subject to}& \nonumber \\	
 - \lambda_\Omega \Delta u + \alpha_\Omega u&= f_\Omega  &&\text{in } D, \\
u &= g_D \quad &&\text{on } \Gamma_D, \\
\lambda_\Omega \partial_n u &= g_N  &&\text{on } \Gamma_N,
\end{alignat}
\end{subequations}
where 
\begin{align*}
\tilde\alpha_\Omega(x) =& \chi_\Omega(x)\tilde\alpha_1 + \chi_{D \setminus \Omega}(x) \tilde\alpha_2, \qquad
\lambda_\Omega(x) = \chi_\Omega(x)\lambda_1 + \chi_{D \setminus \Omega}(x) \lambda_2,\\
\alpha_\Omega(x) =& \chi_\Omega(x)\alpha_1 + \chi_{D \setminus \Omega}(x) \alpha_2, \qquad
f_\Omega(x) = \chi_\Omega(x)f_1 + \chi_{D \setminus \Omega}(x) f_2, 
\end{align*}
for some constants $\lambda_1, \lambda_2 > 0$, $\alpha_1, \alpha_2, \tilde \alpha_1, \tilde \alpha_2 \geq 0$ and $f_1, f_2 \in \mathbb R$ with $\chi_S$ the characteristic function of a set $S$,
\begin{align*}
    \chi_S(x) = \begin{cases}
                    1, & x \in S, \\
                    0, & \mbox{ else}.
                \end{cases}
\end{align*}
Here, $\mathcal A$ denotes a set of admissible subsets of $D$, and the data $g_D \in H^{1/2}(\Gamma_D)$, $g_N \in L^2(\Gamma_N)$ are given. 
The weak formulation of the PDE constraint reads
\begin{align}
    \mbox{Find }u\in V_g:=\{v \in H^1(D): v = g_D \mbox{ on } \Gamma_D \} \mbox{ such that } \nonumber \\
    \int_D \lambda_\Omega \nabla u \cdot \nabla v + \alpha_\Omega u v \; \mbox dx = \int_D f_\Omega v \; \mbox dx \quad  \mbox{ for all }v \in V_0 \label{eq_weakForm}
\end{align}
with $V_0 = \{v \in H^1(D): v = 0 \mbox{ on } \Gamma_D \}$. We assume that either $|\Gamma_D|>0$ or $\alpha_1, \alpha_2>0$ such that, for given $\Omega\in \mathcal A$, \eqref{eq_weakForm} admits a unique solution which we denote by $u(\Omega)$. We introduce the reduced cost function $\redObjectiveFunction(\Omega):= \objectiveFunction(\Omega, u(\Omega))$.

\subsection{Classical topological derivative}\label{sec::contiTopDerivative}
Let $\omega \subset \mathbb R^d$ with $0 \in \omega$.
For a point $z \in \Omega \cup (D \setminus \overline \Omega)$, let $\omega_\varepsilon := z + \varepsilon \omega$ denote a perturbation of the domain around $z$ of (small enough) size $\varepsilon$ and of shape $\omega$.
The continuous topological derivative of the shape function $\redObjectiveFunction=\redObjectiveFunction(\Omega)$ is defined by 
\begin{equation} \label{eq_defTD}
d\redObjectiveFunction(\Omega)(z) = \begin{cases}
\lim_{\varepsilon\searrow 0}\frac{\redObjectiveFunction(\Omega\cup\omega_\varepsilon)-\redObjectiveFunction(\Omega)}{|\omega_\varepsilon|} \quad \text{for } z \in D \setminus \overline \Omega, \\ 
\lim_{\varepsilon\searrow 0}\frac{\redObjectiveFunction(\Omega\setminus\bar\omega_\varepsilon)-\redObjectiveFunction(\Omega)}{|\omega_\varepsilon|} \quad \text{for } z \in \Omega.
\end{cases}
\end{equation}
Note that the topological derivative is not defined for points $z \in \partial \Omega$ on the material interface.
For problem \eqref{eq::ContinuousProblem} we obtain for $z \in D \setminus \overline \Omega$
\begin{align}
    \begin{aligned}
d\redObjectiveFunction_2(\Omega)(z) =&c_1+c_2(\tilde \alpha_1 - \tilde \alpha_2)(u(z)-\targetU(z))^2 \\
&+  2\lambda_2\frac{\lambda_1-\lambda_2}{\lambda_1+\lambda_2}\nabla u(z) \cdot \nabla p(z) +(\alpha_1-\alpha_2)u(z)p(z)- ( f_1 - f_2)p(z), 
    \end{aligned}
\end{align} 
whereas for  $z \in \Omega$
\begin{align}
\begin{aligned}
d\redObjectiveFunction_2(\Omega)(z) =&-c_1+c_2 (\tilde \alpha_2 - \tilde \alpha_1)(u(z)-\targetU(z))^2 \\
&+ 2\lambda_1\frac{\lambda_2-\lambda_1}{\lambda_1+\lambda_2}\nabla u(z) \cdot \nabla p(z) +(\alpha_2-\alpha_1)u(z)p(z)- ( f_2 - f_1)p(z),
\end{aligned}
\end{align}
see, e.g. \cite{Amstutz_2006aa}.

\subsection{Classical shape derivative}\label{sec::contiDerivative}
We recall the definition of the classical shape derivative as well as its formula for our model problem \eqref{eq::ContinuousObjective}--\eqref{eq::ContinuousProblem}. Given an admissible shape $\Omega \in \mathcal A$ and a smooth vector field $V \in C_c^\infty(D)$ that is compactly supported in $D$, we define the perturbed domain
\begin{equation}\label{eq::pertDomain}
\Omega_t = (\text{id}+tV)(\Omega),
\end{equation}
for a small perturbation parameter $t > 0$ where $\text{id}:\mathbb R^d \rightarrow \mathbb R^d$ denotes the identity operator. The classical shape derivative of $\redObjectiveFunction$ at $\Omega$ with respect to $V$ is then given by
\begin{equation}\label{eq::shapeDerivative}
d\redObjectiveFunction(\Omega)(V) = \lim_{t\searrow 0} \frac{\redObjectiveFunction(\Omega_t)-\redObjectiveFunction(\Omega)}{t}
\end{equation}
if this limit exists and the mapping $V \mapsto d \redObjectiveFunction(\Omega)(V)$ is linear and continuous.
Under suitable assumptions it can be shown that this shape derivative admits the tensor representation
\begin{equation}\label{eq::shapeDerivative_tensor}
d\redObjectiveFunction(\Omega)(V) 
=\int_D \mathcal S_1^\Omega : \partial V + \mathcal S_0^\Omega \cdot V dx,
\end{equation}
for some tensors $\mathcal S_0^\Omega \in L^1(D, \mathbb R^{d} )$, $\mathcal S_1^\Omega \in L^1(D, \mathbb R^{d\times d} )$ \cite{LaurainSturm2016}. Here, $\partial V$ denotes the Jacobian of the vector field $V$. The structure theorem of Hadamard-Zol\'esio \cite[pp. 480-481]{DZ2} states that under certain smoothness assumptions the shape derivative of a shape function $\redObjectiveFunction$ with respect to a vector field $V$ can always be written as an integral over the boundary of a scalar function $L$ times the normal component of $V$, i.e.,
\begin{align}\label{eq::shapeDerivative_boundary}
d\redObjectiveFunction(\Omega)(V) = \int_{\partial \Omega} L \,(V \cdot n) \; \mbox ds
\end{align}
where $n$ denotes the unit normal vector pointing out of $\Omega$. For problem \eqref{eq::ContinuousProblem} one obtains \cite{LaurainSturm2016}
\begin{align} \label{eq_S1}
\mathcal S_1^\Omega =& (c_1 \chi_\Omega + c_2 \tilde \alpha_\Omega (u-\targetU)^2 + \lambda_\Omega \nabla u \cdot \nabla p + \alpha_\Omega u p - f_\Omega p) I - \lambda_\Omega \nabla u \otimes \nabla p - \lambda_\Omega \nabla p \otimes \nabla u, \\
\mathcal S_0^\Omega =& -2 \tilde \alpha_\Omega (u-\targetU) \nabla \targetU \label{eq_S0}
\end{align}
where $I \in \mathbb R^{d,d}$ denotes the identity matrix, and
\begin{align*}
L =& [(\mathcal S_1^{\Omega,\text{in}} - \mathcal S_1^{\Omega,\text{out}}) n ]\cdot n \\
=& c_1 + c_2 (\tilde \alpha_1- \tilde \alpha_2) (u-\targetU)^2 + (\alpha_1-\alpha_2) u p - (f_1 - f_2) p  \\
&+ (\lambda_1- \lambda_2)  (\nabla u \cdot \tau)(\nabla p \cdot\tau) -  \left(\frac{1}{\lambda_1} - \frac{1}{\lambda_2}\right)( \lambda_\Omega \nabla u \cdot n)(  \lambda_\Omega \nabla p \cdot n).
\end{align*}
Here, $\mathcal S_1^{\Omega,\text{in}}$ and $\mathcal S_1^{\Omega,\text{out}}$ denote the restrictions of the tensor $\mathcal S_1^{\Omega}$ to $\Omega$ and $D \setminus \Omega$, respectively. Furthermore, for two column vectors $a, b \in \mathbb R^d$, $a \otimes b = a b^\top \in \mathbb R^{d \times d}$ denotes their outer product, $\tau$ denotes the tangential vector and $p \in H^1_0(D)$ is the solution to the adjoint equation
\begin{align*}
    \int_D \lambda_\Omega \nabla v \cdot \nabla p + \alpha_\Omega v p \; \mbox dx = - 2c_2 \int_D \tilde \alpha_\Omega (u - \hat u)v \; \mbox dx \quad \mbox{for all } v \in H^1_0(D).
\end{align*}

Moreover, motivated by the definition of the topological derivative \eqref{eq_defTD} with the volume of the difference of the perturbed and unperturbed domains in the denominator, we introduce the alternative definition of a shape derivative
\begin{align} \label{eq::shapeDerivativeAlternative}
\hat d \redObjectiveFunction(\Omega)(V) = \underset{t \searrow 0}{\mbox{lim }} \frac{\redObjectiveFunction(\Omega_t)-\redObjectiveFunction(\Omega)}{|\Omega_t \symmDiff \Omega|},
\end{align}
with the symmetric difference of two sets $A \symmDiff B := (A \setminus \overline B) \cup (B \setminus \overline A)$. Note that the volume of the symmetric difference in \eqref{eq::shapeDerivativeAlternative} can be written as
\begin{equation}\label{eq::diffSets}
|\Omega_t \symmDiff \Omega| = \int_D |\chi_{\Omega_t} - \chi_{\Omega}| \ddx.
\end{equation}

\begin{lemma} \label{LEM_SYMDIF}
    Let $\Omega$ and $V$ smooth. It holds
    \begin{align*}
        \underset{t \searrow 0}{\mbox{lim }}\frac1t \left \lvert \Omega_t \symmDiff \Omega \right \rvert = \int_{\partial \Omega} |V \cdot n| \; \mbox dS_x.
    \end{align*}
\end{lemma}

The proof is given in Appendix \ref{app::proofSymDif}. 
From Lemma \ref{LEM_SYMDIF}, we immediately obtain the following relation between $d \redObjectiveFunction(\Omega)(V)$ and $\hat d \redObjectiveFunction(\Omega)(V)$.

\begin{corollary}
Suppose that $\redObjectiveFunction$ is shape differentiable at $\Omega$ and that $\Omega$ and $V$ are smooth and $\int_{\partial \Omega} |V \cdot n| \; dS_x > 0$. Then it holds
\begin{align}
\hat d \redObjectiveFunction(\Omega)(V) = \frac{ d \redObjectiveFunction(\Omega)(V)}{ \int_{\partial \Omega} |V \cdot n| \; dS_x}.
\end{align}
\end{corollary}
\begin{proof}
	This follows immediately from the definition of $\hat d \redObjectiveFunction(\Omega)(V)$ by Lemma \ref{LEM_SYMDIF} since
	\begin{align}
	\hat d \redObjectiveFunction(\Omega)(V) =  \underset{t \searrow 0}{\mbox{lim }}\frac{\redObjectiveFunction(\Omega_t)-\redObjectiveFunction(\Omega)}{|\Omega_t \symmDiff \Omega|} =  \frac{\underset{t \searrow 0}{\mbox{lim }}\frac{\redObjectiveFunction(\Omega_t)-\redObjectiveFunction(\Omega)}{t}}{\underset{t \searrow 0}{\mbox{lim }}\frac{|\Omega_t \symmDiff \Omega|}{t}} =  \frac{ d \redObjectiveFunction(\Omega)(V)}{\int_{\partial \Omega} |V \cdot n| \; dS_x}. 
	\end{align}	
	
\end{proof}
\begin{remark}The condition $\int_{\Omega} \mbox{div}(V) \; dx \neq 0$ is a sufficient condition for $\int_{\partial \Omega} |V \cdot n| \; dS_x > 0$, since
\begin{equation}
0 < \left|\int_{\Omega} \mbox{div}(V) \; dx \right| = \left|\int_{\partial \Omega} V \cdot n \; dS_x \right|< \int_{\partial \Omega} |V \cdot n| \; dS_x.
\end{equation}	
\end{remark}

\subsection{The continuous \topshapeDerivative derivative}
Here and in the following, we assume that the domain $\Omega$ is described by a level-set function $\levelset: D \rightarrow \R$ via
\begin{subequations}
\begin{align}
	\levelset(\pos) < 0 &\Longleftrightarrow \pos \in \Omega, \label{eq::levelSet1} \\
	\levelset(\pos) > 0 &\Longleftrightarrow \pos \in D \setminus \overline{\Omega} \label{eq::levelSet2}\\
	\levelset(\pos) = 0 &\Longleftrightarrow \pos \in \partial \Omega \cap D. \label{eq::levelSet3}
\end{align}
\end{subequations}
For given $\levelset$, let $\Omega(\levelset)$ denote the unique domain defined by \eqref{eq::levelSet1}--\eqref{eq::levelSet3}.
In this section, in contrast to the setting in Section \ref{sec::contiDerivative}, we perturb $\Omega$ indirectly by perturbing $\levelset$ such that
$\levelset_\varepsilon = \operator_\varepsilon \levelset$
for some operator $\operator_\varepsilon: C^0(D)\rightarrow C^0(D)$ depending on $\varepsilon \geq 0$ with the property $\Omega(\operator_0 \levelset) = \Omega(\levelset)$.
Later on, in the discrete setting, we will distinguish between two different types of perturbation operators $\operator_\eps$ corresponding to shape or topological perturbations of~$\Omega$.

Let, from now on, $\redPhiObjectiveFunction(\levelset):= \redObjectiveFunction(\Omega(\levelset))$ denote the reduced cost function as a function of the level set function $\phi$.
 This way, a continuous \topshapeDerivative derivative can be defined as 
\begin{equation}\label{eq::topologicalShapeDerivative}
	d\redPhiObjectiveFunction(\levelset) = \lim_{\varepsilon\searrow 0}\frac{\redPhiObjectiveFunction(\levelset_\varepsilon)-\redPhiObjectiveFunction(\levelset)}{|\Omega(\levelset_\varepsilon) \symmDiff \Omega(\levelset)|}.
\end{equation}
Note that this sensitivity depends on the choice of the perturbation operator $\operator_\varepsilon$, which can represent either a shape perturbation or a topological perturbation.
We will mostly be concerned with its discrete counterpart, which will be introduced in Section \ref{sec::numTopShapeDer}.
Note that, in the case of shape perturbations, due to the scaling $|\Omega(\levelset_\varepsilon)\symmDiff \Omega(\levelset)|$ instead of $\varepsilon$ in the denominator the shape derivative is modified and does not correspond to \eqref{eq::shapeDerivative} but rather to \eqref{eq::shapeDerivativeAlternative}.

\paragraph{Relation to literature}

\newcommand{\Gateaux}{G\^{a}teaux }
The sensitivity of shape functions with respect to perturbations of a level set function (representing a shape) was investigated in \cite{laurain2018analyzing} for the case without PDE constraints. There, the author considers smooth level set functions and rigorously computes the \Gateaux (semi-)derivative in the direction of a smooth perturbation of the level set function, both for the case of shape and topological perturbations. In the case of shape perturbations, it is shown that the \Gateaux derivative coincides with the shape derivative \eqref{eq::shapeDerivative} with respect to a suitably chosen vector field. On the other hand, a resemblance between the notions of \Gateaux derivative and topological derivative is shown, yet the \Gateaux derivative may vanish or not exist in cases where the topological derivative is finite.
Evidently, this discrepancy results from the fact that the denominator in the definition of the \Gateaux derivative is always of order one whereas it is of the order of the space dimension in the topological derivative.

While the analysis for shape and topological perturbations is carried out separately in \cite{laurain2018analyzing}, a more unified approach is followed in \cite{delfour2018topological,Delfour_engcomp_2022}. In these publications, the idea is to consider sensitivities with respect to domain perturbations that are obtained by the dilation of lower-dimensional objects. Here, given a set $E \subset \mathbb R^d$ of dimension $k \leq d$, the dilated set of radius $r$ is given by $E_r = \{x \in \mathbb R^d: d_E(x) \leq r\}$ where $d_E(x)$ denotes a distance of a point $x$ to a set $E$. For instance, when $E$ is chosen as a single point, the dilated set is just a ball of radius $r$ around that point and performing a sensitivity analysis with respect to the volume of the dilated object leads to the topological derivative. On the other hand, when $E$ is chosen as the boundary of a domain, $E_r$ can be defined using a signed distance function and corresponds to a uniform expansion of the domain. Then, a similar procedure leads to the shape derivative with respect to a uniform expansion in normal direction  (i.e. $V = n$ in \eqref{eq::pertDomain}--\eqref{eq::shapeDerivative}).
In \cite{delfour2018topological}, a sensitivity analysis for various choices of $E$ is carried out with respect to the volume of the perturbation. We note, however, that arbitrary shape perturbations are not covered and would require an extension of the theory. Comparing \cite{laurain2018analyzing} and \cite{delfour2018topological}, we observe that in the former paper only smooth perturbations of a level set function are admissible whereas, in the latter approach, domain perturbations by dilations can be interpreted as perturbations of level set functions by a (non-smooth) distance function.

Finally, we mention \cite{bernland2018acoustic} where a domain is represented by a discretized level set function and a shape sensitivity analysis is carried out with respect to a perturbation of the level set values close to the boundary. This procedure can be interpreted as an application of the idea of dilation to discretized shape optimization problems.
As the authors point out, this kind of shape sensitivity analysis is more natural compared to the standard approach based on domain transformations when employed in a level-set framework; an observation also made in \cite[Sec. 3]{laurain2018analyzing}. The authors show numerical results for the shape optimization of an acoustic horn, but do not consider topological perturbations in this work.

\medskip

As it can be seen from \eqref{eq::topologicalShapeDerivative}, our approach is related to the dilation concept since we also consider the sensitivity with respect to the volume of the domain perturbation $\Omega_t \symmDiff \Omega$. In the following, we will investigate the \topshapeDerivative derivative in a discretized setting. Similarly to \cite{bernland2018acoustic}, we will consider shape sensitivity analysis with respect to level set values on mesh nodes close to the boundary. Moreover, we will also be able to deal with topological perturbations and treat shape and topological updates in a unified way by a discretized version of \eqref{eq::topologicalShapeDerivative}, called the numerical \topshapeDerivative derivative.
%
%
%

%
%


\section{Numerical setting}\label{sec::discretization}
In this section we consider the discretization of \eqref{eq::ContinuousProblem}. 
Let $\mesh$ be a given finite element mesh covering $D$ with $M$ nodes $\{\pos_k\}_{k=1}^{M}$ and $N$ triangular elements $\{\tau_l\}_{l=1}^N$. We introduce the index set $\setIk$ of all element indices of elements $\tau_l$ where $\pos_k$ is a node of $\tau_l$,
\begin{equation}
\setIk := \{l \in {\{1, \dots, N \} } : \pos_k \in \bar\tau_l\} \quad \text{for } k = 1,\dots ,M. 
\end{equation}
Moreover,
\begin{equation}
J_{\tau_l} := \{k \in {\{1, \dots, M \} } : \pos_k \in \bar \tau_l\} \quad \text{for } l = 1,\dots ,N
\end{equation}
is the index set of all node indices of nodes $\pos_k$ in $\bar \tau_l$. Furthermore, we introduce the one-ring of a node $\pos_k$,
\begin{equation}
	R_{\pos_k} := \{i \in {\{1, \dots, M \} } |\exists l \in \setIk: \pos_i \in \bar \tau_l\} \quad \text{for } k = 1,\dots ,M.
\end{equation}
These sets are illustrated in Figure \ref{fig::FEMsets}.
\begin{figure}
	\begin{subfigure}[t]{0.26\textwidth}
		\includegraphics{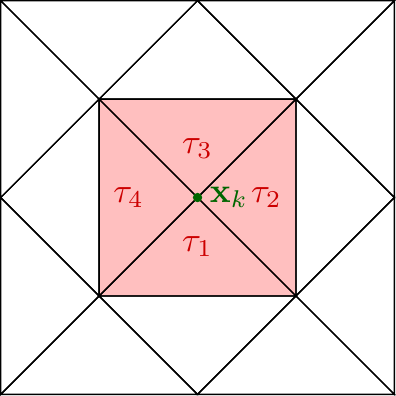}
		\subcaption{$\setIk = \{1,~2,~3,~4\}$}
	\end{subfigure}\hfill
	\begin{subfigure}[t]{0.26\textwidth}
		\includegraphics{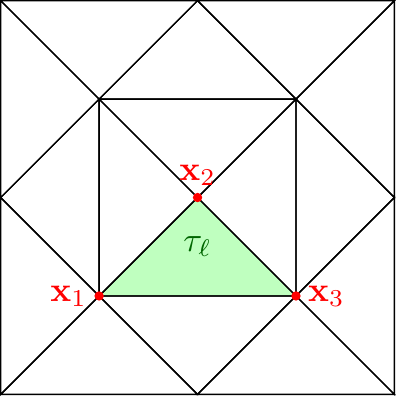}
		\subcaption{$J_{\tau_l} = \{1,~2,~3\}$}	
	\end{subfigure}\hfill
	\begin{subfigure}[t]{0.26\textwidth}
		\includegraphics{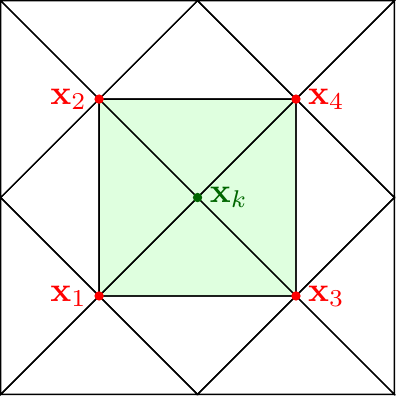}
		\subcaption{$R_{\pos_k} = \{k,~1,~2,~3,~4\}$}	
	\end{subfigure}
	\caption{Illustration of the sets $\setIk$, $J_{\tau_l}$, and $R_{\pos_k}$.}
	\label{fig::FEMsets}
\end{figure}
Let $P_1 =\{a+bx_1 + cx_2: a,b ,c \in \mathbb R\}$ denote the space of affine linear polynomials in two space dimensions and $S_h^1(D)$ the space of piecewise affine linear and globally continuous functions on $\mesh$,
\begin{align*}
S_h^1(D) = \{ v \in H^1(D): v|_{T} \in P_1 \mbox{ for all } T \in \mesh \} = \mbox{span} \{ \varphi_1, \dots, \varphi_M \}
\end{align*}
with the hat basis functions $\varphi_i \in S_h^1(D)$ which satisfy $\varphi_i(\pos_j) = \delta_{ij}$, $i,j = 1, \dots, M$.
The discretization of problem \eqref{eq::ContinuousProblem} leads to the discretized optimization problem
\begin{subequations} \label{eq_discr_opti_problem}
\begin{align}
\underset{\Omega}{\mbox{min }} &c_1|\Omega| + c_2 (\mathbf u - \hat{\mathbf u})^\top \tilde{\mathbf M}(\mathbf u - \hat{\mathbf u}) \\
&\mbox{ subject to }
\mathbf A \mathbf u = \mathbf f, \label{eq::linearSystem}
\end{align}
\end{subequations}
with the solution vector $\mathbf u \in \R^{M}$ and $\mathbf A = \mathbf M + \mathbf K$. Here, the mass matrices $\mathbf M , \tilde{\mathbf M} \in \R^{M\times M}$, the stiffness matrix $\mathbf K \in \R^{M\times M}$, and the right-hand-side vector $\mathbf f \in \R^{M}$ depend on the shape $\Omega$ and are given by
\begin{equation}\label{eq_defMtilde}
    \begin{aligned}
\mathbf M[i,j] &= \int_{D} \alpha_{\Omega} \varphi_j  \varphi_i \ddx, &&&
\mathbf K[i,j] &= \int_{D} \lambda_{\Omega} \nabla \varphi_j \cdot \nabla \varphi_i \ddx,  \\
\mathbf f[i] &= \int_{D}  f_{\Omega}\varphi_i \,   \ddx, &&&
\tilde {\mathbf M}[i,j] &= \int_{D} \tilde \alpha_{\Omega}\varphi_j  \varphi_i \ddx.
    \end{aligned}
\end{equation}
On the reference element $\tau_R = \{\xi\in\R^2:0\le\xi_1\le 1,0\le \xi_2\le1-\xi_1\}$ we have the local form functions
\begin{align*}
\psi_1(\xi_1, \xi_2) &= 1 -\xi_1 -\xi_2, &&&
\psi_2(\xi_1, \xi_2) &= \xi_1, &&&
\psi_3(\xi_1, \xi_2) &= \xi_2.
\end{align*}
For an element $\tau_l \in \mesh$, we denote the global vertex indices of its three vertices by $l_1$, $l_2$, $l_3$ and assume them to be numbered in counter-clockwise orientation.
Then, the respective local finite element matrices and the local right-hand-side vector for element $\tau_l$ are given by   
\begin{align*}
	\mathbf m_l[i,j] &= \detJl \int_{\xi_1=0}^1\int_{\xi_2=0}^{1-\xi_1} {(\alpha_{\Omega} \circ \Phi_l)} \psi_j \psi_i \dd\xi_2 \dd\xi_1, \\
	\mathbf k_l[i,j] &= \detJl \left(J_l^{-1}\nabla_\xi \psi_j\right)^\top \left(J_l^{-1}\nabla_\xi \psi_i\right) \int_{\xi_1=0}^1\int_{\xi_2=0}^{1-\xi_1} {(\lambda_{\Omega} \circ \Phi_l)}  \dd\xi_2 \dd\xi_1, \\
	\mathbf f_l[i] &= \detJl \int_{\xi_1=0}^1\int_{\xi_2=0}^{1-\xi_1}  {(f_{\Omega} \circ \Phi_l)} \, \psi_i    \dd\xi_2 \dd\xi_1, \\
\end{align*}
for $i, j \in \{1,2,3\}$, where $\nabla_\xi \psi_j = [\partial_{\xi_1}\psi_j , \partial_{\xi_2}\psi_j]^\top$ and the mapping $\Phi_l$ between $\tau_R$ and $\tau_l$ and its Jacobian $J_l$ are given by 
\begin{align*}
\Phi_l(\xi_1, \xi_2) = \pos_{l_1} + J_l \begin{bmatrix}\xi_1 \\ \xi_2\end{bmatrix}, \qquad J_l = \begin{bmatrix}
\pos_{l_2}-\pos_{l_1} & \pos_{l_3}-\pos_{l_1}
\end{bmatrix} \in \mathbb R^{2 \times 2}.
\end{align*}

\newcommand{\uu}{\mathbf u}
\newcommand{\pp}{\mathbf p}
\newcommand{\qq}{\mathbf q}
\newcommand{\yy}{\mathbf y}
\newcommand{\vv}{\mathbf v}
\newcommand{\ww}{\mathbf w}
\newcommand{\KK}{\mathbf K}
\newcommand{\MM}{\mathbf M}
\newcommand{\Amat}{\mathbf A}
\newcommand{\ff}{\mathbf f}
\section{Numerical \topshapeDerivative derivative} \label{sec::numTopShapeDer}
Given the discretization introduced in Section \ref{sec::discretization}, in contrast to the continuous \topshapeDerivative derivative, the numerical \topshapeDerivative derivative is only defined at the nodes of the finite element mesh $\mesh$. For a given piecewise linear level set function $\levelset \in S_h^1(D)$ let $\Omega(\levelset)$ be defined by \eqref{eq::levelSet1}--\eqref{eq::levelSet3} and $\redPhiObjectiveFunction(\levelset) = \objectiveFunction(\Omega(\levelset),u(\levelset))$, where $u(\levelset) \in {V_h} = \{v \in S_h^1(D): v = g_D \mbox{ on } \Gamma_D \}$ is the finite element function corresponding to the solution of \eqref{eq::linearSystem}. Note that, in this section, $\Omega(\phi)$ is polygonal since $\phi \in S_h^1(D)$. The \topshapeDerivative derivative at node $\pos_k \in \mesh$ is defined by
\begin{equation}\label{eq::topDerivative}
d\redPhiObjectiveFunction(\levelset)(\pos_k) = \begin{cases}
\lim_{\varepsilon\searrow 0}
\frac{\redPhiObjectiveFunction({\Tplus} \levelset)-\redPhiObjectiveFunction(\levelset)}
{|\Omega({\Tplus}\levelset) \symmDiff  \Omega(\levelset)|} \quad \text{for } \pos_k \in \setTminus(\phi), \\[10pt]
\lim_{\varepsilon\searrow 0}
\frac{\redPhiObjectiveFunction({\Tminus} \levelset)-\redPhiObjectiveFunction( \levelset)}{|\Omega({\Tminus} \levelset) \symmDiff  \Omega(\levelset)|}
\quad \text{for } \pos_k \in \setTplus(\phi), \\[10pt]
\lim_{\varepsilon\searrow 0}\frac{\redPhiObjectiveFunction(S_{k,\varepsilon} \levelset)-\redPhiObjectiveFunction(\levelset)}{|\Omega(S_{k,\varepsilon} \levelset) \symmDiff \Omega( \levelset)|} \quad \text{for } \pos_k \in \setS(\phi).
\end{cases}
\end{equation}
Here, given $\phi \in S_h^1(D)$, the respective sets are defined by
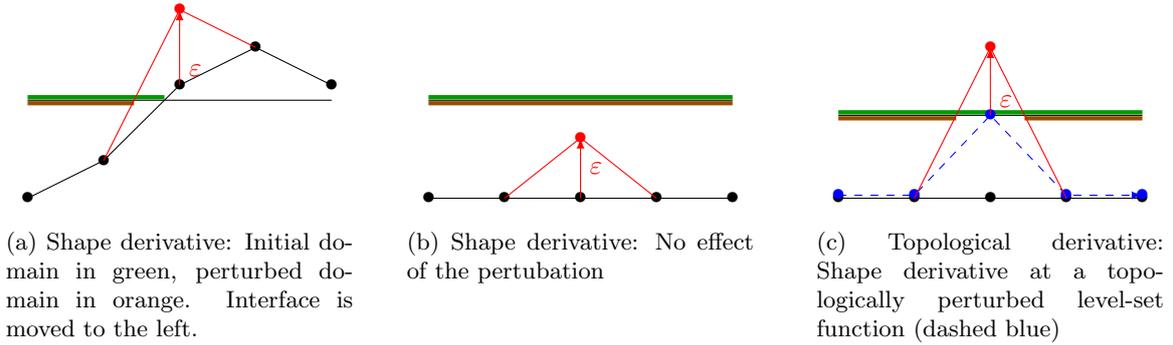
\begin{figure}
\begin{subfigure}[t]{0.3\textwidth}
	\centering
	\begin{tikzpicture}[>=latex]
	\def\z{-1.3}
	\def\dz{0.04}
	\draw (0,0) -- (4,0);
	\draw[ultra thick,green!60!black] (0,\dz) -- (1.8,\dz);
	\draw[ultra thick,orange!60!black] (0,-\dz) -- (1.4,-\dz);
	\draw (0,\z) node{$\bullet$}--(1,\z+0.5)node{$\bullet$} -- (2,\z+1.5)node{$\bullet$}-- (3,\z+2)node{$\bullet$}-- (4,\z+1.5)node{$\bullet$};
	\draw[red] (1,\z+0.5) -- (2,\z+2.5)node{$\bullet$}-- (3,\z+2);
	\draw[red,->] (2,\z+1.5)--(2,\z+2.5) node[pos=0.2,right]{$\varepsilon$};
	\end{tikzpicture}
		\caption{Shape derivative: Initial domain in green, perturbed domain in orange. Interface is moved to the left.}	
\end{subfigure}\hfill
\begin{subfigure}[t]{0.3\textwidth}
	\centering
\begin{tikzpicture}[>=latex]
\def\z{-1.3}
\def\dz{0.04}
\draw (0,0) -- (4,0);
\draw[ultra thick,green!60!black] (0,\dz) -- (4,\dz);
\draw[ultra thick,orange!60!black] (0,-\dz) -- (4,-\dz);
\draw (0,\z) node{$\bullet$}--(1,\z)node{$\bullet$} -- (2,\z)node{$\bullet$}-- (3,\z)node{$\bullet$}-- (4,\z)node{$\bullet$};
\draw[red] (1,\z) -- (2,\z+0.8)node{$\bullet$}-- (3,\z);
\draw[red,->] (2,\z)--(2,\z+0.8) node[pos=0.5,right]{$\varepsilon$};
\end{tikzpicture}
\subcaption{Shape derivative: No effect of the pertubation}	
\end{subfigure}
\hfill
\begin{subfigure}[t]{0.3\textwidth}
	\centering
\begin{tikzpicture}[>=latex]
\def\z{-1.1}
\def\dz{0.04}
\draw (0,0) -- (4,0);
\draw[ultra thick,green!60!black] (0,\dz) -- (4,\dz);
\draw[ultra thick,orange!60!black] (0,-\dz) -- (1.55,-\dz);
\draw[ultra thick,orange!60!black] (2.45,-\dz) -- (4,-\dz);
\draw (0,\z) node{$\bullet$}--(1,\z)node{$\bullet$} -- (2,\z)node{$\bullet$}-- (3,\z)node{$\bullet$}-- (4,\z)node{$\bullet$};
\draw[blue,->,dashed] (0,\z+0.04) node{$\bullet$}--(1,\z+0.04)node{$\bullet$} -- (2,0)node{$\bullet$}-- (3,\z+0.04)node{$\bullet$} -- (4,\z+0.04)node{$\bullet$};
\draw[red] (1,\z) -- (2,\z+2)node{$\bullet$}-- (3,\z);
\draw[red,->] (2,0)--(2,0.9) node[pos=0.2,right]{$\varepsilon$};
\end{tikzpicture}
\subcaption{Topological derivative: Shape derivative at a topologically perturbed level-set function (dashed blue)}
\end{subfigure}
\caption{Illustration of the shape derivative and the topological derivative. }
\label{fig:1}
\end{figure}

\begin{subequations}
	\begin{align}
	\setTminus(\phi) =& \{\pos_k \in \mesh \, | \, \forall i \in \setRk:\levelset(\pos_i) \le 0 \}, \\
	\setTplus(\phi) =& \{\pos_k \in \mesh \, |\, \forall i \in \setRk:\levelset(\pos_i) \ge 0 \}, \\
	\setS(\phi) =& \mesh \setminus (\setTminus(\phi) \cup\setTplus(\phi)).
	\end{align}
\end{subequations}
Whenever, the level set function $\phi$ is clear from the context, we will drop the argument and write $\setTminus, \setTplus, \setS$ for brevity.
Furthermore, $\Tplus: S_h^1(D) \rightarrow S_h^1(D)$ is the positive discrete topological perturbation operator defined by its action
\begin{align}
\Tplus\levelset(\pos) = \levelset(\pos) -( \levelset(\pos_k) - \eps) \varphi_k,
\end{align}
see Figure \ref{fig:1}(c), whereas the negative discrete topological perturbation operator $\Tminus: S_h^1(D) \rightarrow S_h^1(D)$ is defined by  
\begin{align}
\Tminus\levelset(\pos) = \levelset(\pos) -( \levelset(\pos_k) + \eps) \varphi_k.
\end{align}

Finally, the discrete shape perturbation operator $S_{k,\varepsilon}: S_h^1(D) \rightarrow S_h^1(D)$ is defined by 
\begin{align} \label{eq_def_phiEps}
\S\levelset(\pos) := 
\levelset(\pos) + \eps \varphi_k(\pos).            
\end{align}
\begin{remark}
	Note that the discrete perturbation operators defined above only change the nodal value of the finite element function $\levelset \in S_h^1(D)$ at one node $\pos_k$, \eg for $\S$ it holds $\S\levelset(\pos_j) = \levelset(\pos_j)$ for all $j \in \mesh \setminus \{k\}$ and $\S\levelset(\pos_k) = \levelset(\pos_k) + \eps$.
\end{remark}

%
\subsection{Computation of the numerical \topshapeDerivative derivative for the area cost functional}
Before we compute the numerical \topshapeDerivative derivative \eqref{eq::topDerivative} for the full model problem \eqref{eq::ContinuousProblem}, we consider the case of the pure volume cost function and neglect the PDE constraint, i.e., we set $c_1 = 1, c_2 =0$ in \eqref{eq::ContinuousObjective}.

For that purpose, we investigate
\begin{equation}\label{eq::areaChange}
\changeArea  = |\Omega(\discreteOperator\levelset)|-|\Omega(\levelset)|,
\end{equation}
for $\discreteOperator \in \{ \Tplus,\,\Tminus,\,\S \}$. 
We note that for the computation of \eqref{eq::areaChange} only those "cut elements" are relevant which have a node $\pos_k$, \ie
\begin{equation}\label{eq::areaChangeB}
\changeArea = {\sum_{l\in C_k} 
\delta_{k,\eps}a_l} 	\quad \mbox{with }  \delta_{k,\eps}a_l :=
\int_{\tau_l} H(-\discreteOperator\levelset) - H(-\levelset)  \ddx,
\end{equation}
where $H(x)$ denotes the Heaviside step function and
\begin{equation}\label{eq::areaChangeSet}
C_k =  \{l \in I_{\pos_k} :  \tau_l \cap \partial \Omega(\discreteOperator\levelset) \neq \emptyset \}
\end{equation}
 is the set of all indices of elements adjacent to $\pos_k$ which are intersected by the perturbed interface. \black
 Note that, for $\eps>0$ small enough, $C_k$ does not depend on the concrete value of $\eps$. \black
For an element $\tau_l$ with $l \in I_{\pos_k}$ we denote the three vertices in counter-clockwise orientation by $\pos_{l_1}$, $\pos_{l_2}$, $\pos_{l_3}$ and assume that $\pos_k = \pos_{l_1}$. Moreover we denote $ \levelset_{l_j} := \levelset(\pos_{l_j})$ and $\tilde \levelset_{l_j} := \discreteOperator\levelset(\pos_{l_j})$ for $j=1,2,3$ and small enough $\eps$.
In the following, we will be interested in 
\begin{align}
    d_k a := \sum_{l \in C_k} d_k a_l \quad \mbox{ with } d_k a_l := \underset{\eps \searrow 0}{\mbox{lim }} \frac{\delta_{k,\eps} a_l}{\eps^o}.
\end{align}
with $\delta_{k,\eps} a_l$ defined in \eqref{eq::areaChange}.
We consider six different sets (see Figure \ref{fig::sets} for an illustration)
\begin{align}\label{eq_def_Ixk_ABCpm}
    \begin{aligned}
I_{\pos_k}^{A_+} = & \{ l \in I_{\pos_k}: \tilde\phi_{l_1} > 0, \tilde\phi_{l_2} < 0, \tilde\phi_{l_3} < 0 \}, &&&
I_{\pos_k}^{A_-} = & \{ l \in I_{\pos_k}: \tilde\phi_{l_1} < 0, \tilde\phi_{l_2} > 0, \tilde\phi_{l_3} > 0 \}, \\
I_{\pos_k}^{B_+} =& \{ l \in I_{\pos_k}: \tilde\phi_{l_1} < 0, \tilde\phi_{l_2} > 0, \tilde\phi_{l_3} < 0 \}, &&&
I_{\pos_k}^{B_-} =& \{ l \in I_{\pos_k}: \tilde\phi_{l_1} > 0, \tilde\phi_{l_2} < 0, \tilde\phi_{l_3} > 0 \}, \\
I_{\pos_k}^{C_+} =& \{ l \in I_{\pos_k}: \tilde\phi_{l_1} < 0, \tilde\phi_{l_2} < 0, \tilde\phi_{l_3} > 0 \}, &&&
I_{\pos_k}^{C_-} =& \{ l \in I_{\pos_k}: \tilde\phi_{l_1} > 0, \tilde\phi_{l_2} > 0, \tilde\phi_{l_3} < 0 \},
    \end{aligned}
\end{align}
such that 
\begin{align*}
C_k = I_{\pos_k}^{A_+} \cup I_{\pos_k}^{A_-} \cup I_{\pos_k}^{B_+} \cup I_{\pos_k}^{B_-} \cup I_{\pos_k}^{C_+} \cup I_{\pos_k}^{C_-}
\end{align*}
with a direct sum on the right hand side. We can thus split the sum in \eqref{eq::areaChangeB} into six parts,
\begin{align*}
\changeArea = 
\mathcal{I}_{I_{\pos_k}^{A_+}} + \mathcal{I}_{I_{\pos_k}^{A_-}}+
\mathcal{I}_{I_{\pos_k}^{B_+}} + \mathcal{I}_{I_{\pos_k}^{B_-}}+
\mathcal{I}_{I_{\pos_k}^{C_+}} + \mathcal{I}_{I_{\pos_k}^{C_-}}.
\end{align*}
\begin{figure}
	\begin{subfigure}[t]{0.33\textwidth}
	\includegraphics{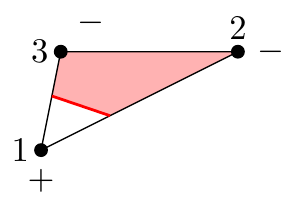}
	\subcaption{Triangle is in $I_{\pos_k}^{A_+}$}	
	\end{subfigure}\hfill
	\begin{subfigure}[t]{0.33\textwidth}
	\includegraphics{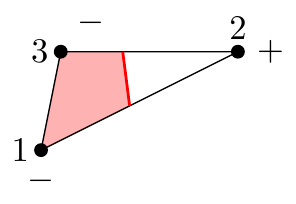}
	\subcaption{Triangle is in $I_{\pos_k}^{B_+}$}
	\end{subfigure}\hfill
	\begin{subfigure}[t]{0.33\textwidth}
	\includegraphics{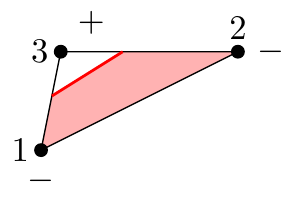}
	\subcaption{Triangle is in $I_{\pos_k}^{C_+}$}
	\end{subfigure}

	\begin{subfigure}[t]{0.33\textwidth}
	\includegraphics{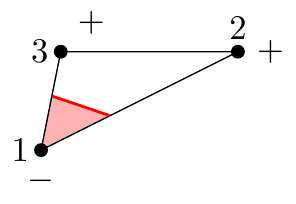}
	\subcaption{Triangle is in $I_{\pos_k}^{A_-}$}	
	\end{subfigure}\hfill
	\begin{subfigure}[t]{0.33\textwidth}
		\includegraphics{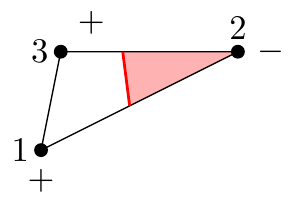}
		\subcaption{Triangle is in $I_{\pos_k}^{B_-}$}		
	\end{subfigure}\hfill
	\begin{subfigure}[t]{0.33\textwidth}
	\includegraphics{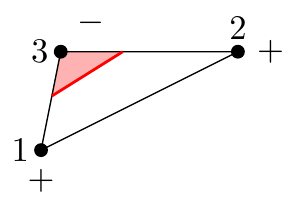}
	\subcaption{Triangle is in $I_{\pos_k}^{C_-}$}		
	\end{subfigure}
	\caption{Illustration of the sets $I_{\pos_k}^{A_+},~I_{\pos_k}^{A_-},~I_{\pos_k}^{B_+},~I_{\pos_k}^{B_-},~I_{\pos_k}^{C_+} ,~ I_{\pos_k}^{C_-}$. The nodal values of the level-set functions are indicated by $-$, $+$. The interface is drawn in red.}
	\label{fig::sets}
\end{figure}
\begin{figure}\centering
\begin{subfigure}[t]{0.49\textwidth}\centering
	\includegraphics{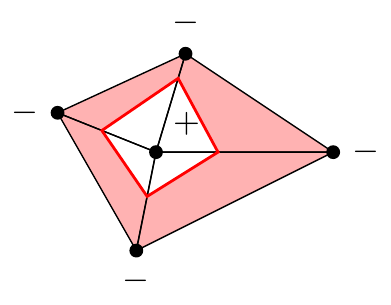}
	\subcaption{Triangles are in $I_{\pos_k}^{A_+}$}
	\label{fig::IAtopA}	
\end{subfigure}
\begin{subfigure}[t]{0.49\textwidth}\centering
	\includegraphics{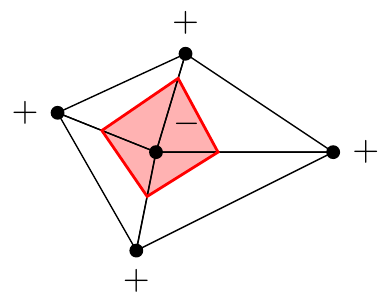}
	\subcaption{Triangles are in $I_{\pos_k}^{A_-}$}
	\label{fig::IAtopB}		
\end{subfigure}
\caption{Illustration of $I_{\pos_k}^{A_+}$ and $I_{\pos_k}^{A_-}$ in the case of topological perturbations.}
\label{fig::topDir}
\end{figure}
\paragraph{Configuration A}
For $l \in I_{\pos_k}^{A_+}$ we have
\begin{equation*}
\begin{aligned}
\int_{\tau_{l}} H(-\discreteOperator\levelset(\pos)) \ddx &= 
 \frac{\detJl}{2} \left(1- \int_{\xi_1=0}^{\ell_1} \int_{\xi_2=0}^{\ell_2\left(1-\frac{\xi_1}{\ell_1}\right)}  \,d\xi_2d\xi_1 \right) \\
&=   \frac{\detJl}{2}(1-\ell_1\ell_2),
\end{aligned}
\end{equation*}
with
\begin{align*}
	\ell_1 = \frac{\levelset_{l_1}+\eps}{\levelset_{l_1}+\eps-\levelset_{l_2}}, \qquad
	\ell_2 = \frac{\levelset_{l_1}+\eps}{\levelset_{l_1}+\eps-\levelset_{l_3}}.
\end{align*}
Therefore,
\begin{equation}
\begin{aligned}\label{eq::intAplus}
\mathcal{I}_{I_{\pos_k}^{A_+}}
&= \sum_{l \in I_{\pos_k}^{A_+}} \frac{\detJl}2 \left(-\frac{(\levelset_{l_1}+\eps)^2}{(\levelset_{l_1}+\eps-\levelset_{l_2})(\levelset_{l_1}+\eps-\levelset_{l_3})}+\frac{\levelset_{l_1}^2}{(\levelset_{l_1}-\levelset_{l_2})(\levelset_{l_1}-\levelset_{l_3})} \right)\\
&= \sum_{l \in I_{\pos_k}^{A_+}} \frac{\detJl}2 \left(\frac{\eps \levelset_{l_1} \left[\levelset_{l_1}(\levelset_{l_2} + \levelset_{l_3})-2\levelset_{l_2}\levelset_{l_3}  \right]+\eps^2\left[\levelset_{l_1}(\levelset_{l_2} + \levelset_{l_3})-\levelset_{l_2}\levelset_{l_3} \right]}{(\levelset_{l_1} - \levelset_{l_2})^2(\levelset_{l_1} - \levelset_{l_3})^2 + O(\varepsilon)} \right).
\end{aligned}
\end{equation}
For $\pos_k \in \setTminus$ and $\discreteOperator = \Tplus$, it holds $I_{\pos_k}=I_{\pos_k}^{A_+}$ (see Figure \ref{fig::IAtopA}) and we have to consider \eqref{eq::intAplus} with $\levelset_{l_1} = 0$. Thus, we obtain
\begin{align}
\mathcal{I}_{I_{\pos_k}^{A_+}}= -
\sum_{l \in I_{\pos_k}^{A_+}}\frac{\eps^2\detJl}2 \frac{\levelset_{l_2}\levelset_{l_3} }{ \levelset_{l_2}^2\levelset_{l_3}^2 + O(\varepsilon)},
\end{align}
and conclude for this case   
\begin{equation}\label{eq::areaChangeTopAplus}
 d_{k}a = \lim_{\varepsilon\searrow 0}\frac{\changeArea}{\eps^2} =- \sum_{l \in I_{\pos_k}^{A_+}}\frac{\detJl}{2\levelset_{l_2}\levelset_{l_3}}.
\end{equation}
Moreover, for $\pos_k \in \setTplus$ and $\discreteOperator = \Tminus$, it holds $I_{\pos_k} = I_{\pos_k}^{A_-}$ (see Figure \ref{fig::IAtopB}) and we have for $l \in I_{\pos_k}^{A_-}$,
\begin{align*}
\int_{\tau_{l}} H(-\discreteOperator\levelset(\pos)) \ddx &= 
\frac{\detJl}{2} \int_{\xi_1=0}^{\ell_1} \int_{\xi_2=0}^{\ell_2\left(1-\frac{\xi_1}{\ell_1}\right)}  \,d\xi_2d\xi_1  =   \frac{\detJl}{2}\ell_1\ell_2,
\end{align*}
which leads to 
\begin{equation}
\begin{aligned}\label{eq::intAminus}
\mathcal{I}_{I_{\pos_k}^{A_-}}
&= -\sum_{l \in I_{\pos_k}^{A_-}} \frac{\detJl}2 \left(\frac{\eps \levelset_{l_1} \left[\levelset_{l_1}(\levelset_{l_2} + \levelset_{l_3})-2\levelset_{l_2}\levelset_{l_3}  \right]+\eps^2\left[\levelset_{l_1}(\levelset_{l_2} + \levelset_{l_3})-\levelset_{l_2}\levelset_{l_3} \right]}{(\levelset_{l_1} - \levelset_{l_2})^2(\levelset_{l_1} - \levelset_{l_3})^2 + O(\varepsilon)} \right).
\end{aligned}
\end{equation}
Therefore, we obtain for this case
\begin{equation*}
d_{k}a = \lim_{\varepsilon\searrow 0}\frac{\changeArea}{\eps^2} = \sum_{l \in I_{\pos_k}^{A_-}}\frac{\detJl}{2\levelset_{l_2}\levelset_{l_3}}.
\end{equation*}
Finally, for $\pos_k \in \setS$ and $\discreteOperator = \S$ we deduce from \eqref{eq::intAplus} 
\begin{equation}\label{eq::shapeA}
\lim_{\varepsilon\searrow 0}\frac{\mathcal{I}_{I_{\pos_k}^{A_+}}}{\eps} = \sum_{l \in I_{\pos_k}^{A_+}}\frac{\detJl\;\levelset_{l_1} \left[\levelset_{l_1}(\levelset_{l_2} + \levelset_{l_3}) - 2\levelset_{l_2}\levelset_{l_3} \right]}{2(\levelset_{l_1} - \levelset_{l_2})^2(\levelset_{l_1} - \levelset_{l_3})^2},
\end{equation}
and from \eqref{eq::intAminus}
\begin{equation} \label{eq_IAplus_o_eps}
\lim_{\varepsilon\searrow 0}\frac{\mathcal{I}_{I_{\pos_k}^{A_-}}}{\eps} = -\sum_{l \in I_{\pos_k}^{A_-}}\frac{\detJl\;\levelset_{l_1} \left[\levelset_{l_1}(\levelset_{l_2} + \levelset_{l_3})-2\levelset_{l_2}\levelset_{l_3}  \right]}{2(\levelset_{l_1} - \levelset_{l_2})^2(\levelset_{l_1} - \levelset_{l_3})^2}.
\end{equation}
\paragraph{Configuration B}
We note that Configuration B can only occur for the case $\pos_k \in \setS$ and $\discreteOperator = \S$. For $l \in I_{\pos_k}^{B_+}$ it holds
\begin{align}\label{eq::configurationB}
\int_{\tau_{l}} H(-\S\levelset(\pos)) \ddx = \frac{\detJl}2\left(1- \frac{\levelset_{l_2}^2}{(\levelset_{l_2}-\levelset_{l_3})(\levelset_{l_2}-\levelset_{l_1}-\eps)} \right),
\end{align}
and
\begin{equation*}
\begin{aligned}
\mathcal{I}_{I_{\pos_k}^{B_+}}
&= \sum_{l \in I_{\pos_k}^{B_+}}\frac{\detJl\levelset_{l_2}^2}2 \frac{(\levelset_{l_2}-\levelset_{l_3})(\levelset_{l_2}-\levelset_{l_1}-\varepsilon)-(\levelset_{l_2}-\levelset_{l_3})(\levelset_{l_2}-\levelset_{l_1})}{(\levelset_{l_2}-\levelset_{l_3})(\levelset_{l_2}-\levelset_{l_1})(\levelset_{l_2}-\levelset_{l_3})(\levelset_{l_2}-\levelset_{l_1}-\varepsilon)}\\
&= \sum_{l \in I_{\pos_k}^{B_+}}\frac{\detJl\levelset_{l_2}^2}2 \frac{-\varepsilon(\levelset_{l_2}-\levelset_{l_3})}{(\levelset_{l_2}-\levelset_{l_3})^2(\levelset_{l_2}-\levelset_{l_1})(\levelset_{l_2}-\levelset_{l_1}-\varepsilon)}.
\end{aligned}
\end{equation*}
Thus,
\begin{equation}\label{eq::shapeB}
\lim_{\eps\searrow 0}\frac{\mathcal{I}_{I_{\pos_k}^{B_+}}}{\eps} = -\sum_{l \in I_{\pos_k}^{B_+}} \frac{\detJl}2 \frac{\levelset_{l_2}^2}{(\levelset_{l_2}-\levelset_{l_3})(\levelset_{l_2}-\levelset_{l_1})^2}.
\end{equation}
For the case $l \in I_{\pos_k}^{B_-}$ we have
\begin{align}\label{eq::configurationBminus}
	\int_{\tau_{l}} H(-\S\levelset(\pos)) \ddx = \frac{\detJl}2 \frac{\levelset_{l_2}^2}{(\levelset_{l_2}-\levelset_{l_3})(\levelset_{l_2}-\levelset_{l_1}-\eps)},
\end{align}
and obtain
\begin{equation*}
\lim_{\varepsilon\searrow 0}\frac{\mathcal{I}_{I_{\pos_k}^{B_-}}}{\eps} = \sum_{l \in I_{\pos_k}^{B_-}} \frac{\detJl}2 \frac{\levelset_{l_2}^2}{(\levelset_{l_2}-\levelset_{l_3})(\levelset_{l_2}-\levelset_{l_1})^2}.
\end{equation*}
\paragraph{Configuration C}
Analogously as for Configuration B, we note that Configuration C can only occur for the case $\pos_k \in \setS$ and $\discreteOperator = \S$. We have  
\begin{equation}\label{eq::shapeC}
\lim_{\eps\searrow 0}\frac{\mathcal{I}_{I_{\pos_k}^{C_+}}}{\eps} = -\sum_{l \in I_{\pos_k}^{C_+}}\frac{\detJl}2 \frac{\levelset_{l_3}^2}{(\levelset_{l_3}-\levelset_{l_2})(\levelset_{l_3}-\levelset_{l_1})^2},
\end{equation}
and
\begin{equation}
\lim_{\varepsilon\searrow 0}\frac{\mathcal{I}_{I_{\pos_k}^{C_-}}}{\eps} = \sum_{l \in I_{\pos_k}^{C_-}}\frac{\detJl}2 \frac{\levelset_{l_3}^2}{(\levelset_{l_3}-\levelset_{l_2})(\levelset_{l_3}-\levelset_{l_1})^2}.
\end{equation}

Summarizing, we have shown the following result.
\begin{theorem}\label{thm::areaDerivative}
For $\pos_k \in \setTminus$ we have
\begin{align}\label{eq_derivAreaTminus}
    d_k a = \sum_{l \in I_{\pos_k}^{A_+}} d_ka_l =  \sum_{l \in I_{\pos_k}^{A_+}}(-1)\frac{\detJl}{2\levelset_{l_2}\levelset_{l_3}}.
\end{align}
For $\pos_k \in \setTplus$ we have
\begin{align}\label{eq_derivAreaTplus}
    d_k a = \sum_{l \in I_{\pos_k}^{A_-}} d_ka_l =   \sum_{l \in I_{\pos_k}^{A_-}}\frac{\detJl}{2\levelset_{l_2}\levelset_{l_3}}.
\end{align}
For $\pos_k \in \setS$ we have
\begin{equation}\label{eq::derivativeArea}
\begin{aligned}
    d_k a =& \sum_{l \in C_{k}} d_k a_l\\=&
    \sum_{l \in I_{\pos_k}^{A_+}}\frac{\detJl\;\levelset_{l_1} \left[\levelset_{l_1}(\levelset_{l_2} + \levelset_{l_3})-2\levelset_{l_2}\levelset_{l_3}  \right]}{2(\levelset_{l_1} - \levelset_{l_2})^2(\levelset_{l_1} - \levelset_{l_3})^2} + \sum_{l \in I_{\pos_k}^{A_-}}  (-1)  \frac{\detJl\;\levelset_{l_1} \left[\levelset_{l_1}(\levelset_{l_2} + \levelset_{l_3}) - 2\levelset_{l_2}\levelset_{l_3} \right]}{2(\levelset_{l_1} - \levelset_{l_2})^2(\levelset_{l_1} - \levelset_{l_3})^2} \\
    &+\sum_{l \in I_{\pos_k}^{B_+}}  (-1)  \frac{\detJl}2 \frac{\levelset_{l_2}^2}{(\levelset_{l_2}-\levelset_{l_3})(\levelset_{l_2}-\levelset_{l_1})^2}  +  \sum_{l \in I_{\pos_k}^{B_-}} \frac{\detJl}2 \frac{\levelset_{l_2}^2}{(\levelset_{l_2}-\levelset_{l_3})(\levelset_{l_2}-\levelset_{l_1})^2} \\
    &+\sum_{l \in I_{\pos_k}^{C_+}}  (-1)  \frac{\detJl}2 \frac{\levelset_{l_3}^2}{(\levelset_{l_3}-\levelset_{l_2})(\levelset_{l_3}-\levelset_{l_1})^2}  +  \sum_{l \in I_{\pos_k}^{C_-}}\frac{\detJl}2 \frac{\levelset_{l_3}^2}{(\levelset_{l_3}-\levelset_{l_2})(\levelset_{l_3}-\levelset_{l_1})^2}. 
\end{aligned}
\end{equation}
\end{theorem}

\begin{remark}\label{remark::area}
	The corresponding computations for the denominators in \eqref{eq::topDerivative}, \ie $|\Omega(\discreteOperator\levelset)\symmDiff \Omega(\levelset)|$, are closely related to the computations presented in this section for \eqref{eq::areaChange}.
	Denoting
	\begin{align*}
        \delta_{k,\eps} \tilde a_l := \int_{\tau_l} | H(-\discreteOperator\levelset) - H(-\levelset)|  \ddx, \quad d_k \tilde a_l := \underset{\eps \searrow 0}{\mbox{lim}} \frac{\delta_{k,\eps} \tilde a_l}{\eps^o},
	\end{align*}
	we get that 
	\begin{align} \label{eq_dkatilde}
        \underset{\eps \searrow 0}{\mbox{lim }} \frac{|\Omega(\discreteOperator\levelset) \symmDiff  \Omega(\levelset)|}{\eps^o} = \sum_{l\in C_k} d_k \tilde a_l =: d_k \tilde a
	\end{align}
	where $d_k \tilde a_l = |d_k a_l|$ with the formulas for $d_k a_l$ given in \eqref{eq_derivAreaTminus}--\eqref{eq::derivativeArea}. It is obvious that, for $\pos_k \in \mathcal T^-$, $d_k a_l <0$ and for $\pos_k \in \mathcal T^+$, $d_k a_l >0$. Moreover, note that, for $\pos_k \in \mathcal S$, it holds $d_k a_l <0$ for all $l \in C_k$. This yields that
	\begin{align}
        \frac{d_k a_l}{d_k \tilde a_l} = \begin{cases} - 1 & \mbox{if }k \in \mathcal S \cup \mathcal T^-, \\
                          \; \;\, 1 & \mbox{if } k \in \mathcal T^+. 
                         \end{cases}
	\end{align}
\end{remark}

\begin{corollary} \label{cor_dVol}
    From Theorem \ref{thm::areaDerivative} and Remark \ref{remark::area}, it follows that the numerical \topshapeDerivative derivative of the volume cost function $\mbox{Vol}(\phi) := |\Omega(\phi)|$ is given by
    \begin{align*}
        d \mbox{Vol}(\phi)(\pos_k) = \begin{cases} - 1 & \mbox{if }k \in \mathcal S \cup \mathcal T^-, \\
                          \; \;\, 1 & \mbox{if } k \in \mathcal T^+. 
                         \end{cases}
    \end{align*}

\end{corollary}
\subsection{Computation of the numerical \topshapeDerivative derivative via Lagrangian framework}\label{sec::Lagrangian}
\newcommand{\uuzero}{{\uu}}
Next, we consider the computation of the numerical \topshapeDerivative derivative of an optimization problem that is constrained by a discretized PDE. For that purpose, we set $c_1 =0$ in \eqref{eq::ContinuousObjective} and $\phiObjectiveFunction(\levelset,\uu) := \objectiveFunction(\Omega(\levelset),\uu)$. The discretized problem reads
\begin{align} \label{eq_discOptiProb_J}
\underset{\levelset}{\mbox{min }} \phiObjectiveFunction(\levelset, \uu)&  \\
\mbox{s.t. } \Amat_{\levelset}  \uu &= \ff_{\levelset} \label{eq_discOptiProb_constr}
\end{align}
and we are interested in the sensitivity of $\phiObjectiveFunction$ when the level set function $\levelset$ representing the geometry is replaced by a perturbed level set function $\levelset_{\eps}=\discreteOperator\levelset$. 
 The perturbed Lagrangian for \eqref{eq_discOptiProb_J}--\eqref{eq_discOptiProb_constr} with respect to a perturbation of $\levelset$ reads 
\begin{align} \label{eq_defLagDiscr}
L(\eps, \uu, \vv) =& \phiObjectiveFunction(\phi_\eps, \uu) + \Amat_\eps\uu \cdot \vv - \ff_\eps \cdot \vv
\end{align}
where we use the abbreviated notation $\Amat_\eps := \Amat_{\levelset_\eps}$, and $\ff_\eps := \ff_{\levelset_\eps}$. Moreover, for $\eps \geq 0$, we define the perturbed state $\uu_\eps$ as the solution to
\begin{align*}
0 = \partial_\vv L(\eps, \uu_\eps, \vv),
\end{align*}
i.e. $\uu_\eps$ is the solution to
\begin{align}  \label{eq_defueps2}
\Amat_\eps \uu_\eps = \ff_\eps
\end{align}
and the (unperturbed) adjoint state $\pp$ as the solution to
\begin{align} \label{eq_defAdjointDiscr}
0 = \partial_\uu L(0, \uu, \pp),
\end{align}
for the state $\uu$ given, i.e. $\pp$ solves
\begin{align*}
\Amat^\top \pp = - \partial_\uu J(\levelset, \uu).
\end{align*}
Note that we use the notation $\uu$ for $\uu_{\eps = 0}$.
The numerical \topshapeDerivative derivative at the node $\pos_k$ can be computed as the limit
\begin{align*}
d \redPhiObjectiveFunction(\levelset)(\pos_k) = \underset{\eps \searrow 0}{\mbox{lim }} \frac{1}{|\Omega(\levelset_\eps) \symmDiff \Omega(\levelset)|} (\phiObjectiveFunction(\phi_\eps, \uu_\eps) - \phiObjectiveFunction(\phi, \uu) ).
\end{align*}

With the help of the Lagrangian \eqref{eq_defLagDiscr}, we can rewrite the right hand side as
\begin{align*}
\phiObjectiveFunction(\phi_\eps, \uu_\eps) - \phiObjectiveFunction(\phi,\uu) = L(\eps, \uu_\eps, \pp) - L(0, \uuzero, \pp)
\end{align*}
where we used that $\uu_\eps$ solves \eqref{eq_defueps2} for $\eps \geq 0$. Following the approach used in \cite{simplified}, we use the fundamental theorem of calculus as well as \eqref{eq_defAdjointDiscr} to rewrite this as
\begin{align}
\phiObjectiveFunction(\phi_\eps, \uu_\eps) - \phiObjectiveFunction(\phi,\uuzero) =& \int_0^1 [\partial_\uu L(\eps, \uuzero + s(\uu_\eps - \uuzero), \pp) - \partial_\uu L(\eps, \uu, \pp)](\uu_\eps - \uuzero) \mbox ds \\
&+[ \partial_\uu L(\eps, \uu, \pp) - \partial_\uu L(0, \uu, \pp)](\uu_\eps - \uu) \\
&+ L(\eps, \uu, \pp) - L(0, \uu, \pp).
\end{align}

Thus the numerical \topshapeDerivative derivative can be obtained as the sum of three limits,
\begin{align*}
d \redPhiObjectiveFunction(\levelset)(\pos_k) = R_1(\uu, \pp) + R_2(\uu, \pp) + R_0(\uu, \pp)
\end{align*}
where
\begin{align*}
R_1(\uu, \pp) :=& \underset{\eps \searrow 0}{\mbox{lim }} \frac{1}{|\Omega(\levelset_\eps) \symmDiff  \Omega(\levelset)|}  \int_0^1 [\partial_\uu L(\eps, \uuzero + s(\uu_\eps - \uuzero), \pp) - \partial_\uu L(\eps, \uu, \pp)](\uu_\eps - \uuzero) \mbox ds, \\
R_2(\uu, \pp) :=& \underset{\eps \searrow 0}{\mbox{lim }} \frac{1}{|\Omega(\levelset_\eps) \symmDiff  \Omega(\levelset)|}[ \partial_\uu L(\eps, \uu, \pp) - \partial_\uu L(0, \uu, \pp)](\uu_\eps - \uu), \\
R_0(\uu, \pp) := & \underset{\eps \searrow 0}{\mbox{lim }} \frac{1}{|\Omega(\levelset_\eps) \symmDiff  \Omega(\levelset)|} [ L(\eps, \uu, \pp) - L(0, \uu, \pp)].
\end{align*}

For $\phiObjectiveFunction(\levelset_\eps, \uu) = c_2 (\uu - \hat \uu) \tilde \MM_\eps (\uu - \hat \uu)$, where $\tilde \MM_\eps := \tilde \MM_{\levelset_\eps}$ represents the matrix $\tilde \MM$ defined in \eqref{eq_defMtilde} with $\Omega = \Omega(\levelset_\eps)$, we get
\begin{align*}
R_1(\uu, \pp) =& c_2 \underset{\eps \searrow 0}{\mbox{lim }}  \frac{1}{|\Omega(\levelset_\eps) \symmDiff  \Omega(\levelset)|} (\uu_\eps - \uuzero)^\top \tilde \MM_\eps (\uu_\eps - \uuzero).
\end{align*}
Moreover, 
\begin{align} \label{eq_R2discr}
R_2(\uu, \pp) =& \underset{\eps \searrow 0}{\mbox{lim }} \frac{1}{|\Omega(\levelset_\eps) \symmDiff  \Omega(\levelset)|} \left[ 2 c_2 (\tilde \MM_\eps - \tilde \MM) (\uu - \hat \uu)  \cdot  (\uu_\eps - \uuzero) + (\Amat_\eps - \Amat) (\uu_\eps  - \uuzero) \cdot \pp \right],
\end{align}
and
\begin{align} \label{eq_R0discr}
R_0(\uu, \pp) =& \underset{\eps \searrow 0}{\mbox{lim }} \frac{1}{|\Omega(\levelset_\eps) \symmDiff \Omega(\levelset)|} \left[ c_2(\uuzero - \hat \uu)^\top(\tilde \MM_\eps - \tilde \MM)  (\uuzero - \hat \uu) + \pp^\top (\Amat_\eps - \Amat) \uu - (\ff_\eps - \ff)\cdot \pp \right].
\end{align}

\begin{lemma}\label{lem_uepsu0} There exist constants $c>0, \hat \eps >0$ such that for all $\eps \in (0,\hat\eps)$
    \begin{align*}
        \| \uu_\eps - \uuzero \| \leq c \, \eps^o.
    \end{align*}
    Here, $o=1$ in the case of a shape perturbation and $o=2$ in the case of a topological perturbation.
\end{lemma}
\begin{proof}    
    Subtracting \eqref{eq_defueps2} for $\eps = 0$ from the same equation with $\eps >0$, we get
    \begin{align*}
         \Amat_\eps (\uu_\eps - \uu) &= \ff_\eps - \ff - (\Amat_\eps - \Amat) \uu 
    \end{align*}
    and thus, by the ellipticity of the bilinear form corresponding to $\Amat_\eps$ and the triangle inequality, there is a constant $c>0$ such that for all $\eps>0$ small enough
    \begin{align}
        \| \uu_\eps - \uu \| \leq c ( \| \ff_\eps - \ff \| - \| \Amat_\eps - \Amat \|  \| \uu \| ).
    \end{align}
    For the difference between the perturbed and unperturbed right hand sides we have
    \begin{align*}
        |(\ff_\eps - \ff)_i| \leq& |(f_1 - f_2)| \int_{\Omega(\phi_\eps) \symmDiff \Omega(\phi)} |\varphi_i(x)| \; \mbox dx \\
        |(\Amat_\eps - \Amat)_{i,j}| \leq & \int_{\Omega(\phi_\eps) \symmDiff \Omega(\phi)}  \bigg( |(\lambda_1 - \lambda_2)| |\nabla \varphi_j(x)| |\nabla \varphi_i(x)|  +  |(\alpha_1 - \alpha_2)| |\varphi_j(x)| |\varphi_i(x)|  \bigg)\; \mbox dx.
        \end{align*}
    The result follows from the boundedness of $|\varphi_i(x)|$ and $|\nabla \varphi_i(x)|$ together with \eqref{eq_dkatilde} which implies the existence of $c>0$ such that $|\Omega(\phi_\eps) \symmDiff \Omega(\phi)| \leq c \eps^o$ (cf. Remark \ref{remark::area}). 
\end{proof}
From Lemma \ref{lem_uepsu0} it follows that the terms $R_1(\uu, \pp)$ and $R_2(\uu, \pp)$ vanish. We remark that this is in contrast to the continuous case, where asymptotic analysis shows that $R_2$ does not vanish. We will address this issue in more detail in Section \ref{sec::connection}. 
Thus, in the discrete setting we obtain $d \redPhiObjectiveFunction(\levelset)(\pos_k)=R_0(\uu, \pp)$, i.e.,
\begin{align}\label{eq::derivativeTracking}
d \redPhiObjectiveFunction(\levelset)(\pos_k) &= 
\frac{\mathbf p^\top(d_k \mathbf A \,\uuzero - d_k \mathbf f)
	+c_2 (\uuzero - \hat \uu)^\top d_k\tilde \MM  (\uuzero - \hat \uu)
}{ d_k\tilde a} 
\end{align}
where%
\begin{align} \label{eq_derivativeLimits}
\begin{aligned}
d_k \Amat &= \lim_{\varepsilon\searrow 0} \frac{\Amat_\eps- \Amat}{\varepsilon^o}, &&&
d_k \tilde \MM& = \lim_{\varepsilon\searrow 0} \frac{\tilde\MM_\eps- \tilde\MM}{\varepsilon^o}, \\
d_k \ff &= \lim_{\varepsilon\searrow 0} \frac{\ff_\eps-\ff}{\varepsilon^o}, &&&
d_k \tilde a &= \lim_{\varepsilon\searrow 0} \frac{|\Omega(\levelset_\eps)\symmDiff \Omega(\levelset)|}{\varepsilon^o}, 
\end{aligned}
\end{align}
with $o=1$ for $\pos_k\in\setS$ and $o =2$ for $\pos_k\in\setTminus \cup \setTplus$. To obtain \eqref{eq::derivativeTracking}, we divided both numerator and denominator of \eqref{eq_R0discr} by $\eps^o$ and used that the limit of the quotient coincides with the quotient of the limits provided both limits exist and the limit in the denominator does not vanish. Next we state the numerical \topshapeDerivative derivative of problem \eqref{eq::ContinuousProblem} for arbitrary constant weights $c_1, c_2 \geq 0$ in the cost function \eqref{eq::ContinuousObjective}.

\begin{theorem}[Numerical \topshapeDerivative derivative] \label{thm_numTopShapeDer_formula}
	For $l=1,\dots,N$, let $\mathbf u_l = [u_{l_1},u_{l_2},u_{l_3}]^\top$ and $\mathbf p_l = [p_{l_1},p_{l_2},p_{l_3}]^\top$ be the nodal values for element $\tau_{l}$ of the solution and the adjoint, and 
	\begin{equation*}
	\mathbf k_{0,l}[i,j] = \left(J_l^{-1}\nabla_\xi \psi_j\right)^\top \left(J_l^{-1}\nabla_\xi \psi_i\right), \quad i,j =1,\dots,3.
	\end{equation*}
	Moreover, $u_k = u(\pos_k)$, $p_k = p(\pos_k)$, and $\hat u_k = \hat u(\pos_k)$.  
	For $\pos_k \in \setTminus$ the numerical topological derivative reads
	\begin{align}\label{eq::formulaTopologicalDerivativePlus}
	d\redPhiObjectiveFunction(\levelset)(\pos_k) = -c_1 
	-\Delta \lambda	\frac{\sum_{l \in \setIk}  \frac{\mathbf p_{l}^\top \mathbf k_{0,l} \mathbf u_{l}\detJl}{\levelset_{l_2}\levelset_{l_3}}}{\sum_{l \in \setIk}\frac{\detJl}{\levelset_{l_2}\levelset_{l_3}} }	-\Delta \alpha p_ku_k	+\Delta fp_k	- c_2\Delta \tilde\alpha(u_k-\hat u_k)^2,	
	\end{align}
	whereas for $\pos_k \in \setTplus$ we have
	\begin{align}\label{eq::formulaTopologicalDerivativeMinus}
	d\redPhiObjectiveFunction(\levelset)(\pos_k) = c_1 + 
	\Delta \lambda	\frac{\sum_{l \in \setIk}  \frac{\mathbf p_{l}^\top \mathbf k_{0,l} \mathbf u_{l}\detJl}{\levelset_{l_2}\levelset_{l_3}}}{\sum_{l \in \setIk}\frac{\detJl}{\levelset_{l_2}\levelset_{l_3}} }	+\Delta \alpha p_ku_k	-\Delta f p_k + c_2\Delta \tilde\alpha(u_k-\hat u_k)^2,
	\end{align}
	with 
	\begin{align*}
	\Delta \alpha = \alpha_1-\alpha_2, \quad
	\Delta \lambda = \lambda_1-\lambda_2, \quad
	\Delta f = f_1-f_2, \quad
	\Delta \tilde\alpha = \tilde\alpha_1-\tilde\alpha_2 .
	\end{align*}
	For $\pos_k \in \setS$ the numerical shape derivative reads
	\begin{equation}\label{eq::formulaShapeDerivative}
	\begin{aligned}
	d\redPhiObjectiveFunction(\levelset)(\pos_k) = -c_1 + & 
	\Delta \lambda
	\frac{\sum_{l \in C_k}  \mathbf p_{l}^\top \mathbf k_{0,l} \mathbf u_{l} \,d_k a_{l}}{d_k\tilde a}
	+\Delta \alpha\frac{\sum_{l \in C_k}  \mathbf p_{l}^\top d_k\mathbf m^I_{l} \,\mathbf u_{l}}{d_k\tilde a} 
	  \\
	 &  -\Delta f \frac{\sum_{l \in C_k} \mathbf p_{l}^\top d_k\mathbf f_l^I }{d_k\tilde a}
	+c_2\Delta \tilde\alpha \frac{\sum_{l \in C_k}  (\mathbf u_{l}-\hat{\mathbf u}_{l})^\top d_k\mathbf m^I_{l} \,(\mathbf u_{l}-\hat{\mathbf u}_{l})}{d_k\tilde a},
	\end{aligned}
	\end{equation}
	where the entries of the element matrix $d_k\mathbf m_{l}^I$ and of the element vector $d_k\mathbf f_{l}^I$ are dependent on the local cut situation (cases  $I=A^\pm$, $B^\pm$, $C^\pm$) and are given in Appendix \ref{app::shapeDerivative}. The values $d_k \tilde a$ can be computed by \eqref{eq::derivativeArea} considering Remark \ref{remark::area}.
	\black
\end{theorem}
\begin{proof}
	We evaluate \eqref{eq::derivativeTracking} for $\pos_k \in \setTminus$ and $\discreteOperator = \Tplus$. Thus, $o=2$ in \eqref{eq_derivativeLimits}. We note that 
	\begin{equation}
	\mathbf p^\top d_k \mathbf A \,\uuzero = \sum_{l=1}^N  \mathbf p_{l}^\top (d_k\mathbf m_{l} +d_k\mathbf k_{l}) \mathbf u_{l}.
	\end{equation}
	We have for $l\in \setIk$
	\begin{align*}
	\mathbf k_{l}(\levelset_{\eps})-\mathbf k_l(\levelset) &= \mathbf k_{0,l} \detJl\bigg(  \int_{\xi_1=0}^1\int_{\xi_2=0}^{1-\xi_1} \lambda_1 H(-\levelset_{\eps}\circ\Phi_l) +\lambda_2 H(\levelset_{\eps}\circ\Phi_l) \dd\xi_2 \dd\xi_1 \\ &\quad-\int_{\xi_1=0}^1\int_{\xi_2=0}^{1-\xi_1} \lambda_1 H(-\levelset\circ\Phi_l) +\lambda_2 H(\levelset\circ\Phi_l) \dd\xi_2 \dd\xi_1\bigg) \\
	&=\mathbf k_{0,l} \detJl\bigg(  \lambda_1\int_{\xi_1=0}^1\int_{\xi_2=0}^{1-\xi_1}  (H(-\levelset_{\eps}\circ\Phi_l) -H(-\levelset\circ\Phi_l))  \dd\xi_2 \dd\xi_1 \\ &\quad+\lambda_2\int_{\xi_1=0}^1\int_{\xi_2=0}^{1-\xi_1}  (H(\levelset_\eps\circ\Phi_l) - H(\levelset\circ\Phi_l)) \dd\xi_2 \dd\xi_1\bigg)
	\\
	&= \mathbf k_{0,l}  (\lambda_1 - \lambda_2)\changeArea_l
	\end{align*}
	due to \eqref{eq::areaChangeB} because $|\Omega_\eps|-|\Omega| = -(|D \setminus \Omega_\eps|-|D \setminus \Omega|)$
	and with \eqref{eq::areaChangeTopAplus} we obtain
	\begin{align}\label{eq::termA}
	d_k\mathbf k_l = \underset{\eps \searrow 0}{\mbox{lim }} \frac{\mathbf k_l(\phi_\eps) - \mathbf k_l(\phi)}{\eps^2} 
	&= -\Delta\lambda\frac{\detJl}{2\levelset_{l_2}\levelset_{l_3}} \mathbf k_{0,l}.
	\end{align}
	Due to 
	\begin{equation*}
	\int_{\xi_1=0}^{l_1} \int_{\xi_2=0}^{l_2\left(1-\frac{\xi_1}{l_1}\right)} 
	\xi_1^a \xi_2^b \dd\xi_2 \dd\xi_1 = \frac{\epsilon^{a+b+2}}{\levelset_{l_2}^{a+1}\levelset_{l_3}^{b+1} + O(\epsilon^{a+b+2})}
	\end{equation*}
	for some $a,b \in \mathbb{N}$, we have
	\begin{equation}\label{eq::limitChangeMass}
	\begin{aligned}
	d_k\mathbf m_l = \lim_{\varepsilon\searrow 0} \frac{\mathbf m_{l}(\levelset_{\eps})-\mathbf m_l(\levelset)}{\varepsilon^2} &= - \lim_{\varepsilon\searrow 0}	 \frac{\Delta\alpha}{\varepsilon^2}
	\int_{\xi_1=0}^{l_1} \int_{\xi_2=0}^{l_2\left(1-\frac{\xi_1}{l_1}\right)} 
	\psi_i(\xi)\psi_j(\xi) \detJl\dd\xi_2 \dd\xi_1 \\ &= 
	\begin{bmatrix}
	-\frac{\Delta\alpha\detJl}{2\levelset_{l_2}\levelset_{l_3}} & 0 & 0 \\
	0 & 0 & 0 \\
	0 & 0 & 0 
	\end{bmatrix},
	\end{aligned}
	\end{equation}
	and conclude
	\begin{align}\label{eq::termB}
	\sum_{l=1}^N \mathbf p_{l}^\top d_k\mathbf m_{l}  \mathbf u_{l} = -\Delta\alpha p_k u_k\sum_{l \in \setIk}	\frac{\detJl}{2\levelset_{l_2}\levelset_{l_3}}.
	\end{align}
	Furthermore, with 
	\begin{equation}
	\begin{aligned}\label{eq::limitChangeLoad}
	d_k\mathbf f_l =\lim_{\varepsilon\searrow 0} \frac{\mathbf f_{l}(\levelset_{\eps})-\mathbf f_l(\levelset)}{\varepsilon^2} &= - \lim_{\varepsilon\searrow 0}	 \frac{\Delta f}{\varepsilon^2}
	\int_{\xi_1=0}^{l_1} \int_{\xi_2=0}^{l_2\left(1-\frac{\xi_1}{l_1}\right)} 
	\psi_i(\xi) \detJl\dd\xi_2 \dd\xi_1 \\ &= 
	\begin{bmatrix}
	-\frac{\Delta f\detJl}{2\levelset_{l_2}\levelset_{l_3}}  \\
	0  \\
	0  
	\end{bmatrix},
	\end{aligned}
	\end{equation}
	it follows that
	\begin{equation}\label{eq::termC}
	 \mathbf p^\top d_k\mathbf f = \sum_{l=1}^N  \mathbf p_{l}^\top d_k\mathbf f_{l} = -\Delta f p_k \sum_{l \in \setIk}	\frac{\detJl}{2\levelset_{l_2}\levelset_{l_3}}.
	\end{equation}
	Analogously to \eqref{eq::limitChangeMass} we have
	\begin{equation}\label{eq::limitChangeTracking}
	\begin{aligned}
	d_k\tilde{\mathbf  m}_l  = \lim_{\varepsilon\searrow 0} \frac{\tilde{\mathbf m}_{l}(\levelset_{\eps})-\tilde{\mathbf  m}_l(\levelset)}{\varepsilon^2} = 
	\begin{bmatrix}
	-\frac{\Delta\tilde\alpha\detJl}{2\levelset_{l_2}\levelset_{l_3}} & 0 & 0 \\
	0 & 0 & 0 \\
	0 & 0 & 0 
	\end{bmatrix},
	\end{aligned}
	\end{equation}
	and obtain
	\begin{equation}\label{eq::termD}
	\begin{aligned}
	(\uuzero - \hat \uu)^\top d_k\tilde \MM  (\uuzero - \hat \uu) &= \sum_{l \in \setIk} (\uu_{0,l} - \hat \uu_l)^\top d_k\tilde {\mathbf m}_l  (\uu_{0,l} - \hat \uu_l) \\ &= -\Delta \tilde\alpha(u_k-\hat u_k)^2 \sum_{l \in \setIk}	\frac{\detJl}{2\levelset_{l_2}\levelset_{l_3}}.	
	\end{aligned}
	\end{equation}
	In the present situation, $d_k\tilde a$ is given by the absolute value of \eqref{eq::areaChangeTopAplus} (see also Remark \ref{remark::area}),
	\begin{equation}\label{eq::changeAreaProof}
	d_k\tilde a = \sum_{l \in I_{\pos_k}^{A_+}}\frac{\detJl}{2\levelset_{l_2}\levelset_{l_3}}.
	\end{equation}
	By inserting \eqref{eq::termA}, \eqref{eq::termB}, \eqref{eq::termC}, \eqref{eq::termD}, and \eqref{eq::changeAreaProof} in \eqref{eq::derivativeTracking}, together with Corollary \ref{cor_dVol}, we obtain the sought expression \eqref{eq::formulaTopologicalDerivativePlus}. Formula \eqref{eq::formulaTopologicalDerivativeMinus} can be obtained in an analogous way as \eqref{eq::formulaTopologicalDerivativePlus}.
	
	The formula in \eqref{eq::formulaShapeDerivative} follows directly from \eqref{eq::derivativeTracking} together with Corollary \ref{cor_dVol}. The values of $d_k \mathbf m^I_l$ and $d_k \mathbf f^I_l$ for all possible cut situations $I \in \{A^+, A^-, B^+, B^-, C^+, C^-\}$ are given in Appendix \ref{app::shapeDerivative} and were computed using symbolic computer algebra tools. A mathematical derivation of these terms is omitted here for brevity.
\end{proof}


\section{Connection between continuous and discrete sensitivities} \label{sec::connection}

The \topshapeDerivative derivative introduced in \eqref{eq::topDerivative} and computed for model problem \eqref{eq::ContinuousProblem} in Theorem \ref{thm_numTopShapeDer_formula} represents a sensitivity of the discretized problem \eqref{eq_discr_opti_problem}. In this section, we draw some comparisons with the classical topological and shape derivatives defined on the continuous level before discretization. While the purpose of this paper is to follow the idea \textit{discretize-then-differentiate}, we consider the other way here for comparison reasons.

\subsection{Connections between continuous and discrete topological derivative} \label{sec_connTD}
For comparison, we also illustrate the derivation of the continuous topological derivative according to \eqref{eq_defTD} for problem \eqref{eq::ContinuousProblem}. We use the same Lagrangian framework as introduced in Section \ref{sec::Lagrangian}, see also \cite{simplified}. Given a shape $\Omega \in \mathcal A$, a point $z \in D \setminus \partial \Omega$, an inclusion shape $\omega \subset \mathbb R^d$ with $0 \in \omega$ and $\eps \geq 0$, we define the inclusion $\omega_\eps = z + \eps \omega$ and the perturbed Lagrangian
\begin{align*}
    G(\eps, \varphi, \psi) := c_1 |\Omega_\eps| + c_2 \int_D \tilde{\alpha}_{\Omega_\eps} | \varphi - \targetU|^2\; \mbox dx + \int_D \lambda_{\Omega_\eps} \nabla \varphi \cdot \nabla \psi + \alpha_{\Omega_\eps} \varphi  \psi \;\mbox dx - \int_D f_{\Omega_\eps} \psi \; \mbox dx
\end{align*}
where $\Omega_\eps = \Omega \setminus \omega_\eps$ for $z \in \Omega$ and $\Omega_\eps = \Omega \cup \omega_\eps$ for $z \in D \setminus \overline \Omega$. For simplicity, we only consider the latter case, \ie $z \in D \setminus \overline \Omega$ in the sequel.

Noting that $u_\eps$, $\eps \geq 0$, is the solution to the perturbed state equation with parameter $\eps$, the topological derivative can also be written as
\begin{align}
   d \mathcal J(\Omega)(z) =  \underset{\eps \searrow 0}{\mbox{lim }}\frac{1}{|\omega_\eps|} (G(\eps, u_\eps, p) - G(0, u_0, p))
\end{align}
with the solution to the unperturbed adjoint state equation $p$. As in Section \ref{sec::Lagrangian}, this leads to the topological derivative consisting of the three terms
\begin{align*}
d \redPhiObjectiveFunction(\levelset)(z) = R_1(u, p) + R_2(u, p) + R_0(u, p)
\end{align*}
where
\begin{align*}
R_1(u, p) :=& \underset{\eps \searrow 0}{\mbox{lim }} \frac{1}{|\Omega_\eps \symmDiff \Omega|}  \int_0^1 [\partial_u G(\eps, u_0 + s(u_\eps - u_0), p) - \partial_u G(\eps, u, p)](u_\eps - u_0) \mbox ds, \\
R_2(u, p) :=& \underset{\eps \searrow 0}{\mbox{lim }} \frac{1}{|\Omega_\eps \symmDiff \Omega|}[ \partial_u G(\eps, u, p) - \partial_u G(0, u, p)](u_\eps - u), \\
R_0(u, p) := & \underset{\eps \searrow 0}{\mbox{lim }} \frac{1}{|\Omega_\eps \symmDiff \Omega|} [ G(\eps, u, p) - G(0, u, p)],
\end{align*}
provided that these limits exist.
It is straightforwardly checked that for $z \in D \setminus \overline \Omega$
\begin{align*}
    R_0(u,p) = c_1 + c_2 (\tilde \alpha_1 - \tilde \alpha_2) (u - \targetU)^2(z) + (\lambda_1 - \lambda_2) \nabla u(z) \cdot \nabla p(z) + (\alpha_1 - \alpha_2) u(z) p(z) - (f_1-f_2)(z) p(z).
\end{align*}

For the term $R_2(u,p)$, we obtain
\begin{align*}
    R_2(u,p) 
    =& \underset{\eps \searrow 0}{\mbox{lim }} \frac{1}{|\omega_\eps|} \bigg[ 2 c_2 \int_{\omega_\eps} (\tilde \alpha_{1} - \tilde \alpha_2) (u_0 - \targetU)(u_\eps - u_0) \; dy+  \int_{\omega_\eps} (\lambda_{1} - \lambda_2) \nabla (u_\eps - u_0) \cdot \nabla p \; dy \\
    & \qquad + \int_{\omega_\eps} (\alpha_{1} - \alpha_2) (u_\eps - u_0) p \; dy 
    \bigg].
\end{align*}
A change of variables $y = T_\eps(x) = z + \eps x$ yields
\begin{align} \label{eq_R2omega}
    \begin{aligned}
 R_2(u,p) =&  \underset{\eps \searrow 0}{\mbox{lim }} \frac{1}{|\omega|} \bigg[  2 c_2 (\tilde \alpha_{1} - \tilde \alpha_2)\int_{\omega}  (u_0 - \targetU)\circ T_\eps (u_\eps - u_0) \circ T_\eps  \; dx \\ & +  (\lambda_{1} - \lambda_2) \int_{\omega}  (\nabla (u_\eps - u_0) )\circ T_\eps \cdot (\nabla p)\circ T_\eps \; dx  + (\alpha_{1} - \alpha_2) \int_{\omega}  (u_\eps - u_0)\circ T_\eps \, p\circ T_\eps \; dx \bigg] 
    \end{aligned}
\end{align}

We have a closer look at the diffusion term 
\begin{align} \label{eq_R2lambdaDef}
    R_2^\lambda(u,p):= \underset{\eps \searrow 0}{\mbox{lim }} \frac{1}{|\omega|} (\lambda_1 - \lambda_2)\int_\omega (\nabla(u_\eps - u_0))\circ T_\eps \cdot (\nabla p)\circ T_\eps \; \mbox dx.
\end{align}
In the continuous setting, we now define $K_\eps := \frac{1}{\eps}(u_\eps - u_0)\circ T_\eps$ and use the chain rule $(\nabla \varphi )\circ T_\eps = \frac1\eps \nabla (\varphi \circ T_\eps)$ to obtain
\begin{align}
    R_2^\lambda (u,p)=& \underset{\eps \searrow 0}{\mbox{lim }} \frac{1}{|\omega|} (\lambda_1 - \lambda_2) \int_\omega \nabla K_\eps \cdot (\nabla p) \circ T_\eps \; \mbox dx
\end{align}
Next, one can show the weak convergence $\nabla K_\eps \rightharpoonup \nabla K$ for $K \in \BL$ being defined as the solution to the exterior problem
\begin{align*}
    \int_{\mathbb R^d} \lambda_\omega \nabla K \cdot \nabla \psi \; \mbox dx = -(\lambda_1 - \lambda_2) \int_\omega \nabla u(z) \cdot \nabla \psi \; \mbox dx \quad \mbox{for all }\psi \in \BL,
\end{align*}
where $\BL := \{v \in H^1_{loc}(\mathbb R^2): \nabla v \in L^2(\mathbb R^2) \}/_{\mathbb R}$ is a Beppo-Levi space.
Assuming continuity of $\nabla p$ around the point of perturbation $z$,
it follows that 
\begin{align} \label{eq_R2lambda}
    R_2^\lambda(u,p) = \frac{1}{|\omega|} (\lambda_1 - \lambda_2) \int_\omega \nabla K \cdot \nabla p(z)\; \mbox dx.
\end{align}
It can be shown that the other terms in \eqref{eq_R2omega} vanish and thus $R_2(u,p) = R_2^\lambda(u,p)$.
Finally, it follows from the analysis in \cite[Sec. 5]{simplified} that $R_1(u,p) + R_2(u,p) = \frac{1}{|\omega|}(\lambda_1 - \lambda_2) \int_\omega \nabla K \cdot \nabla p(z) \, \mbox dx = R_2(u,p)$, thus $R_1(u,p) = 0$ and $d \redPhiObjectiveFunction(\Omega)(z)= R_0(u,p) + R_2^\lambda(u,p)$.

\begin{remark}
Comparing the topological derivative formula obtained here with the sensitivity for interior nodes $\pos_k \in \setTminus \cup \setTplus$ obtained in Section \ref{sec::numTopShapeDer}, we see that the term corresponding to $R_2^\lambda(u,p)$, i.e., the term
\begin{align*}
        \underset{\eps \searrow 0}{\mbox{lim }} \frac{1}{|\Omega(\levelset_\eps)\symmDiff \Omega(\levelset)|}(\KK_\eps - \KK_0) (\uu_\eps  - \uu_0) \cdot \pp
\end{align*}
in \eqref{eq_R2discr}, vanishes in the discrete setting. This can be seen as follows:
For $u_\eps^h$, $\eps \geq 0$, and $p^h \in V_h$, we have the expansion in the finite element basis
\begin{align*}
    u_\eps^h(x) = \sum_{i=1}^M u_\eps^{(i)} \varphi_i , \qquad p^h(x) = \sum_{i=1}^M p^{(i)} \varphi_i .
\end{align*}
If we now plug in these discretized functions into \eqref{eq_R2lambdaDef} and consider a fixed mesh size $h$, we get on the other hand
\begin{align*}
    R_2^\lambda(u^h, p^h) = &  \underset{\eps \searrow 0}{\mbox{lim }} \frac{1}{|\omega|} (\lambda_1 - \lambda_2) \int_\omega (\nabla(u_\eps^h - u_0^h))\circ T_\eps \cdot (\nabla p^h)\circ T_\eps \; \mbox dx \\
    =&  \underset{\eps \searrow 0}{\mbox{lim }} \frac{1}{|\omega|}(\lambda_1 - \lambda_2) \sum_{i,j=1}^M (u_\eps^{(i)}-u_0^{(i)})p^{(j)}   \int_\omega (\nabla \varphi_i)(z+\eps x) \cdot (\nabla \varphi_j)(z+\eps x) \; \mbox dx = 0
\end{align*}
where we used the continuity of $\eps \mapsto \uu_\eps$ according to Lemma \ref{lem_uepsu0}. Note that, since the mesh and the basis functions are assumed to be fixed and independent of $\eps$, unlike in the continuous setting, here the continuity of the normal flux of the discrete solution $u_\eps^h$ across the interface $\partial \omega_\eps$ is not preserved. We mention that, when using an extended discretization technique that accounts for an accurate resolution of the material interface (e.g. XFEM \cite{MoesDolbowBelytschko1999} or CutFEM \cite{BurmanClausHansboLarsonMassing2014}), the corresponding discrete sensitivities would include a term corresponding to $R_2^\lambda(u,p)$.
\end{remark}

\subsection{Connection between continuous and discrete shape derivative}
The continuous shape derivative $dg(\Omega)(V)$ for a PDE-constrained shape optimization problem given a shape $\Omega \in \mathcal A$ and a smooth vector field $V$ can also be obtained via a Lagrangian approach. For our problem \eqref{eq::ContinuousProblem}, it can be obtained as $dg(\Omega)(V) = \partial_t G(0, u, p)$ with
\begin{align*}
    G(t, \varphi, \psi) :=& c_1 \int_\Omega \xi(t) \; \mbox dx + c_2 \int_D \tilde{\alpha}_{\Omega}  | \varphi - \targetU\circ T_t|^2 \xi(t) \; \mbox dx \\
    &+ \int_D \lambda_{\Omega} A(t) \nabla \varphi \cdot \nabla \psi + \alpha_{\Omega}  \varphi  \psi \xi(t) \;\mbox dx- \int_D f_{\Omega} \psi   \xi(t)\; \mbox dx
\end{align*}
where $T_t(x) = x + t V(x)$, $\xi(t) = \mbox{det}(\partial T_t)$, $A(t) = \xi(t) \partial T_t^{-1} \partial T_t^{-T}$,
see \cite{Sturm2015} for a detailed description. In the continuous setting, the shape derivative reads in its volume form
\begin{align*}
    d g(\Omega)(V) = \int_D \mathcal S_1^{\Omega} : \partial V + \mathcal S_0^\Omega \cdot V \; \mbox dx
\end{align*}
with $\mathcal S_1^\Omega$ and $\mathcal S_0^\Omega$ given in \eqref{eq_S1} and \eqref{eq_S0}, respectively. Under certain smoothness assumptions, it can be transformed into the Hadamard or boundary form
\begin{align} \label{eq_dg_Hadamard}
    d g(\Omega)(V) = \int_{\partial \Omega} L \, (V\cdot n) \; \mbox dS_x
\end{align}
with $L =( \mathcal S_1^{\Omega, \text{in}} - \mathcal S_1^{\Omega, \text{out}})n \cdot n$ given by
\begin{align} \label{eq_L}
    L = c_1 + c_2(\tilde \alpha_1 - \tilde \alpha_2) (u - \targetU)^2 + (\alpha_1 - \alpha_2) u p - (f_1 - f_2) p + L^\lambda
\end{align}
where $L^\lambda$ is given by
\begin{align}
    L^\lambda :=& (\lambda_1- \lambda_2)  (\nabla u \cdot \tau)(\nabla p \cdot\tau) -  \left(\frac{1}{\lambda_1} - \frac{1}{\lambda_2}\right)( \lambda_\Omega \nabla u \cdot n)(  \lambda_\Omega \nabla p \cdot n) \nonumber \\
    =&\lambda_1 \nabla u_{\text{in}} \cdot \nabla p_{\text{in}} -  \lambda_2 \nabla u_{\text{out}} \cdot \nabla p_{\text{out}} - 2 \lambda_1 (\nabla u_{\text{in}}\cdot n) ( \nabla p_{\text{in}} \cdot n) + 2 \lambda_2 (\nabla u_{\text{out}}\cdot n) ( \nabla p_{\text{out}} \cdot n). \label{eq_Llambda_inout}
\end{align}
Here, $\nabla u_{\text{in}}, \nabla p_{\text{in}}$ and $\nabla u_{\text{out}}, \nabla p_{\text{out}}$ denote the restrictions of the gradients to $\Omega$ and $D \setminus \overline \Omega$, respectively, and $n$ denotes the unit normal vector pointing out of $\Omega$.
Note that, when using a finite element discretization which does not resolve the interface such that the gradients of the discretized state and adjoint variable are continuous, i.e. $\nabla u_{h, \text{in}} = \nabla u_{h, \text{out}}$ and $\nabla p_{h, \text{in}} = \nabla p_{h, \text{out}}$, \eqref{eq_Llambda_inout} becomes
\begin{align} \label{eq_Llambdah}
    L^\lambda_h = (\lambda_1 - \lambda_2) \nabla u_h \cdot \nabla p_h - 2(\lambda_1 - \lambda_2) (\nabla u_h \cdot n)(\nabla p_h \cdot n)
\end{align}

We now discretize the continuous shape derivative formula \eqref{eq_dg_Hadamard} for the vector field $V^{(k)}$ that is obtained from the perturbation of the level set function $\levelset$ in (only) node $\pos_k$, $k \in \setS$ fixed. For that purpose we fix $\phi \in S_h^1(D)$ and the corresponding domain $\Omega(\phi)$. Note that we consider $V^{(k)}$ to be supported only on the discretized material interface $\partial \Omega(\phi) \cap D$. We begin with the case of the pure volume cost function by setting $c_2=0$.

\begin{proposition} \label{prop_SD_vol}
    Let $c_2=0$ and $\pos_k \in \setS$ fixed. Let $V^{(k)}$ the vector field that corresponds to a perturbation of the value of $\levelset$ at position $\pos_k$. Then
    \begin{align*}
        dg(\Omega(\phi))(V^{(k)}) = c_1 \sum_{l \in C_k} d_ka_l
    \end{align*}
    where $d_k a_l$ is as in \eqref{eq::derivativeArea}.

\end{proposition}

\begin{proof}
    For $c_2=0$ we also have $p=0$ and thus $L = c_1$, i.e., we are in the case of pure volume minimization. From \eqref{eq_dg_Hadamard} we know that $dg(\Omega(\phi))(V^{(k)}) = c_1 \int_{\partial \Omega(\phi)} V^{(k)} \cdot n \; \mbox dS_x$.
    First of all, we note that the vector field $V^{(k)}$ corresponding to a perturbation of $\levelset$ at node $x_k$ is only nonzero in elements $\tau_l$ for $l \in C_k$ with $C_k$ as defined in \eqref{eq::areaChangeSet}. Thus, the shape derivative reduces to
    \begin{align*}
        dg(\Omega(\phi))(V^{(k)}) = c_1 \sum_{l \in C_k} \int_{\tau_l \cap \partial \Omega(\phi)} V^{(k)}\cdot n \; \mbox dS_x.
    \end{align*}
    We compute the vector field $V^{(k)}$ and normal vector $n$ explicitly depending on the cut situation. Recall the sets $\setIk^{A_+}, \setIk^{A_-}, \setIk^{B_+}, \setIk^{B_-}, \setIk^{C_+}, \setIk^{C_-}$ introduced in \eqref{eq_def_Ixk_ABCpm}, see also Figure \ref{fig::sets}. 
    
    Given two points $\pp$ and $\qq$ and their respective level set values $a$ and $b$ of different sign, $ab<0$, we denote the root of the linear interpolating function by
    \begin{align*}
        \yy(\pp, \qq, a, b) = \pp -  \frac{a}{b-a}(\qq - \pp)
    \end{align*}
    and note that $\frac{d}{da} \yy(\pp, \qq, a, b) =  -  \frac{b}{(b-a)^2}(\qq - \pp)$.

    We begin with Configuration A. For an element index $l \in \setIk^{A_+}\cup \setIk^{A_-}$, we denote the vertices of element $\tau_l$ by $\pos_{l_1}, \pos_{l_2}, \pos_{l_3}$ and assume their enumeration in counter-clockwise order with $\pos_k = \pos_{l_1}$. The corresponding values of the given level set function $\phi$ are denoted by $\phi_{l_1}, \phi_{l_2}, \phi_{l_3}$, respectively. We parametrize the line connecting the two roots of the perturbed level set function along the edges by
    \begin{align*}
        p^A[\eps](s) = (1-s) \yy(\pos_{l_1}, \pos_{l_2}, \phi_{l_1}+\eps, \phi_{l_2}) +  s \yy(\pos_{l_1}, \pos_{l_3}, \phi_{l_1}+\eps,\phi_{l_3})
    \end{align*}
    and obtain the vector field corresponding to the perturbation of $\phi_{k} = \phi_{l_1}$ along the line $\tau_l \cap \partial \Omega(\phi)$ as
    \begin{align*}
        \hat V^A(s) = \frac{d}{d\eps} p^A[\eps](s)|_{\eps = 0} = (1-s) \frac{- \phi_{l_2}}{(\phi_{l_2}-\phi_{l_1})^2} (\pos_{l_2} - \pos_{l_1}) + s \frac{- \phi_{l_3}}{ (\phi_{l_3} - \phi_{l_1})^2} (\pos_{l_3}-\pos_{l_1}).
    \end{align*}
    Introducing the notation $d_{ki,kj} := |\yy(\pos_{l_k}, \pos_{l_i}, \phi_{l_k}, \phi_{l_i})-\yy(\pos_{l_k}, \pos_{l_j}, \phi_{l_k}, \phi_{l_j}) |$ for the length of the interface in element $\tau_l$, the normed tangential vector along $\tau_l \cap \partial \Omega(\phi)$ and the normed normal vector pointing out of $\Omega(\phi)$ are given by
    \begin{align*}
        \hat t^A =& \frac{\yy(\pos_{l_1}, \pos_{l_3}, \phi_{l_1}, \phi_{l_3})  - \yy(\pos_{l_1}, \pos_{l_2}, \phi_{l_1}, \phi_{l_2}) }{d_{13,12}} \\
        \hat n^{A_+} =& \frac{\phi_{l_1}}{d_{13,12}} \left( \frac{1}{\phi_{l_2} - \phi_{l_1}} R (\pos_{l_2} - \pos_{l_1})- \frac{1}{\phi_{l_3} - \phi_{l_1}} R (\pos_{l_3} - \pos_{l_1}) \right)
        =-\hat n^{A_-} 
    \end{align*}
    where $R$ denotes a 90 degree counter-clockwise rotation matrix, $R = \begin{pmatrix} 0& -1\\ 1 & 0 \end{pmatrix}$. Noting that $(\pos_{l_3} - \pos_{l_1})^\top R (\pos_{l_2} - \pos_{l_1}) = |\mbox{det}J_l| = - (\pos_{l_2} - \pos_{l_1})^\top R (\pos_{l_3} - \pos_{l_1})$ and $(\pos_{l_j} - \pos_{l_1})^\top R (\pos_{l_j} - \pos_{l_1}) = 0$, $j=2,3$, we get
    \begin{align} \label{eq_VAnA}
        \hat V^A(s) \cdot \hat n^{A_+} = -\frac{|\mbox{det} J_l| \phi_{l_1}}{d_{13,12}} \frac{\phi_{l_2}(\phi_{l_3}-\phi_{l_1}) + s \phi_{l_1}(\phi_{l_2} - \phi_{l_3})}{(\phi_{l_2}-\phi_{l_1})^2(\phi_{l_3}-\phi_{l_1})^2} = - \hat V^A(s) \cdot \hat n^{A_-} .
    \end{align}
    Finally, by elementary computation we obtain for $l \in \setIk^{A_+}$
    \begin{align*}
        \int_{\tau_l \cap \partial \Omega(\phi)} V^{(k)} \cdot n \; \mbox dS_x =& d_{13,12} \int_0^1 \hat V^A(t) \cdot \hat n^{A_+} \; \mbox dt \\
        =&| \mbox{det} J_l| \phi_{l_1} \frac{-\phi_{l_2}\phi_{l_3} + \frac12 \phi_{l_1}(\phi_{l_2}+ \phi_{l_3})}{(\phi_{l_2}-\phi_{l_1})^2(\phi_{l_3}-\phi_{l_1})^2} 
    \end{align*}
    and the same formula with a different sign for $l \in \setIk^{A_-}$. Proceeding analogously, we obtain for $l \in \setIk^{B_+}$ and $l \in \setIk^{C_+}$
    \begin{align}
        \hat V^B(s) \cdot \hat n^{B_+} =& \frac{|\mbox{det} J_l|}{d_{23,21}} \frac{(1-s) \phi_{l_2}^2 }{(\phi_{l_2} - \phi_{l_1})^2 (\phi_{l_3} - \phi_{l_2})} = - \hat V^B(s) \cdot \hat n^{B_-}, \\
        \hat V^C(s) \cdot \hat n^{C_+} = & \frac{|\mbox{det} J_l|}{d_{31,32}} \frac{(1-s) \phi_{l_3}^2}{(\phi_{l_3}-\phi_{l_1})^2 (\phi_{l_2}-\phi_{l_3})} = -\hat V^C(s) \cdot \hat n^{C_-},    \label{eq_VCnC}
    \end{align}
    and further
    \begin{align*}
         \int_{\tau_l \cap \partial \Omega(\phi)} V^{(k)} \cdot n \; \mbox dS_x=& \frac{|\mbox{det}J_l|}{2} \frac{\phi_{l_2}^2}{(\phi_{l_2}-\phi_{l_1})^2(\phi_{l_3}-\phi_{l_2})}, \quad l \in \setIk^{B_+}, \\
        \int_{\tau_l \cap \partial \Omega(\phi)} V^{(k)} \cdot n \; \mbox dS_x =& \frac{|\mbox{det}J_l|}{2} \frac{-\phi_{l_3}^2}{(\phi_{l_3}-\phi_{l_1})^2(\phi_{l_3}-\phi_{l_2})}, \quad l \in \setIk^{C_+},
    \end{align*}
    respectively. Again, the formulas for $l \in \setIk^{B_-}$ and $l \in \setIk^{C_-}$ just differ by a different sign.
    
    Finally, comparing the computed values with the formulas of $d_k a_l$ \eqref{eq::derivativeArea} yields the claimed result.
\end{proof}

In view of Proposition \ref{prop_SD_vol}, the definition in \eqref{eq::shapeDerivativeAlternative} and Remark \ref{remark::area}, we see that, in the case $c_2=0$, it holds
\begin{align}
    \hat d g(\Omega(\phi))(V^{(k)}) =  \frac{c_1 \sum_{l\in C_k} d_ka_l}{\sum_{l\in C_k} d_k \tilde a_l} = - c_1,
\end{align}
which is in alignment with the first term of the formula in \eqref{eq::formulaShapeDerivative}.

Next, we consider the general PDE-constrained case where $c_2>0$.
\begin{proposition} \label{prop_SD_pde}
 Let $c_1 = 0$ and $\pos_k \in \setS$ fixed. Let $V^{(k)}$ the vector field that corresponds to a perturbation of the value of $\phi$ at position $\pos_k$. Then
 \begin{align}  \label{eq_formula_dg_pde}
  \begin{aligned}
    dg(\Omega(\phi))(V^{(k)}) =& \; \; \;\; \, \Delta \lambda
	\sum_{l \in C_k}  \mathbf p_{l}^\top \mathbf k_{0,l} \mathbf u_{l} \,d_k a_{l} 
	+ \Delta \alpha \sum_{l \in C_k}  \mathbf p_{l}^\top d_k\mathbf m^I_{l} \,\mathbf u_{l}
	  \\
	 &  - \Delta f \sum_{l \in C_k} \mathbf p_{l}^\top d_k\mathbf f_l^I 
	+c_2 \Delta \tilde\alpha \sum_{l \in C_k}  (\mathbf u_{l}-\hat{\mathbf u}_{l})^\top d_k\mathbf m^I_{l} \,(\mathbf u_{l}-\hat{\mathbf u}_{l}) \\
	& -2 \Delta \lambda \sum_{l \in C_k}(\nabla u_h \cdot n^I)|_{\tau_l \cap \partial \Omega(\phi)}(\nabla p_h \cdot n)|_{\tau_l \cap \partial \Omega(\phi)} d_ka_l,
    \end{aligned}
 \end{align}
where we use the same notation as in Theorem \ref{thm_numTopShapeDer_formula}. In particular, $d_k \mathbf m^I_l$ and $d_k \mathbf f^I_l$ depend on the cut situation, $I \in \{A^+, A^-, B^+, B^-, C^+, C^-\}$ and are given explicitly in Appendix \ref{app::shapeDerivative}.
\end{proposition}
\begin{proof}
Let an element index $l \in C_k$ fixed and $\mathbf u_l = [u_{l_1}, u_{l_2}, u_{l_3}]^\top$, $\mathbf p_l = [p_{l_1}, p_{l_2}, p_{l_3}]^\top$ contain the nodal values of the finite element functions $u_h$ and $p_h$ corresponding to the three vertices $\pos_{l_1}$, $\pos_{l_2}$, $\pos_{l_3}$ of element $l$, respectively. Also here, the ordering is in counter-clockwise direction starting with $\pos_{l_1} = \pos_k$.
We compute the shape derivative \eqref{eq_dg_Hadamard} with $L$ given in \eqref{eq_L} after discretization (i.e. after replacing the functions $u$, $p$, $\hat u$ by finite element approximations $u_h$, $p_h$, $\hat u_h$). In particular, the term $L^\lambda$ is approximated by \eqref{eq_Llambdah}.
Depending on how the material interface $\partial \Omega(\phi)$ cuts through element $\tau_l$, the normal component of the vector field $V^{(k)}$ along the line $\tau_l \cap \partial \Omega(\phi)$ is given in \eqref{eq_VAnA}--\eqref{eq_VCnC}.
For and $I \in \{A^+, A^-, B^+, B^-, C^+, C^-\}$, it can be seen by elementary yet tedious calculations that
\begin{align*}
    \int_{\tau_l \cap \partial \Omega(\phi)} p_h(x) V^{I}(x) \cdot n^{I} \, \mbox dS_x =& \; \mathbf p_l^\top d_k \mathbf f_l^I,\\
    \int_{\tau_l \cap \partial \Omega(\phi)} u_h(x) p_h(x) V^{I}(x) \cdot n^{I} \, \mbox dS_x =& \; \mathbf p_l^\top d_k \mathbf m^I_l \mathbf u_l,\\
    \int_{\tau_l \cap \partial \Omega(\phi)} (u_h(x)-\hat u_h(x))^2 V^{I}(x) \cdot n^{I} \, \mbox dS_x =& \; (\mathbf u_l - \hat{\mathbf u}_l)^\top d_k \mathbf m^I_l (\mathbf u_l - \hat{\mathbf u}_l),
\end{align*}
with $d_k \mathbf m^I_l$ and $d_k \mathbf f^I_l$ as given in Appendix \ref{app::shapeDerivative}.
Examplarily, we illustrate the calculation for the second of these terms for the cut situation $I = A^+$, see Figure \ref{fig::sets}(a). Let $u_{l,12}$ and $u_{l,13}$ denote the values of the linear function $u_h|_{\tau_l}$ at the intersection of the interface $\partial \Omega(\phi)$ with the edges that connect the point $\pos_{l_1}$ with $\pos_{l_2}$ and $\pos_{l_1}$ with $\pos_{l_3}$, respectively. Note the relations $u_{l,12} = \frac{u_{l_2} \phi_{l_1} - u_{l_1} \phi_{l_2}}{\phi_{l_1}-\phi_{l_2}}$ and $u_{l,13} = \frac{u_{l_3} \phi_{l_1} - u_{l_1} \phi_{l_3}}{\phi_{l_1}-\phi_{l_3}}$. Analogously we define the values $p_{l,12}$ and $p_{l,13}$. The function $u_h$ along the line $\tau_l \cap \partial \Omega(\phi)$ can now be written as $\hat u_h(s) = u_{l,12} + s (u_{l,13} - u_{l,12})$, $s \in [0,1]$ and we get
\begin{align*}
    \int_{\tau_l \cap \partial \Omega(\phi)}& u_h(x) p_h(x) V^{I}(x) \cdot n^{I} \, \mbox dS_x \\
    &= d_{13,12} \int_0^1 (u_{l,12} + s (u_{l,13} - u_{l,12})) (p_{l,12} + s (p_{l,13} - p_{l,12})) \hat V^{A^+}(s) \cdot \hat n^{A^+} \; \mbox ds \\
    &= \mathbf p_l^\top d_k \mathbf m^{A^+}_l \mathbf u_l
\end{align*}
where $d_{13,12} = |\tau_l \cap \partial \Omega(\phi)|$. The last equality is obtained by plugging in \eqref{eq_VAnA} and straightforward (yet tedious) calculation.
Finally, since $u_h$ and $p_h$ are linear and the normal vector is constant on $\tau_l \cap \partial \Omega(\phi)$, we see that $L_h^\lambda$ is constant and, using Proposition \ref{prop_SD_vol}, we obtain
\begin{align*}
    \int_{\tau_l \cap \partial \Omega(\phi)} L_h^\lambda(x) V^I(x)\cdot n^I \; \mbox dS_x =& L_h^\lambda(x) \int_{\tau_l \cap \partial \Omega(\phi)}  V^I(x)\cdot n^I \; \mbox dS_x \\
    =&\Delta \lambda  \left( \mathbf p_l^\top \mathbf k_{0,l} \mathbf u_l -2 (\nabla u_h \cdot n^I)|_{\tau_l \cap \partial \Omega(\phi)}(\nabla p_h \cdot n)|_{\tau_l \cap \partial \Omega(\phi)} \right) d_ka_l.
\end{align*}
\end{proof}

Combining the findings of Propositions \ref{prop_SD_vol} and \ref{prop_SD_pde} and dividing by $d_k \tilde a$ (defined in \ref{remark::area}), we obtain the resulting formula for the alternative definition of the shape derivative as defined in \eqref{eq::shapeDerivativeAlternative}
\begin{align} \label{eq_dhatg_discr}
    \begin{aligned}
    \hat d &g(\Omega(\phi))(V^{(k)}) = -c_1 + \frac{1}{d_k \tilde a} \Bigg( \Delta \lambda
\sum_{l \in C_k}  \mathbf p_{l}^\top \mathbf k_{0,l} \mathbf u_{l} \,d_k a_{l} 
+ \Delta \alpha \sum_{l \in C_k}  \mathbf p_{l}^\top d_k\mathbf m^I_{l} \,\mathbf u_{l}  - \Delta f \sum_{l \in C_k} \mathbf p_{l}^\top d_k\mathbf f_l^I 
    \\
    & +c_2 \Delta \tilde\alpha \sum_{l \in C_k}  (\mathbf u_{l}-\hat{\mathbf u}_{l})^\top d_k\mathbf m^I_{l} \,(\mathbf u_{l}-\hat{\mathbf u}_{l}) 
    -2 \Delta \lambda \sum_{l \in C_k}(\nabla u_h \cdot n^I)|_{\tau_l \cap \partial \Omega(\phi)}(\nabla p_h \cdot n)|_{\tau_l \cap \partial \Omega(\phi)} d_ka_l \Bigg).
    \end{aligned}
\end{align}
\begin{remark}
    Note that \eqref{eq_dhatg_discr} is obtained by discretizing the continuous shape derivative \eqref{eq_dg_Hadamard}--\eqref{eq_Llambda_inout}. We see immediately that \eqref{eq_dhatg_discr} resembles the formula for the discrete \topshapeDerivative derivative for nodes $\pos_k \in \setS$ \eqref{eq::formulaShapeDerivative}. The only difference is the occurence of the last term in \eqref{eq_dhatg_discr}, which is not accounted for when performing the sensitivity analysis in the discrete setting.

Note that this term stems from the presence of $\partial T_t^{-1} \partial T_t^{-T}$ in the matrix $A(t)$, which, in turn, originates from two applications of the chain rule, $(\nabla \varphi) \circ T_t = \partial T_t^{-\top} \nabla (\varphi \circ T_t)$. 
Similarly as in the case of the topological derivative in Section \ref{sec_connTD}, the reason for this discrepancy is the fact that, for the given discretization scheme, the gradients of the finite element basis functions are constant on each element and thus $(\nabla \varphi) \circ T_t = \nabla \varphi$ for small enough shape perturbation parameter $t$.

\end{remark}

\begin{remark}
    As a second, conceptual difference between the classical continuous shape derivative defined by \eqref{eq::shapeDerivative} and our discrete counterpart \eqref{eq::formulaShapeDerivative}, we recall that in the definition \eqref{eq::topDerivative}
    we divided by the change of volume also in the case $\pos_k \in \setS$.
\end{remark}

\section{Numerical Experiments}\label{sec::numericalExamples}
In this section, we verify our implementation of the numerical \topshapeDerivative derivative derived in Section \ref{sec::numTopShapeDer} by numerical experiments, before applying a level-set based topology optimization algorithm based on these sensitivities to our model problem.
\subsection{Verification}
The implementation of the \topshapeDerivative derivative is verified by comparing the computed sensitivity values against numerical values obtained by three different approaches. These are (i) a finite difference test, (ii) an application of the complex step derivative \cite{martins2003complex} and (iii) a test based on hyper-dual numbers developed in \cite{fike2011development}. All tests are conducted for a fixed configuration.

We recall the definition of the \topshapeDerivative derivative \eqref{eq::topDerivative} at a node $\pos_k$ of the mesh
\begin{align} \label{eq_dJ_lim}
    d\redPhiObjectiveFunction(\levelset)(\pos_k) = \lim_{\varepsilon\searrow 0}
        \delta \redPhiObjectiveFunction_{\eps}(\levelset)(\pos_k) \quad \mbox{ with } \quad
    \delta \redPhiObjectiveFunction_{\eps}(\levelset)(\pos_k) :=
\frac{\redPhiObjectiveFunction({O_{k,\eps}} \levelset)-\redPhiObjectiveFunction(\levelset)}
{|\Omega({O_{k,\eps}}\levelset) \symmDiff \Omega(\levelset)|}
\end{align}
where $O_{k,\eps}$ represents the operator $\Tplus$, $\Tminus$, $S_{k,\varepsilon}$ depending on whether the node $\pos_k$ is in $\setTminus$, $\setTplus$ or $\setS$, respectively.

\subsubsection{Finite difference test}
For the finite difference (FD) test, we compute the errors
\begin{align} \label{eq_def_errors_eSeT}
e_S^{FD}(\eps) = \sqrt{\sum_{\pos_k \in \setS} (\delta \redPhiObjectiveFunction_{\eps}(\levelset)(\pos_k) - d\redPhiObjectiveFunction(\levelset)(\pos_k))^2}, \quad
e_T^{FD}(\eps) = \sqrt{\sum_{\pos_k \in \setTminus \cup \setTplus} (\delta \redPhiObjectiveFunction_{\eps}(\levelset)(\pos_k) - d\redPhiObjectiveFunction(\levelset)(\pos_k))^2}
\end{align}
for a decreasing sequence of values for $\eps$. 
The results are shown in Figure \ref{fig:verify}(a). We observe convergence of order $\eps$ up to a point where the cancellation error dominates. 
\begin{figure}
    \begin{tabular}{ccc}
        \includegraphics[width=0.33\linewidth]{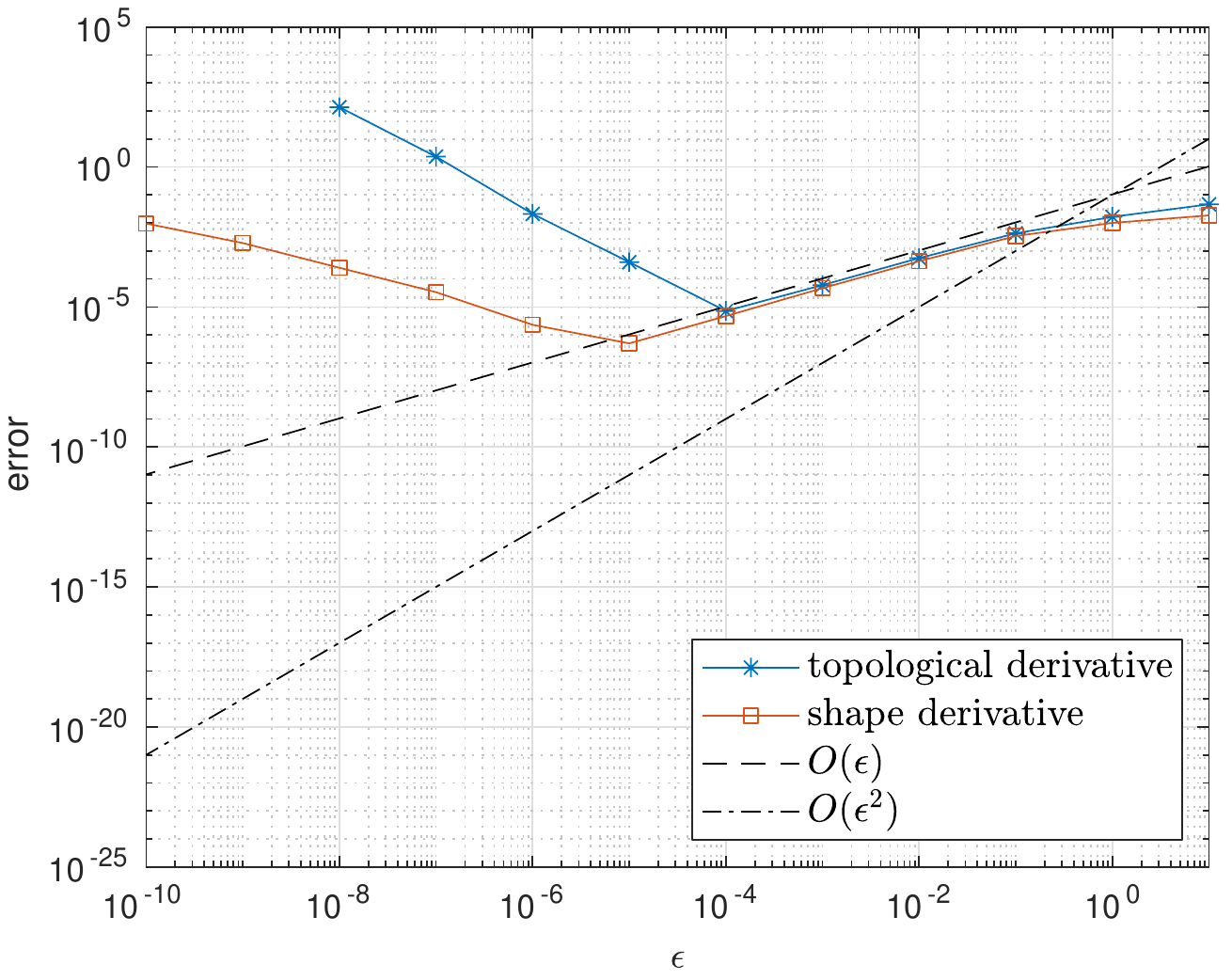} & \includegraphics[width=0.33\linewidth]{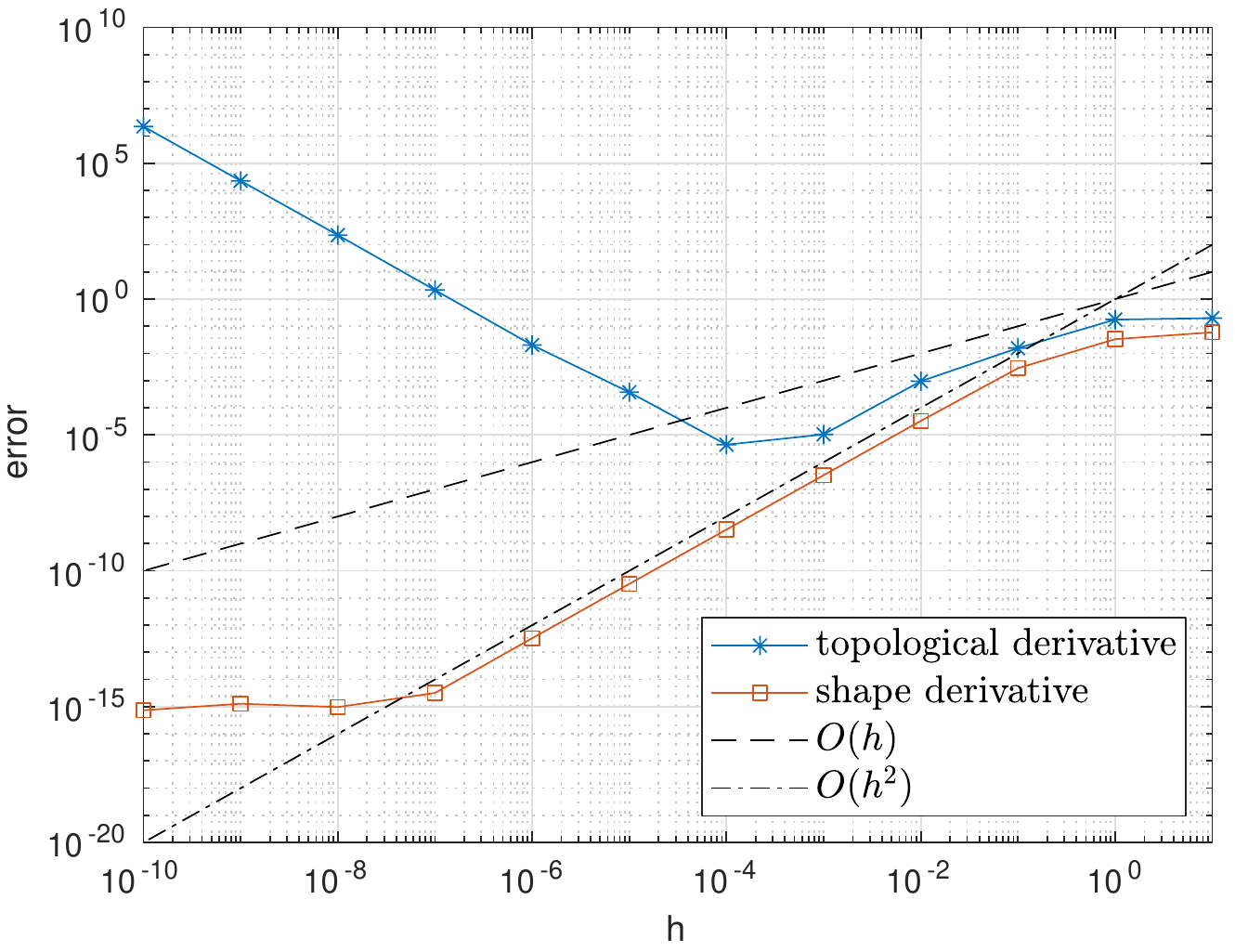} &
        \includegraphics[width=0.33\linewidth]{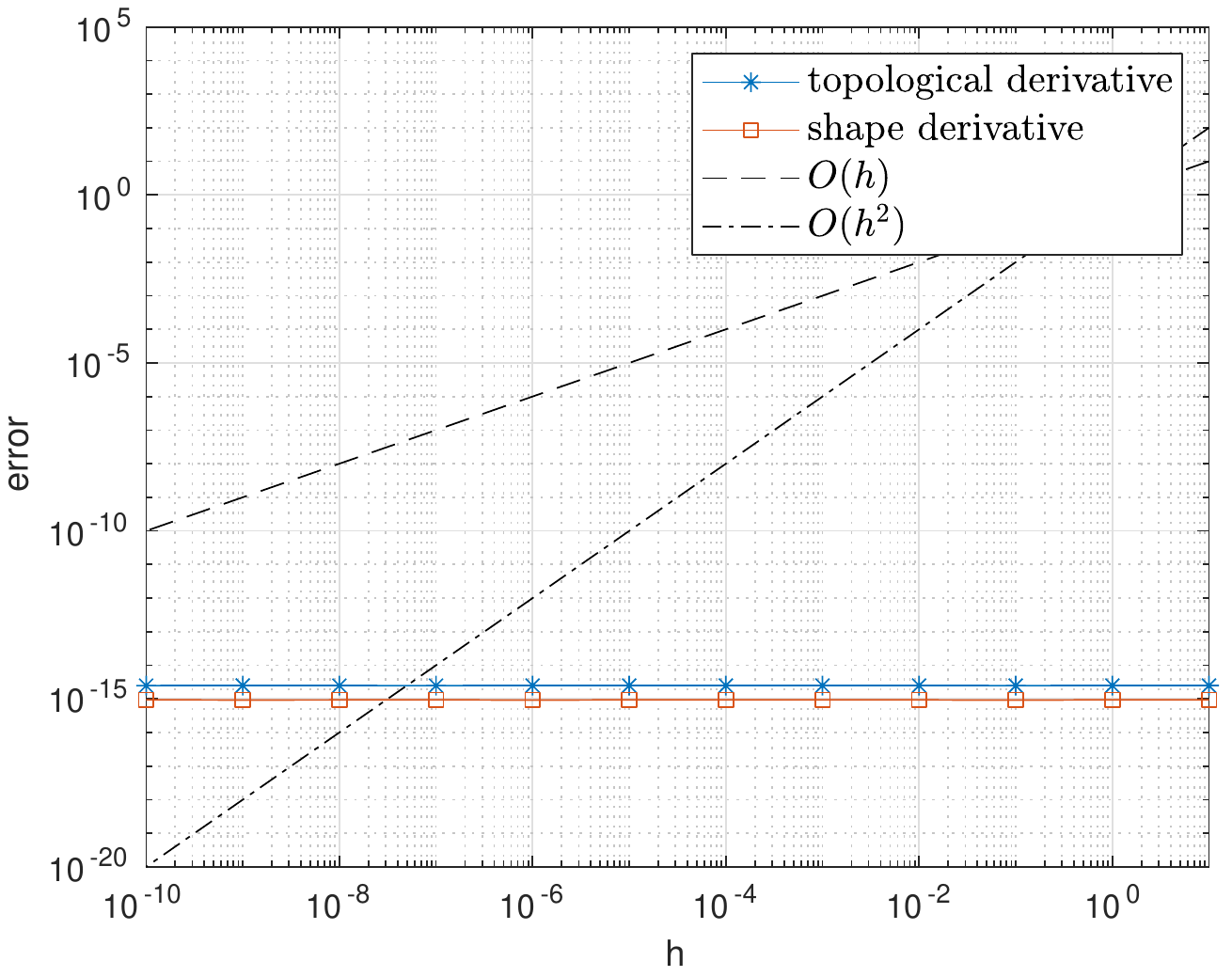} \\
        (a) & (b) & (c) 
    \end{tabular}
\caption{(a) Results of the finite difference test. (b) Results obtained with the complex step derivative. (c) Results obtained with hyper-dual numbers.}
    \label{fig:verify}
\end{figure}

\subsubsection{Complex step derivative test}
In order to overcome this drawback of the finite difference test, we next consider a test based on the complex step (CS) derivative \cite{martins2003complex}. For that purpose, using Remark \ref{remark::area}, let us first rewrite \eqref{eq_dJ_lim} as
\begin{align} \label{eq_dJ_lim2}
    d\redPhiObjectiveFunction(\levelset)(\pos_k) = \frac{\lim_{\varepsilon\searrow 0}\frac{\redPhiObjectiveFunction({O_{k,\eps}} \levelset)-\redPhiObjectiveFunction(\levelset)}
{\eps^o}}{\lim_{\varepsilon\searrow 0}\frac{|\Omega({O_{k,\eps}}\levelset) \symmDiff \Omega(\levelset)|}{\eps^o}}
 =\frac{1}{d_k \tilde a} \lim_{\varepsilon\searrow 0}\frac{\redPhiObjectiveFunction({O_{k,\eps}} \levelset)-\redPhiObjectiveFunction(\levelset)}
{\eps^o},
\end{align}
where $o=1$ if $\pos_k \in \setS$ and $o=2$ if $\pos_k \in \setTminus \cup \setTplus$. Moreover,
assuming a higher order expansion of the form
\begin{align} \label{eq_expansion_op3}
    \redPhiObjectiveFunction(O_{k,\eps}\levelset) = \redPhiObjectiveFunction(\levelset) 
    + \eps^o \, d_k \tilde a \, d \redPhiObjectiveFunction(\levelset)(\pos_k)
    + \eps^{o+1} \, d_k \tilde a \, d^2 \redPhiObjectiveFunction(\levelset)(\pos_k)
    + \eps^{o+2} \, d_k \tilde a \, d^3 \redPhiObjectiveFunction(\levelset)(\pos_k) + \mathfrak{o}(\eps^{o+2})
\end{align}
with some higher order sensitivities $
d^2 \redPhiObjectiveFunction(\levelset)(\pos_k)$, $d^3 \redPhiObjectiveFunction(\levelset)(\pos_k)$ and assuming that \eqref{eq_expansion_op3} also holds for complex-valued $\eps$, we can follow the idea of the complex step derivative \cite{martins2003complex}: Setting $\eps = ih$ in \eqref{eq_expansion_op3} with $h>0$ and $i$ the complex unit yields
\begin{align}
    d \redPhiObjectiveFunction(\levelset)(\pos_k) = \frac{\text{Im}(\redPhiObjectiveFunction(O_{k,ih}\levelset))}{h \, d_k \tilde a} + \mathcal O(h^2)
\end{align}
in the case $o=1$ where $\pos_k \in \setS$, and 
\begin{align*}
    d \redPhiObjectiveFunction(\levelset)(\pos_k) = \frac{\text{Re}(\redPhiObjectiveFunction(O_{k,ih}\levelset)-\redPhiObjectiveFunction(\levelset))}{-h^2 \, d_k \tilde a} + \mathcal O(h^2)
\end{align*}
in the case $o=2$ where $\pos_k \in \setTminus \cup \setTplus$. This means
\begin{align*}
    d \redPhiObjectiveFunction(\levelset)(\pos_k) = \delta \redPhiObjectiveFunction_h^{CS}(\levelset)(\pos_k) + \mathcal O(h^2)
\end{align*}
with
\begin{align} \label{eq_deltaJh_CS}
    \delta \redPhiObjectiveFunction_h^{CS}(\levelset)(\pos_k) := \begin{cases}
        \frac{\text{Re}(\redPhiObjectiveFunction(\TplusComplex \levelset)-\redPhiObjectiveFunction(\levelset))}{-h^2 \, d_k \tilde a}, & \pos_k \in \setTminus ,\\
        \frac{\text{Re}(\redPhiObjectiveFunction(\TminusComplex \levelset)-\redPhiObjectiveFunction(\levelset))}{-h^2 \, d_k \tilde a}, & \pos_k \in \setTplus ,\\
        \frac{\text{Im}(\redPhiObjectiveFunction(\SComplex \levelset))}{h \, d_k \tilde a}, & \pos_k \in \setS.
        \end{cases}
\end{align}
Analogously to \eqref{eq_def_errors_eSeT}, we define the summed errors $e_S^{CS}(h)$ and $e_T^{CS}(h)$ by just replacing $\delta \redPhiObjectiveFunction_\eps(\levelset)(\pos_k)$ by $\delta \redPhiObjectiveFunction_h^{CS}(\levelset)(\pos_k)$ defined above. Figure  \ref{fig:verify}(b) shows the errors $e_S^{CS}$ and $e_T^{CS}$ for a positive, decreasing sequence of $h$ where we observe quadratic decay for both errors. While the error $e_S^{CS}$ corresponding to the shape nodes $\pos_k \in \setS$ decays to machine precision, the error $e_T^{CS}$ corresponding to the interior nodes $\pos_k \in \setTminus \cup \setTplus$ deteriorates at some point due to the cancellation error ocurring when subtracting $\redPhiObjectiveFunction(\levelset)$ from $\redPhiObjectiveFunction(O_{k,ih}\levelset)$ in \eqref{eq_deltaJh_CS}.

\subsubsection{Test based on hyper-dual numbers}
In order to overcome this cancellation error also for the case of $\pos_k \in \setTminus \cup \setTplus$, we resort to hyper-dual (HD) numbers as introduced in \cite{fike2011development}. Here, the idea is to consider numbers with three non-real components denoted by $E_1$, $E_2$ and $E_1 E_2$ with $E_1^2 = E_2^2 = (E_1E_2)^2=0$. Assuming that expansion \eqref{eq_expansion_op3} holds up to order $o+1$ also for such hyper-dual values of $\eps$, we can set $\eps = h E_1 + h E_2 + 0 E_1E_2$ for some $h>0$. For $o=1$, considering only the first non-real part (i.e., the $E_1$-part) and exploiting that $E_1^2 = 0$, we obtain the equality
\begin{align} \label{eq_dJHD_shape}
    d \redPhiObjectiveFunction(\levelset)(\pos_k) = \frac{E_1\text{part}(\redPhiObjectiveFunction(O_{k, h E_1 + h E_2}\levelset))}{h \, d_k \tilde a}
\end{align}
for $\pos_k \in \setS$. Similarly, with the same choice of $\eps$, by considering only the $E_1E_2$-part of the expansion and exploiting $E_1^2=E_2^2=E_1^2E_2^2=0$, we obtain for $o=2$
\begin{align} \label{eq_dJHD_topo}
    d \redPhiObjectiveFunction(\levelset)(\pos_k) = \frac{E_1E_2\text{part}(\redPhiObjectiveFunction(O_{k, h E_1 + h E_2}\levelset))}{2 h^2 \, d_k \tilde a}
\end{align}
for $\pos_k \in \setTminus \cup \setTplus$. In this case, the corresponding summed errors $e_S^{HD}(h)$ and $e_T^{HD}(h)$ vanish for arbitrary $h \in \mathbb R$. This is also observed numerically since both \eqref{eq_dJHD_shape} and \eqref{eq_dJHD_topo} suffer neither from a truncation nor a cancellation error. Figure \ref{fig:verify}(c) shows that the the obtained results agree up to machine precision with the derivatives obtained by \eqref{eq::formulaTopologicalDerivativePlus}, \eqref{eq::formulaTopologicalDerivativeMinus}, and the respective formula for the shape derivative \eqref{eq::formulaShapeDerivative}.
\subsection{Application of optimization algorithm to model problem}
Finally we show the use of the numerical \topshapeDerivative derivative computed in Section \ref{sec::numTopShapeDer} within a level-set based topology optimization algorithm. We first state the precise model problem, before introducing the algorithm and showing numerical results.
\subsubsection{Problem setting}\label{sec::numericalProblemSetting}
We consider the unit square $D=[0,1]^2$ and minimize the objective function \eqref{eq::ContinuousObjective} with $c_1=0$ and $c_2=1$ subject to the PDE constraint in \eqref{eq::ContinuousProblem}. The chosen problem parameters are shown in Table \ref{tab::parameters}.
\begin{table}
	\centering
	\begin{tabular}{cccccccc}
		\toprule
		$\tilde\alpha_1$ & $\tilde\alpha_2$ &
		$\alpha_1$ &  $\alpha_2$ &
		$\lambda_1$ &  $\lambda_2$ &
		$f_1$ &  $f_2$ \\
		1 & 0.9 & 1 & 0.2 & 1 & 0.6 & 1 & 0.5 \\
		\bottomrule
	\end{tabular}
	\caption{Problem parameters for the numerical experiment}
	\label{tab::parameters}
\end{table}
We consider a mixed Dirichlet-Neumann problem by choosing
\begin{equation*}
\Gamma_D = \{(x,y)\in\partial D| y=0 \mbox{ or } y=1\} , \qquad
\Gamma_N = \partial D \setminus \Gamma_D,
\end{equation*}
and $g_D(x,y) = y$, $g_N(x,y) = 0$.
In order to define a desired state $\hat u$, we choose a level-set function $\phi_d$ which implies a desired shape $\Omega_d$, compute the corresponding solution $u^*$ to $\phi_d$ and set $\hat u := u^*$. Then, by construction, $(\Omega_d, \hat u)$ is also the solution of the design optimization problem.
In the numerical tests we used five different meshes with 145, 545, 2113, 8321, and 33025 nodes respectively. 
For each mesh we obtain $\phi_d$ by interpolation of
\begin{equation}
\bar\phi_d(x,y) = \left((x-0.3)^2+(y-0.4)^2-0.2^2\right)\left((x-0.7)^2+(y-0.7)^2-0.1^2\right).
\end{equation}
This yields that $\Omega$ are two (approximated) circles with radii $0.2$ and $0.1$ respectively, see Figure \ref{fig:holes-meshes}.
\begin{figure}
	\begin{subfigure}[t]{0.19\textwidth}
		\includegraphics[width=\textwidth]{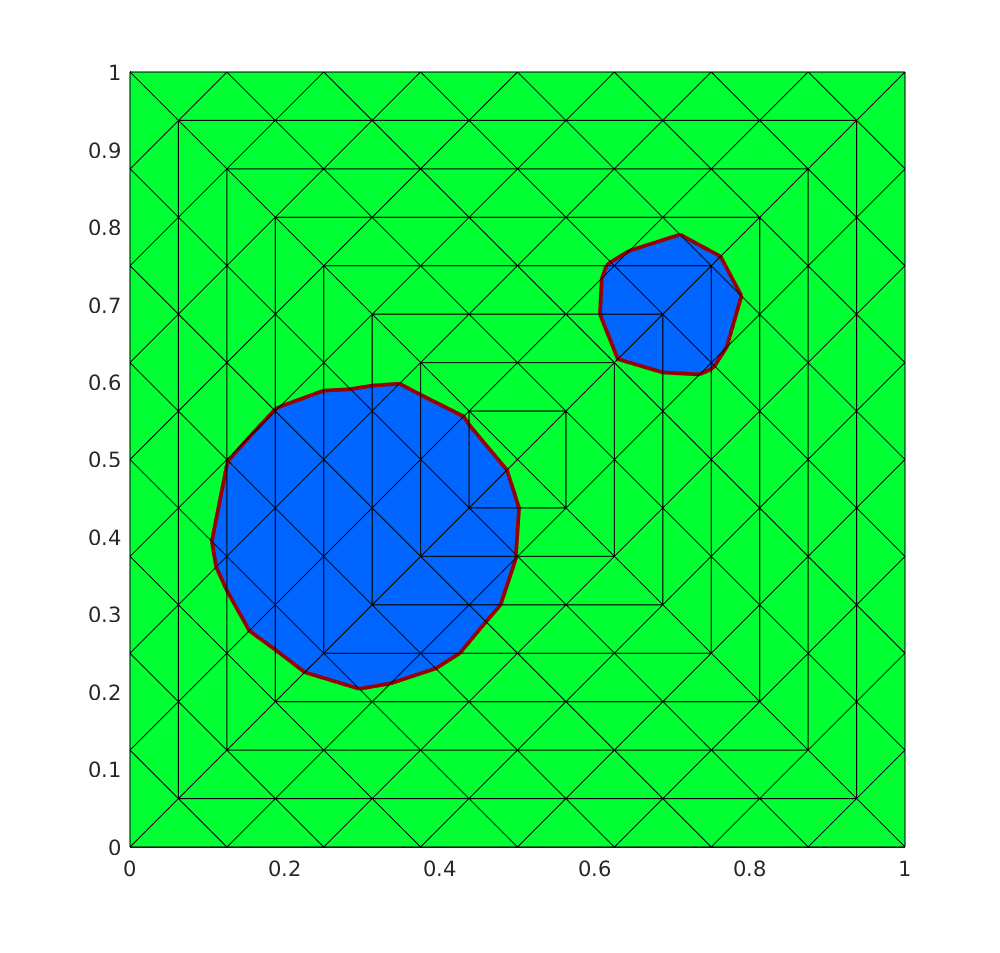}
		\subcaption{145 nodes}
	\end{subfigure}\hfill
	\begin{subfigure}[t]{0.19\textwidth}
		\includegraphics[width=\textwidth]{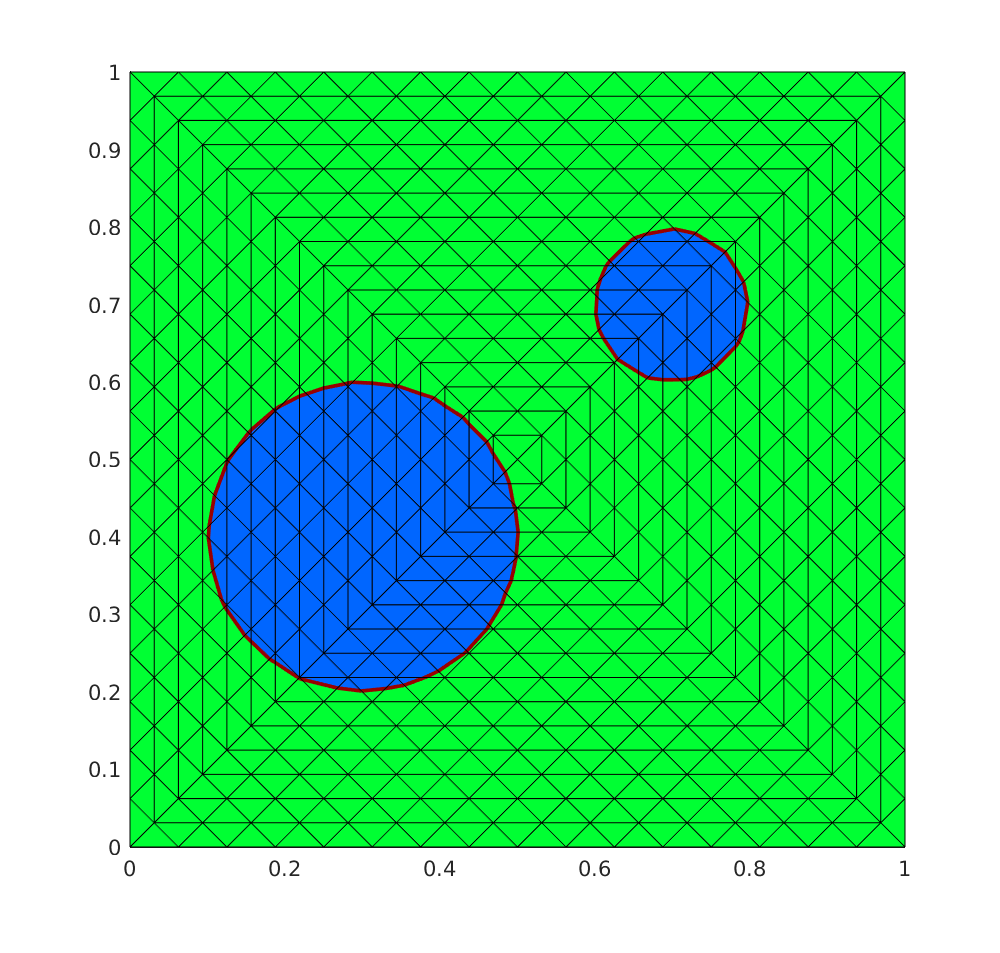}
		\subcaption{545 nodes}
	\end{subfigure}\hfill
	\begin{subfigure}[t]{0.19\textwidth}
		\includegraphics[width=\textwidth]{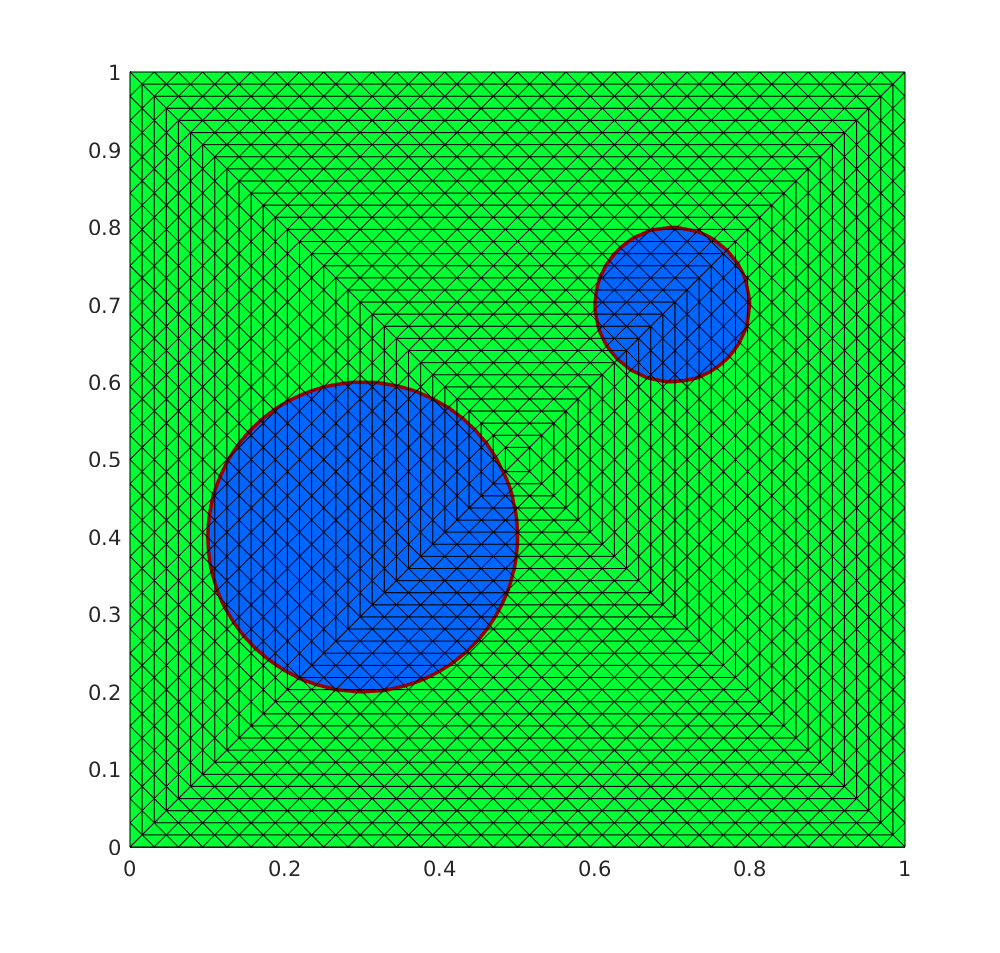}
		\subcaption{2113 nodes}
	\end{subfigure}\hfill
	\begin{subfigure}[t]{0.19\textwidth}
		\includegraphics[width=\textwidth]{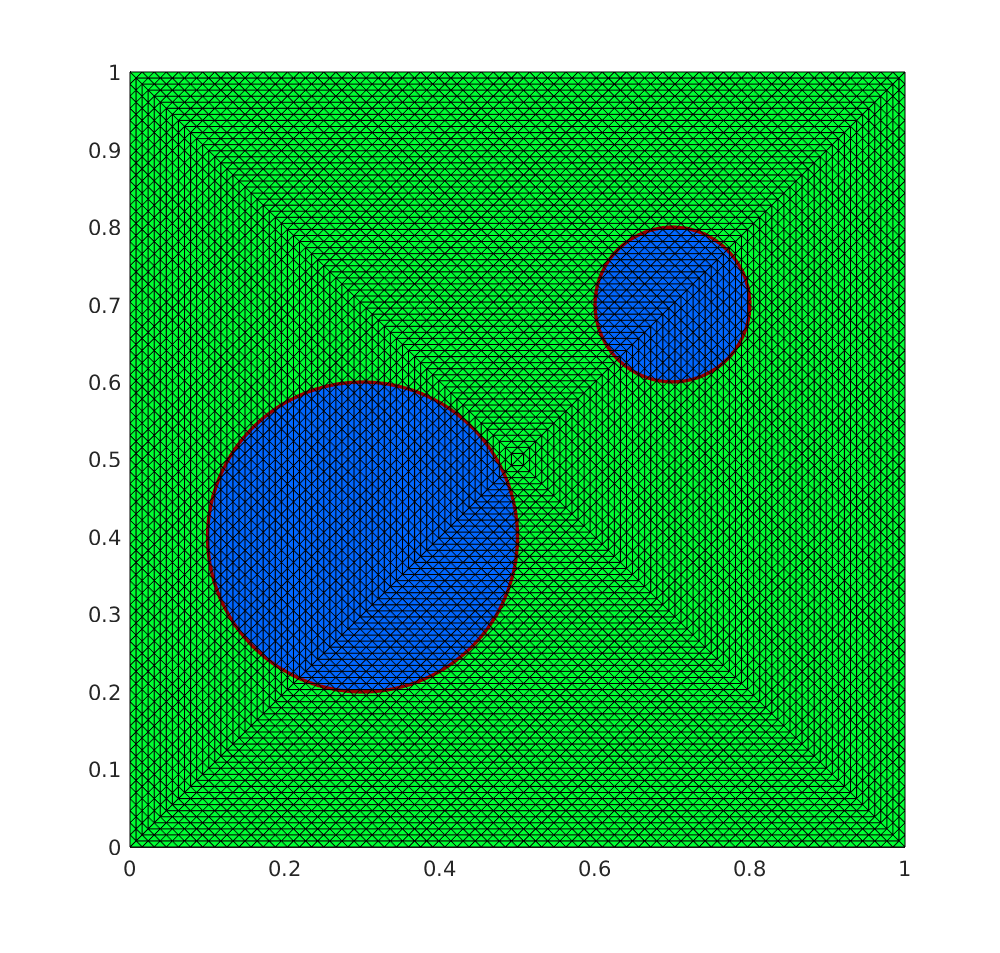}
		\subcaption{8321 nodes}
	\end{subfigure}\hfill
	\begin{subfigure}[t]{0.19\textwidth}
		\includegraphics[width=\textwidth]{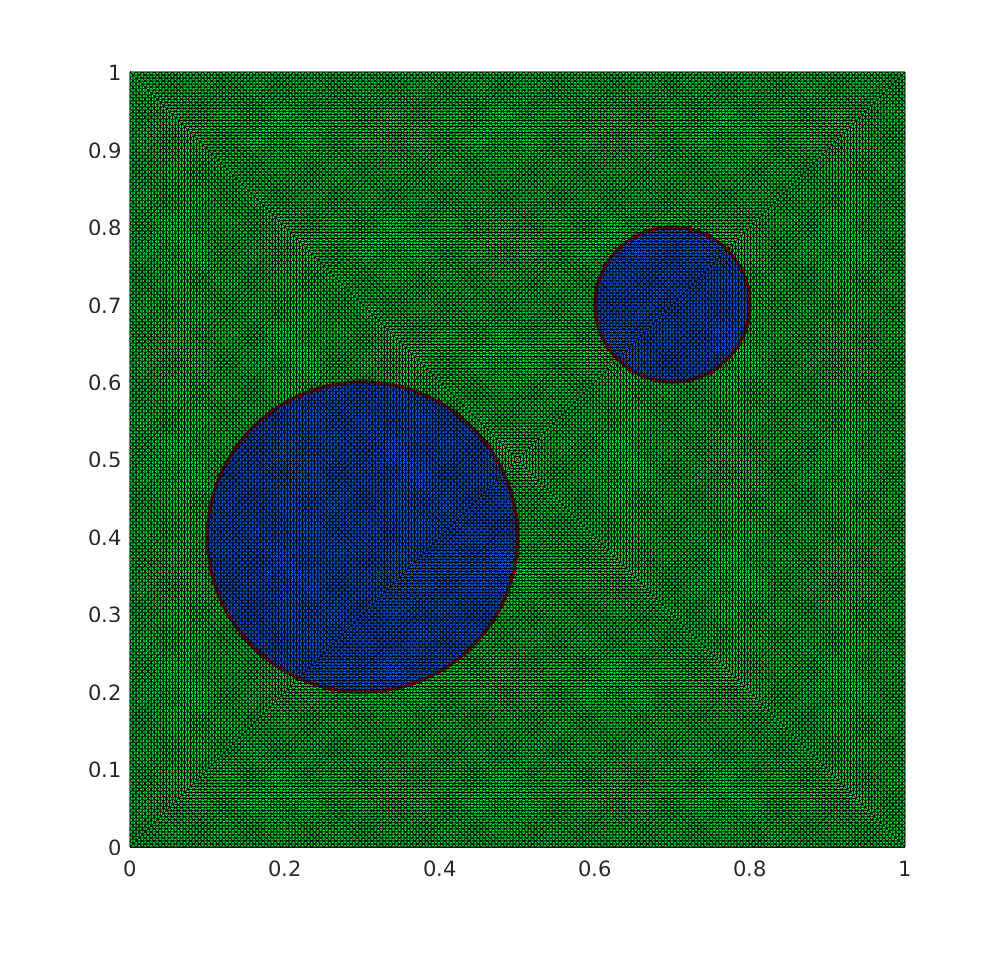}
		\subcaption{33025 nodes}
	\end{subfigure}\hfill
	\caption{The different meshes and corresponding sought shapes used in the numerical experiments.}
	\label{fig:holes-meshes}
\end{figure}
\subsubsection{Optimization algorithm}\label{sec::algorithm}
\newcommand{\iter}{i}
The optimization algorithm we use to solve the problem introduced in Section \ref{sec::numericalProblemSetting} is inspired by \cite{amstutz2006new}. 

\begin{definition} \label{def_locOpti}
    We say a level set function $\phi \in S_h^1(D)$ is locally optimal for the problem described by $\mathcal J$ if
    \begin{align}
        \begin{cases}
            d \mathcal J(\phi)(\pos_k) \geq 0 & \text{for } \pos_k \in \setTminus \cup \setTplus,\\
            d \mathcal J(\phi)(\pos_k) = 0 & \text{for } \pos_k \in \setS.
        \end{cases}
    \end{align}
\end{definition}

We introduce the generalized numerical \topshapeDerivative derivative $G_\levelset \in S_h^1(D)$ with
\begin{equation} \label{eq_defGphi}
 G_\levelset(\pos_k) = \begin{cases}
 -\min(d\redPhiObjectiveFunction(\levelset)(\pos_k),0) \quad &\text{for } \pos_k \in \setTminus, \\[8pt]
 \min(d\redPhiObjectiveFunction(\levelset)(\pos_k),0) \quad &\text{for } \pos_k \in \setTplus, \\[8pt]
 -d\redPhiObjectiveFunction(\levelset)(\pos_k) \quad &\text{for } \pos_k \in \setS.
 \end{cases}
\end{equation}
With this definition, we immediately get the following optimality condition:
\begin{lemma}
    Let $\phi \in S_h^1(D)$ and
	\begin{equation} \label{eq_optiCondition}
		 G_\levelset(\pos_k) = 0, \quad \text{for } k = 1,\dots ,M.
	\end{equation}
	Then $\phi$ is locally optimal in the sense of Definition \ref{def_locOpti}.
\end{lemma}
The update of the level-set function based on the information of the \topshapeDerivative derivative is done by spherical linear interpolation (see also \cite{amstutz2006new})
\begin{equation} \label{eq_update_slerp}
	\phi_{\iter+1} = \frac{1}{\sin(\theta_\iter)}\left(\sin((1-\kappa_\iter)\theta_\iter) \phi_\iter + \sin(\kappa_\iter\theta_\iter)\frac{G_{\phi_i}}{\|G_{\phi_i}\|_{L_{2}(D)}} \right),
\end{equation}
where $\theta_i = \mbox{arc cos}(( \phi_i, G_{\phi_i})_{L^2(D)})$ is the angle between the given level set function $\phi_i$ and the sensitivity $G_{\phi_i}$ in an $L^2(D)$-sense. Here, $\kappa_\iter \in (0,1)$ is a line search parameter which is adapted such that a decrease in the objective function is achieved. Note that, by construction, the update \eqref{eq_update_slerp} preserves the $L^2(D)$-norm, $\|\phi_{\iter+1}\|_{L^2(D)} = \|\phi_{\iter}\|_{L^2(D)}$. As in \cite{amstutz2006new,gangl2020multi}, we can also show that $\phi$ is evolving along a descent direction:
\begin{lemma}
    Let $\phi_\iter, \phi_{\iter+1} \in S_h^1(D)$ two subsequent iterates related by \eqref{eq_update_slerp}. Then we have for $\pos_k \in \setTminus(\phi_i) \cup \setTplus(\phi_i)$
    \begin{align}
        \phi_{\iter}(\pos_k) > 0 > \phi_{\iter+1}(\pos_k) \Longrightarrow d \mathcal J(\phi_\iter)(\pos_k) <0  \label{eq_descentTplus}\\
        \phi_{\iter}(\pos_k) < 0 < \phi_{\iter+1}(\pos_k) \Longrightarrow d \mathcal J(\phi_\iter)(\pos_k) <0 \label{eq_descentTminus}
    \end{align}
\end{lemma}
\begin{proof}
    Let $\pos_k \in \setTplus(\phi_i)$, i.e. $\phi_i(\pos_k) > 0$ and assume that $\phi_{\iter+1}(\pos_k) < 0$. Since $\mbox{sin}(\theta)>0$ and $\mbox{sin}(s\theta)>0$ for all $\theta \in (0, \pi)$ and $s\in (0,1)$, it follows from \eqref{eq_update_slerp} that $G_{\phi_i}(x_k) < 0$ and thus, by \eqref{eq_defGphi}, $d\mathcal J(\phi_i)(\pos_k) < 0$ as claimed in \eqref{eq_descentTplus}. An analogous argument yields \eqref{eq_descentTminus}.
\end{proof}

We can also show that $G_\phi$ constitutes a descent direction for $\pos_k \in \setS$.
\begin{lemma}
    Let $\phi \in S^1_h(D)$ and suppose that
    \begin{align}   \label{eq_limsubsuper}
    \underset{\eps \searrow 0}{\mbox{lim }} \frac{\mathcal J(S_{k,\eps} \phi)-\mathcal J(\phi)}{|\Omega(S_{k,\eps} \phi) \symmDiff \Omega(\phi)|} = - \underset{\eps \nearrow 0}{\mbox{lim }} \frac{\mathcal J(S_{k,\eps} \phi)-\mathcal J(\phi)}{|\Omega(S_{k,\eps} \phi) \symmDiff \Omega(\phi)|}.
    \end{align}
    Let $\pos_k \in \setS(\phi)$ be fixed and let $\phi^{\kappa}$ be the level set function according to \eqref{eq_update_slerp} with line search parameter $\kappa \in (0,1)$ that is updated only in $\pos_k$, i.e.,
        $\phi^{\kappa} = a(\kappa) \phi + b(\kappa) G_{\phi} \varphi_k$
    with $a(\kappa) = \sin((1-\kappa)\theta) / \sin(\theta)$ and $b(\kappa) = \sin(\kappa \theta) / \left( \sin(\theta) \|G_{\phi}\|_{L_{2}(D)} \right)$. Moreover assume that $|d \mathcal J(\phi)(\pos_k)|>0$. Then there exists $\overline \kappa \in (0,1)$ such that for all $\kappa \in (0,\overline \kappa)$
    \begin{align*}
        \mathcal J(\phi^\kappa) < \mathcal J(\phi).
    \end{align*}
\end{lemma}

\begin{proof}
    Suppose that $0 > d \mathcal J(\phi)(\pos_k)$. Then it follows from \eqref{eq::topDerivative} that $\mathcal J(\phi + \eps \varphi_k) < \mathcal J(\phi)$ for $\eps >0$ small enough.
    Thus, since $a(\kappa), b(\kappa)>0$ for $\theta \in (0, \pi)$ and $\kappa \in (0,1)$ and since $\mathcal J(\phi^\kappa) = \mathcal J(\frac{1}{a(\kappa)} \phi^\kappa)$, it follows
    \begin{align*}
        \mathcal J(\phi^\kappa) = \mathcal J(\phi + b(\kappa)/a(\kappa) G_{\phi} \varphi_k) = \mathcal J(\phi_i - b(\kappa)/a(\kappa) d \mathcal J(\phi_i)(\pos_k) \varphi_k) < \mathcal J(\phi_i)
    \end{align*}
    for $\kappa>0$ small enough since $-b(\kappa)/a(\kappa) d\mathcal J(\phi_i)(\pos_k)>0$ and $b(\kappa)/a(\kappa) \rightarrow 0$ as $\kappa \searrow 0$.
    On the other hand, if $0 < d \mathcal J(\phi_i)(\pos_k)$, it follows from \eqref{eq::topDerivative} and \eqref{eq_limsubsuper} that $\mathcal J(\phi + \eps \varphi_k) < \mathcal J(\phi)$ for $\eps <0$ small enough and further for $\kappa$ small enough
    \begin{align*}
        \mathcal J(\phi^\kappa)  = \mathcal J(\phi - b(\kappa)/a(\kappa) d \mathcal J(\phi)(\pos_k) \varphi_k) < \mathcal J(\phi).
    \end{align*}
\end{proof}

\begin{remark}
    In the continuous setting, the property corresponding to \eqref{eq_limsubsuper} is fulfilled for smooth domains which can be seen as follows. Let $\Omega_t^V = (\textrm{id}+t V)(\Omega)$ and note that $\Omega_{-t}^V = \Omega_t^{-V}$. Then, by Lemma \ref{LEM_SYMDIF},
    \begin{align*}
        \underset{s \nearrow 0}{\mbox{lim }} \frac{|\Omega_s^V \symmDiff \Omega|}{s}
     = -\underset{t \searrow 0}{\mbox{lim }} \frac{|\Omega_{t}^{-V} \symmDiff \Omega|}{t}
     = - \int_{\partial \Omega} |V \cdot n| \; \mbox d S_x = -\underset{s \searrow 0}{\mbox{lim }} \frac{|\Omega_{s}^{V} \symmDiff \Omega|}{s}
    \end{align*}
    and, assuming differentiability of $s \mapsto \mathfrak g(\Omega_s^V)$,
    \begin{align*}
        \underset{s \searrow 0}{\mbox{lim }} \frac{\mathfrak g(\Omega_s^V) - \mathfrak g(\Omega)}{|\Omega_s^V \symmDiff \Omega|} =
        \frac{\underset{s \rightarrow 0}{\mbox{lim }} \left( \mathfrak g(\Omega_s^V) - \mathfrak g(\Omega)\right) / s}{\underset{s \searrow 0}{\mbox{lim }} |\Omega_s^V \symmDiff \Omega| / s} = -\underset{s \nearrow 0}{\mbox{lim }} \frac{\mathfrak g(\Omega_s^V) - \mathfrak g(\Omega)}{|\Omega_s^V \symmDiff \Omega|}.
    \end{align*}

    In the discrete case, however, there may occur situations where the limits in \eqref{eq_limsubsuper} do not coincide. This can be the case in particular in situations where $\phi(\pos_k) = 0$. We remark that this issue seemed not to cause problems in our numerical experiments.

\end{remark}

\begin{remark}
    In practice it turned out to be advantageous to include a smoothing step of the level set function. Thus, we chose the following update strategy: We first set
    \begin{equation*}
        \psi = \frac{1}{\sin(\theta_\iter)}\left(\sin((1-\kappa_\iter)\theta_\iter) \phi_\iter + \sin(\kappa_\iter\theta_\iter)\frac{G_{\phi_i}}{\|G_{\phi_i}\|_{L_{2}(D)}} \right),
    \end{equation*}
    with the same notation as above before smoothing the level set function in $\setTminus(\psi) \cup \setTplus(\psi)$ by
    \begin{equation}
	\hat\psi(\pos_k) = \begin{cases}
		\frac{\sum_{i\in R_{\pos_k}} \psi(\pos_i) }{|R_{\pos_k}|} \quad &\text{for } \pos_k \in \setTminus(\psi)\cup\setTplus(\psi), \\[8pt]
		\psi(\pos_k) \quad &\text{for } \pos_k \in \setS.
	\end{cases}
    \end{equation}
    Finally, the level-set function is normalized and the next iterate is given by
    \begin{equation}
        \levelset_{i+1} = \frac{ \hat\psi}{\| \hat\psi\|_{L_{2}(D)}}.
    \end{equation}
\end{remark}
\black

\subsubsection{Numerical results}
As an initial design for the optimisation, we take the empty set, $\Omega = \emptyset$. This is realized by choosing $\levelset_0 = 1/\|1\|_{L_2(D)}$ as the initial level set function. We use the algorithm outlined in Section \ref{sec::algorithm} 
to update this level set function. We terminated the algorithm after the fixed number of 800 iterations. The final as well as some intermediate configurations are illustrated in Figures  \ref{fig:holes-pics2-3}-\ref{fig:holes-pics2-7} for the five different levels of discretization.
\newcommand{\myincludepic}[2]{\begin{subfigure}{0.49\linewidth}\centering\includegraphics[width=0.7\linewidth,trim=0cm 0cm 0cm 0cm,clip]{holes-pics2-1e+20_1_1-#1}\caption{#2 iteration}
\end{subfigure}}
\newcommand{\picE}{100}
\newcommand{\picText}[1]{\centering
	\includegraphics[width=0.5\linewidth]{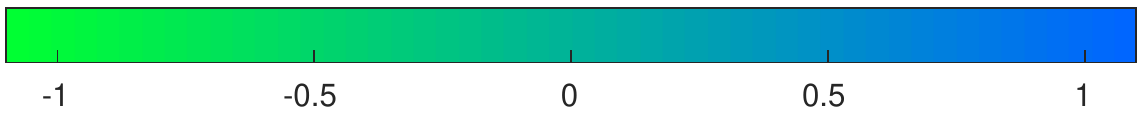}
	\caption{Evolution of the level-set function for the #1 nodes mesh}}
\begin{figure}
	\myincludepic{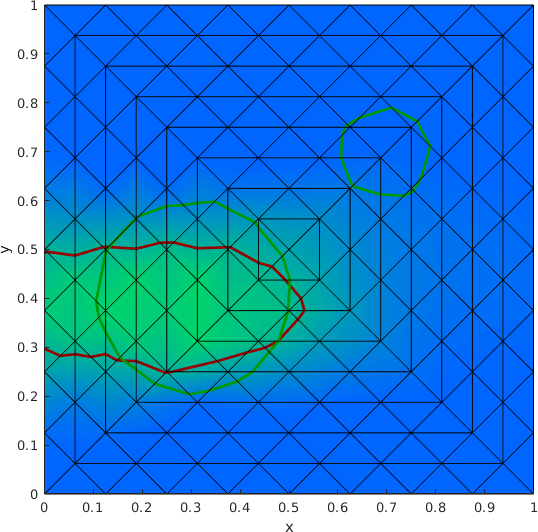}{1}
	\myincludepic{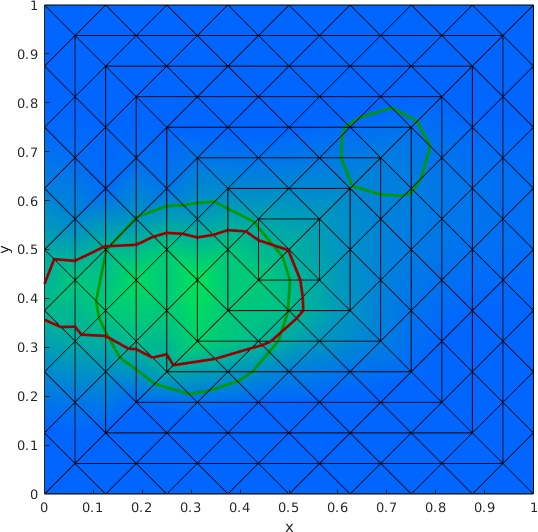}{2}
	
	\myincludepic{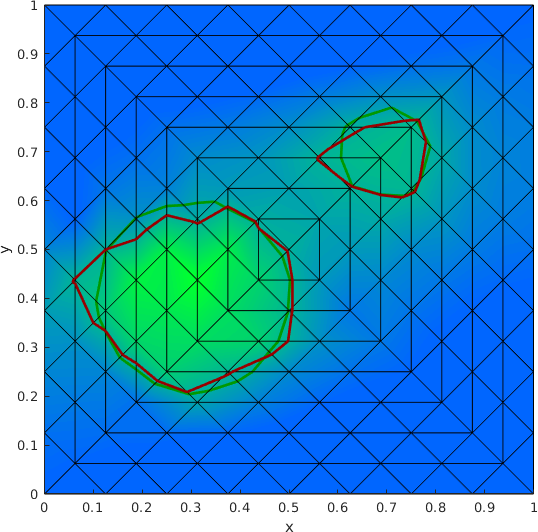}{10}
	\myincludepic{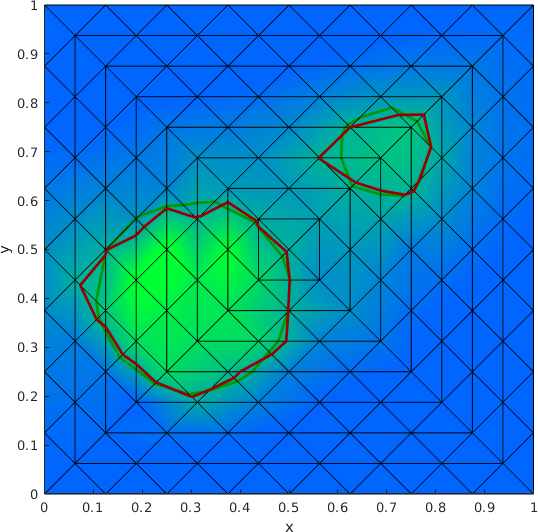}{20}
	
	\myincludepic{3-\picE}{\picE}
	\myincludepic{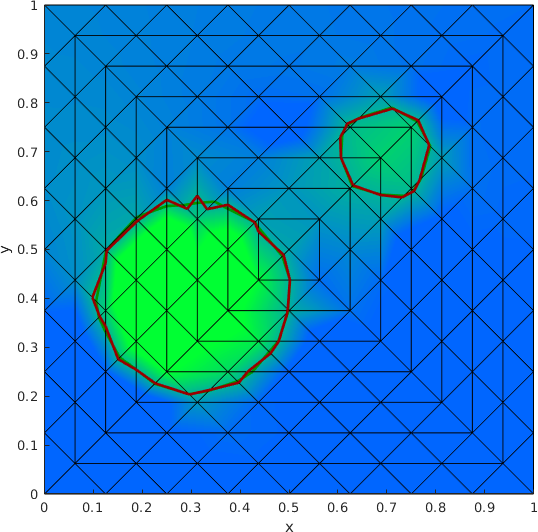}{800}

	\picText{145}
	\label{fig:holes-pics2-3}
\end{figure}
\begin{figure}
	\myincludepic{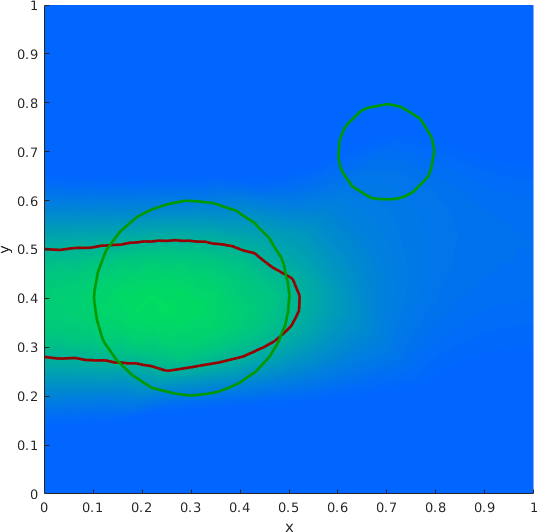}{1}
	\myincludepic{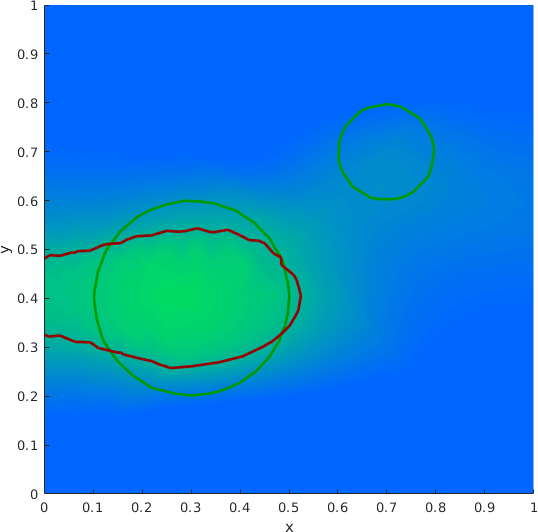}{2}
	
	\myincludepic{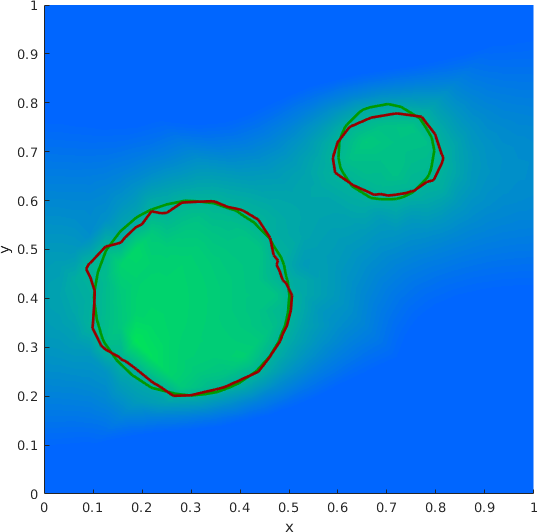}{10}
	\myincludepic{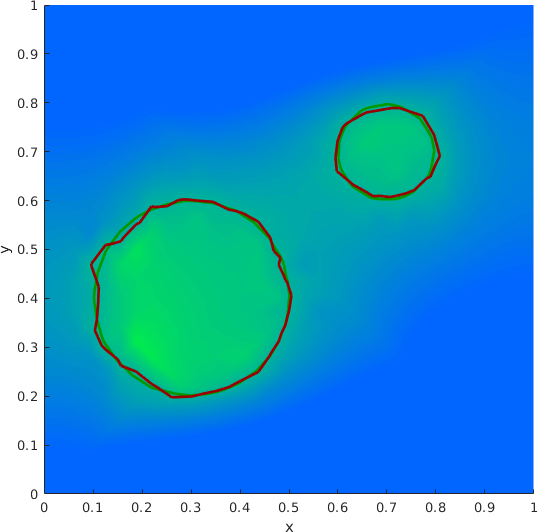}{20}
	
	\myincludepic{4-\picE}{\picE}
	\myincludepic{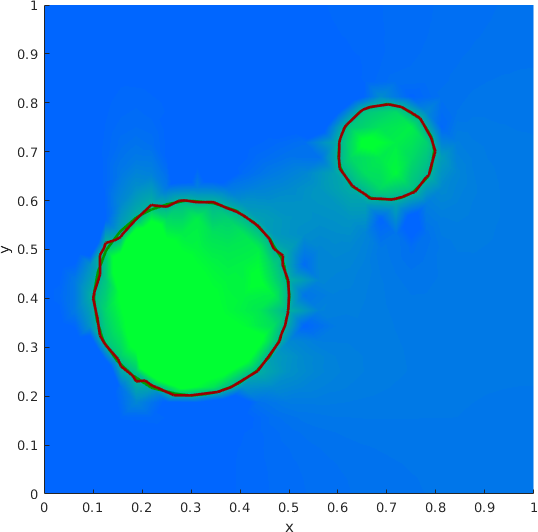}{800}
	\picText{545}
	\label{fig:holes-pics2-4}
\end{figure}
\begin{figure}
	\myincludepic{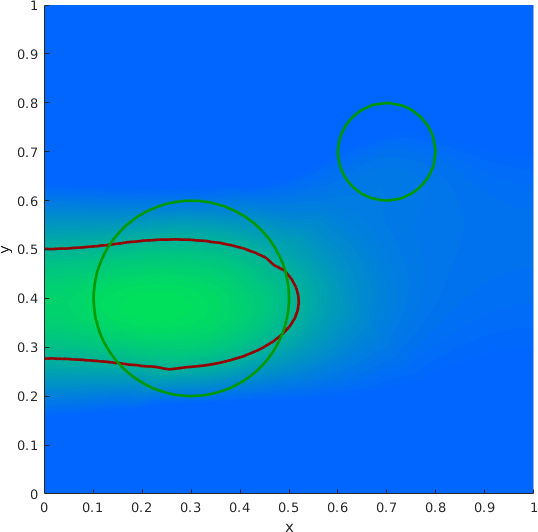}{1}
	\myincludepic{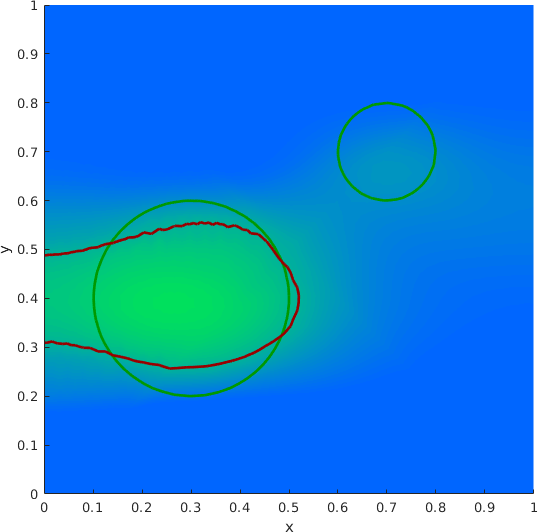}{2}
	
	\myincludepic{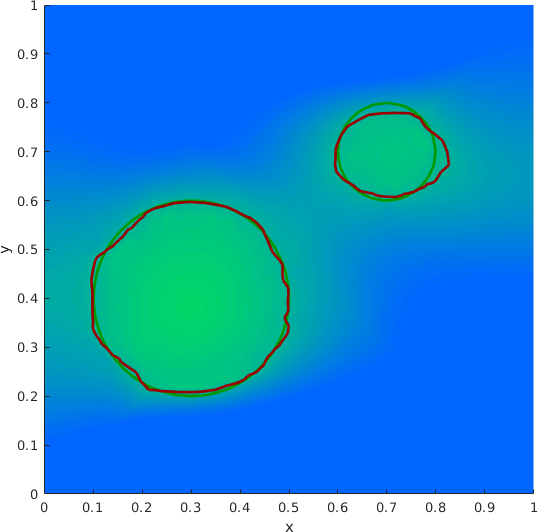}{10}
	\myincludepic{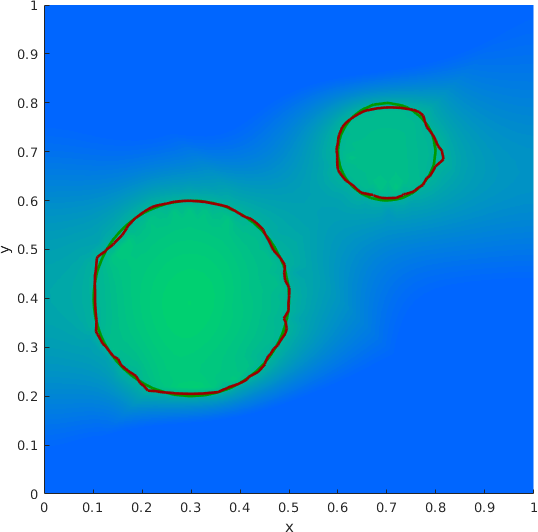}{20}
	
	\myincludepic{5-\picE}{\picE}
	\myincludepic{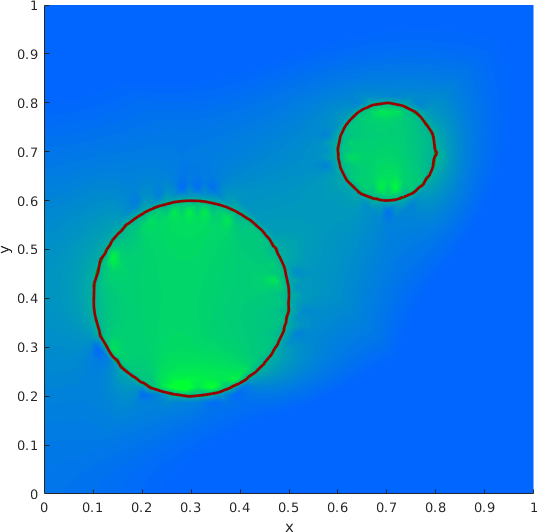}{800}
	\picText{2113}
	\label{fig:holes-pics2-5}
\end{figure}
\begin{figure}
	\myincludepic{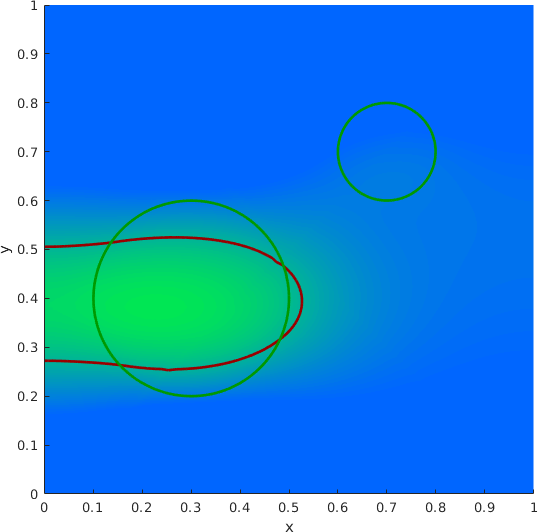}{1}
	\myincludepic{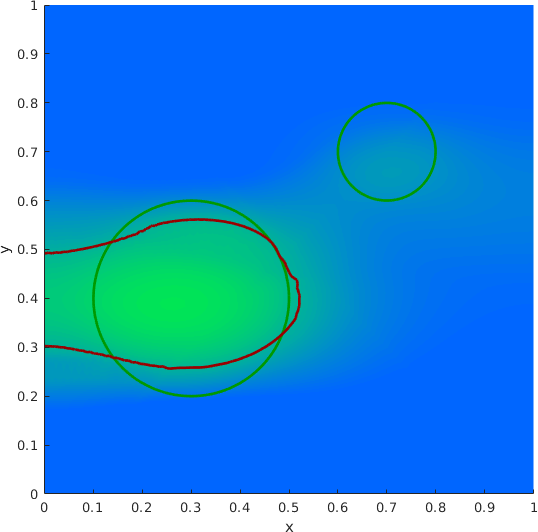}{2}
	
	\myincludepic{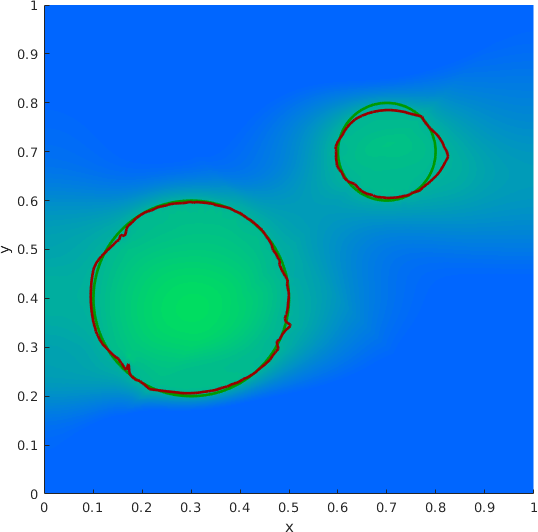}{10}
	\myincludepic{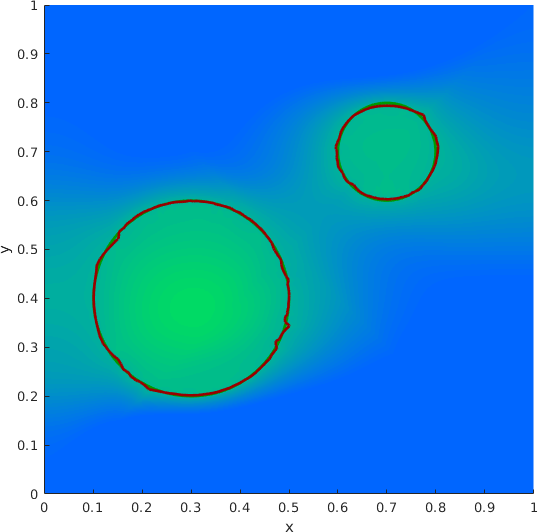}{20}
	
	\myincludepic{6-\picE}{\picE}
	\myincludepic{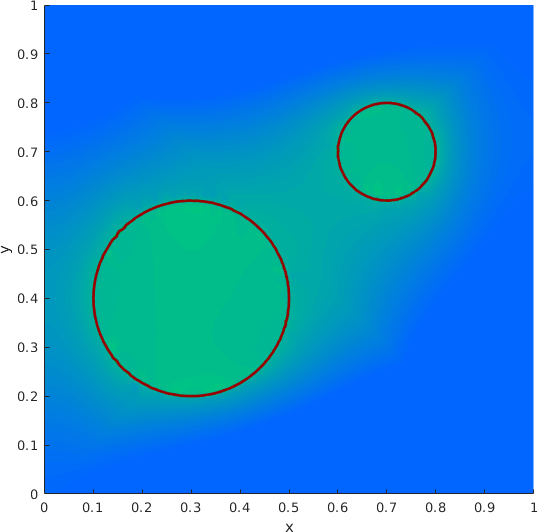}{800}
	
	\picText{8321}
	\label{fig:holes-pics2-6}
\end{figure}
\begin{figure}
	\myincludepic{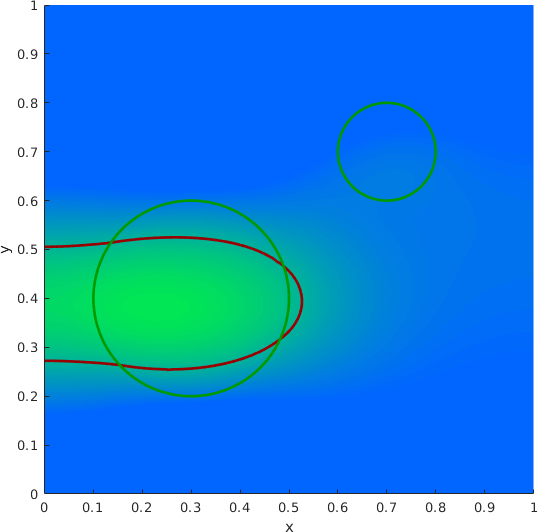}{1}
	\myincludepic{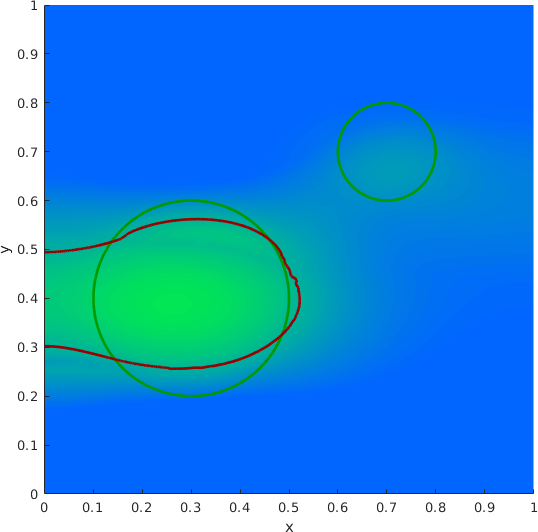}{2}

	\myincludepic{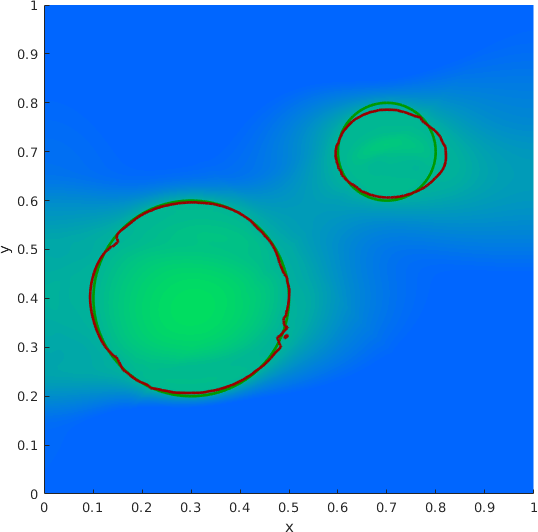}{10}
	\myincludepic{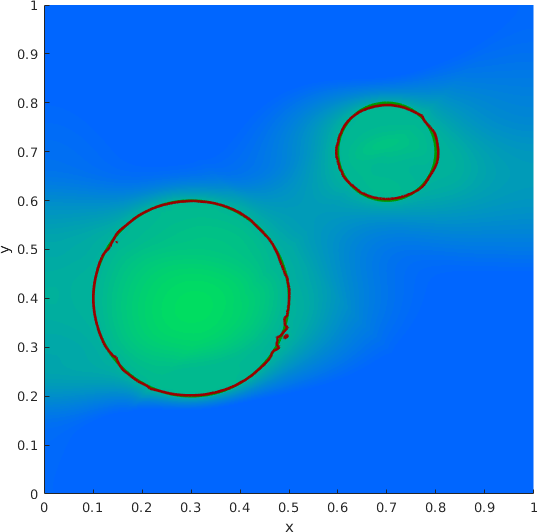}{20}

	\myincludepic{7-\picE}{\picE}
	\myincludepic{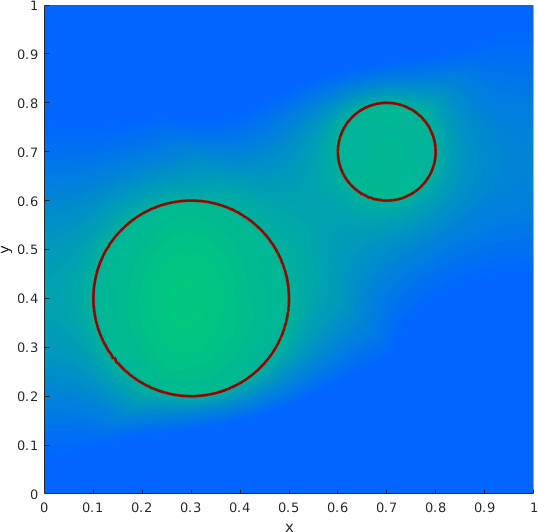}{800}

	\picText{33025}
	\label{fig:holes-pics2-7}
\end{figure}
\begin{figure}
	\centering
	\begin{tabular}{cc}
		\includegraphics[width=0.49\linewidth]{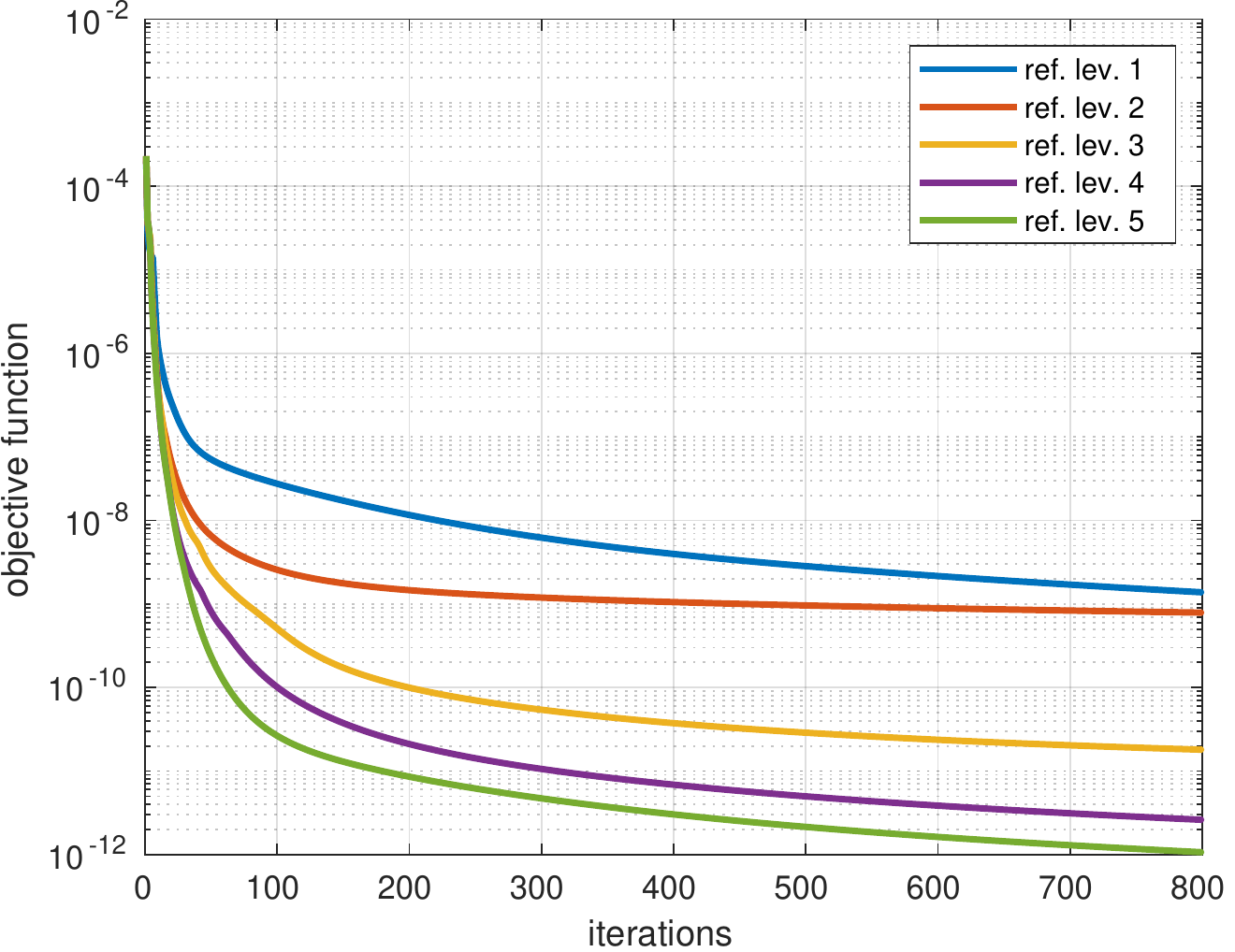}
		&
		\includegraphics[width=0.49\linewidth]{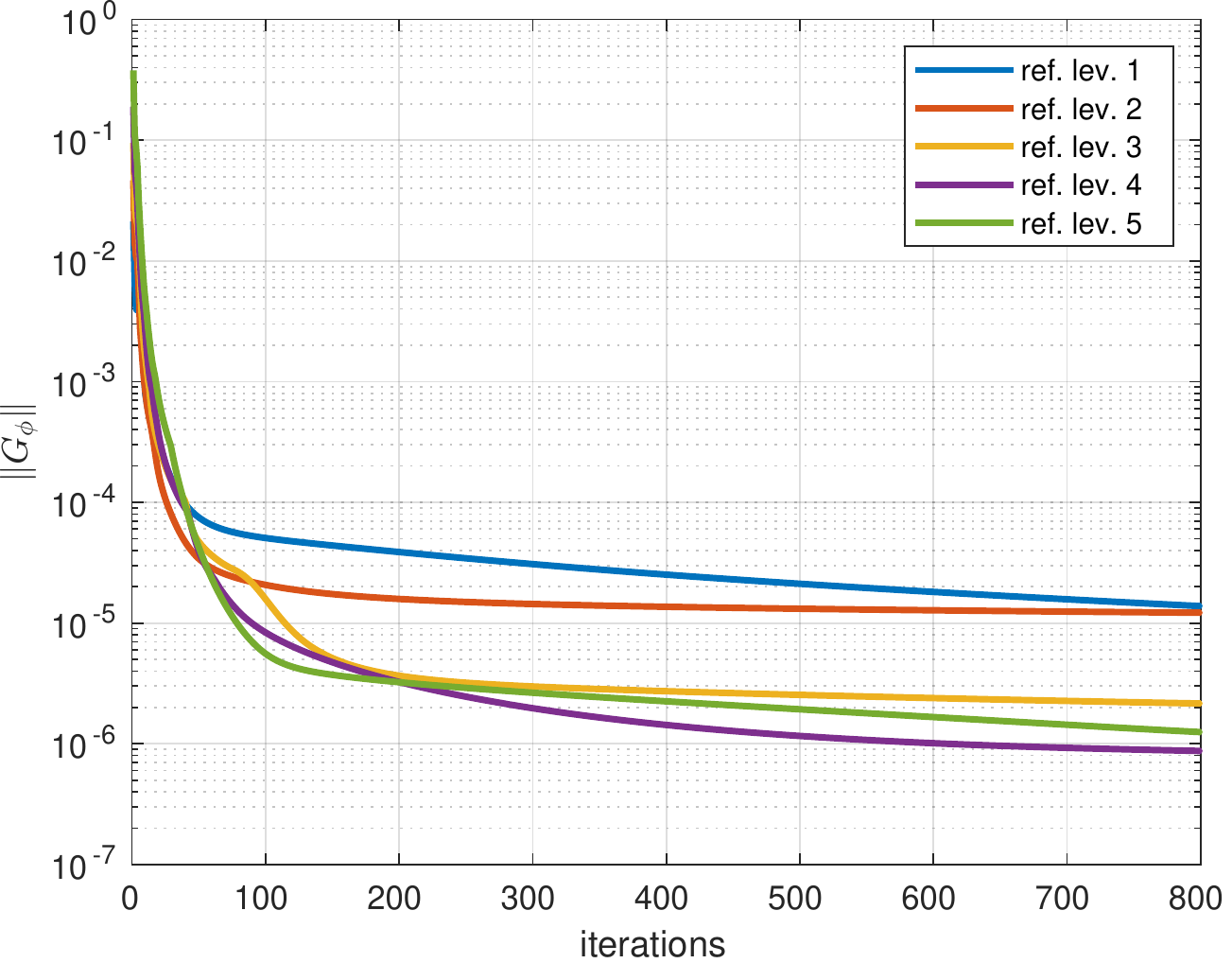}
		\\
		(a) & (b) 
	\end{tabular}
	\caption{Evolution of the objective function (a) and the norm of the \topshapeDerivative derivative (b) in course of the optimization.}
	\label{fig:holeshistoryobjective-1}
\end{figure}

We observe that in all cases the two circles are recovered with high accuracy. In Figure \ref{fig:holeshistoryobjective-1} the evolution of the objective function as well as of the norm of the generalized numerical \topshapeDerivative derivative is plotted. We observe that objective function decreases fast and after 800 iterations a reduction by a factor of approximately $10^{-5}-10^{-8}$ could be achieved. Moreover, we observe that the norm of the \topshapeDerivative derivative decreases continuously, more and more approaching the optimality condition \eqref{eq_optiCondition}.

\section{Conclusions} \label{sec::conclusion}
In this work we presented a new sensitivity concept, called the \topshapeDerivative derivative which is based on a level set representation of a domain. This approach allows for a unified sensitivity analysis for shape and topological perturbations, which we carried out for a discretized PDE-constrained design optimization problem in two space dimensions.
For the discretization we used a standard first order finite element method which does not account for the interface position in the approximation space. Therefore, kinks in the solution of the state and adjoint equations at material interfaces are not resolved.
Comparing the computed sensitivities of the discretized problem with the discretization of the continuous topological and shape derivatives, we saw that certain terms did not appear in the former approach. These lack of these terms can be traced back to the inability of the chosen discretization method to represent such kinks.
Thus, it would be interesting to study discretization methods which do account for these kinks, \eg XFEM or CutFEM, and perform the sensitivity analysis in these settings in future work. Furthermore, the extension to higher space dimensions, higher polynomial degree and other PDE constraints such as elasticity would be further interesting research directions.



%
\section*{Conflict of interest}
The authors declare that they have no conflict of interest.

\bibliographystyle{spmpsci}      
\bibliography{literature.bib}   

\begin{thebibliography}{10}
\providecommand{\url}[1]{{#1}}
\providecommand{\urlprefix}{URL }
\expandafter\ifx\csname urlstyle\endcsname\relax
  \providecommand{\doi}[1]{DOI~\discretionary{}{}{}#1}\else
  \providecommand{\doi}{DOI~\discretionary{}{}{}\begingroup
  \urlstyle{rm}\Url}\fi

\bibitem{AllaireJouve:2006a}
Allaire, G., Jouve, F.: Coupling the level set method and the topological
  gradient in structural optimization.
\newblock In: M.P. Bends{\o}e, N.~Olhoff, O.~Sigmund (eds.) IUTAM Symposium on
  Topological Design Optimization of Structures, Machines and Materials, pp.
  3--12. Springer Netherlands, Dordrecht (2006)

\bibitem{allaire2004structural}
Allaire, G., Jouve, F., Toader, A.M.: Structural optimization using sensitivity
  analysis and a level-set method.
\newblock Journal of Computational Physics \textbf{194}(1), 363--393 (2004)

\bibitem{Amstutz_2006aa}
Amstutz, S.: Sensitivity analysis with respect to a local perturbation of the
  material property.
\newblock Asymptotic Analysis \textbf{49}(1,2), 87--108 (2006)

\bibitem{amstutz2006new}
Amstutz, S., Andr{\"a}, H.: A new algorithm for topology optimization using a
  level-set method.
\newblock Journal of Computational Physics \textbf{216}(2), 573--588 (2006)

\bibitem{amstutz2018consistent}
Amstutz, S., Dapogny, C., Ferrer, {\`A}.: A consistent relaxation of optimal
  design problems for coupling shape and topological derivatives.
\newblock Numerische Mathematik \textbf{140}(1), 35--94 (2018)

\bibitem{AmstutzGangl2019}
Amstutz, S., Gangl, P.: Topological derivative for the nonlinear magnetostatic
  problem.
\newblock Electron. Trans. Numer. Anal. \textbf{51}, 169--218 (2019)

\bibitem{bendsoe2003topology}
Bendsoe, M., Sigmund, O.: Topology Optimization: Theory, Methods, and
  Applications.
\newblock Springer Berlin Heidelberg (2003)

\bibitem{bernland2018acoustic}
Bernland, A., Wadbro, E., Berggren, M.: Acoustic shape optimization using cut
  finite elements.
\newblock International Journal for Numerical Methods in Engineering
  \textbf{113}(3), 432--449 (2018)

\bibitem{burger2004incorporating}
Burger, M., Hackl, B., Ring, W.: Incorporating topological derivatives into
  level set methods.
\newblock Journal of Computational Physics \textbf{194}(1), 344--362 (2004)

\bibitem{BurmanClausHansboLarsonMassing2014}
Burman, E., Claus, S., Hansbo, P., Larson, M.G., Massing, A.: Cutfem:
  Discretizing geometry and partial differential equations.
\newblock International Journal for Numerical Methods in Engineering
  \textbf{104}(7), 472--501 (2015)

\bibitem{delfour2018topological}
Delfour, M.C.: Topological derivative: a semidifferential via the {M}inkowski
  content.
\newblock J. Convex Anal. \textbf{25}(3), 957--982 (2018)

\bibitem{Delfour_engcomp_2022}
Delfour, M.C.: Topological derivatives via one-sided derivative of parametrized
  minima and minimax.
\newblock Engineering Computations \textbf{39}(1), 34–59 (2022)

\bibitem{DZ2}
Delfour, M.C., Zol{\'e}sio, J.P.: Shapes and geometries: metrics, analysis,
  differential calculus, and optimization, \emph{Advances in Design and
  Control}, vol.~22, second edn.
\newblock Society for Industrial and Applied Mathematics (SIAM), Philadelphia,
  PA (2011)

\bibitem{fike2011development}
Fike, J., Alonso, J.: The development of hyper-dual numbers for exact
  second-derivative calculations.
\newblock In: 49th AIAA Aerospace Sciences Meeting including the New Horizons
  Forum and Aerospace Exposition, p. 886 (2011)

\bibitem{gangl2020multi}
Gangl, P.: A multi-material topology optimization algorithm based on the
  topological derivative.
\newblock Computer Methods in Applied Mechanics and Engineering \textbf{366},
  113090 (2020)

\bibitem{GanglLangerLaurainMeftahiSturm2015}
Gangl, P., Langer, U., Laurain, A., Meftahi, H., Sturm, K.: Shape optimization
  of an electric motor subject to nonlinear magnetostatics.
\newblock SIAM Journal on Scientific Computing \textbf{37}(6), B1002–B1025
  (2015)

\bibitem{simplified}
{Gangl, P.}, {Sturm, K.}: A simplified derivation technique of topological
  derivatives for quasi-linear transmission problems.
\newblock ESAIM: COCV \textbf{26}, 106 (2020)

\bibitem{HaubnerSiebenbornUlbrich2021}
Haubner, J., Siebenborn, M., Ulbrich, M.: A continuous perspective on shape
  optimization via domain transformations.
\newblock SIAM Journal on Scientific Computing \textbf{43}(3), A1997–A2018
  (2021)

\bibitem{Hagg2018}
Hägg, L., Wadbro, E.: On minimum length scale control in density based
  topology optimization.
\newblock Structural and Multidisciplinary Optimization \textbf{58}(3),
  1015–1032 (2018)

\bibitem{laurain2018analyzing}
Laurain, A.: Analyzing smooth and singular domain perturbations in level set
  methods.
\newblock SIAM Journal on Mathematical Analysis \textbf{50}(4), 4327--4370
  (2018)

\bibitem{LaurainSturm2016}
{Laurain, A.}, {Sturm, K.}: Distributed shape derivative via averaged adjoint
  method and applications.
\newblock ESAIM: M2AN \textbf{50}(4), 1241--1267 (2016)

\bibitem{martins2003complex}
Martins, J.R., Sturdza, P., Alonso, J.J.: The complex-step derivative
  approximation.
\newblock ACM Transactions on Mathematical Software (TOMS) \textbf{29}(3),
  245--262 (2003)

\bibitem{MoesDolbowBelytschko1999}
Mo\"es, N., Dolbow, J., Belytschko, T.: A finite element method for crack
  growth without remeshing.
\newblock Int. J. Numer. Meth. Engng. \textbf{46}(1), 131--150 (1999)

\bibitem{b_NovoSoko}
Novotny, A., Sokolowski, J.: Topological derivatives in shape optimization.
\newblock Interaction of Mechanics and Mathematics. Springer, Heidelberg (2013)

\bibitem{sigmund2013topology}
Sigmund, O., Maute, K.: Topology optimization approaches.
\newblock Structural and Multidisciplinary Optimization \textbf{48}(6),
  1031--1055 (2013)

\bibitem{sokolowski1999topological}
Sokolowski, J., Zochowski, A.: On the topological derivative in shape
  optimization.
\newblock SIAM Journal on Control and Optimization \textbf{37}(4), 1251--1272
  (1999)

\bibitem{Sturm2015}
Sturm, K.: Minimax {L}agrangian approach to the differentiability of nonlinear
  {PDE} constrained shape functions without saddle point assumption.
\newblock SIAM Journal on Control and Optimization \textbf{53}(4), 2017--2039
  (2015)

\end{thebibliography}
\appendix
\section{Appendix}

\subsection{Proof of Lemma \ref{LEM_SYMDIF}} \label{app::proofSymDif}
\begin{proof}
We investigate the limit 
    \begin{align*}
	\underset{t \searrow 0}{\mbox{lim}} \frac{1}{t} |\Omega_t \symmDiff \Omega| = &\underset{t \searrow 0}{\mbox{lim }} \frac{1}{t} \left( \int_{\Omega_t\setminus \Omega}dx + \int_{\Omega\setminus \Omega_t} dx\right). 
	\end{align*}
	Let $c:[0,1] \rightarrow \partial \Omega$ be a smooth parametrization of the boundary $\partial \Omega$ in counter-clockwise direction with smooth inverse and define $c_t:[0,1] \rightarrow \partial\Omega_t$,
	\begin{equation}
	c_t(s) := c(s) + t V(c(s)) = (\textrm{id} + t V)(c(s)).
	\end{equation}
	The derivative is given by
	\begin{equation}
	\dot \defC(s) = \frac{d}{ds}\defC(s) = \dot c(s) + t \, \partial V(c(s))  \dot c(s) = (I + t \, \partial V(c(s))) \dot c(s).
	\end{equation}
	\begin{figure}
		\centering
		\includegraphics[width=0.4\linewidth]{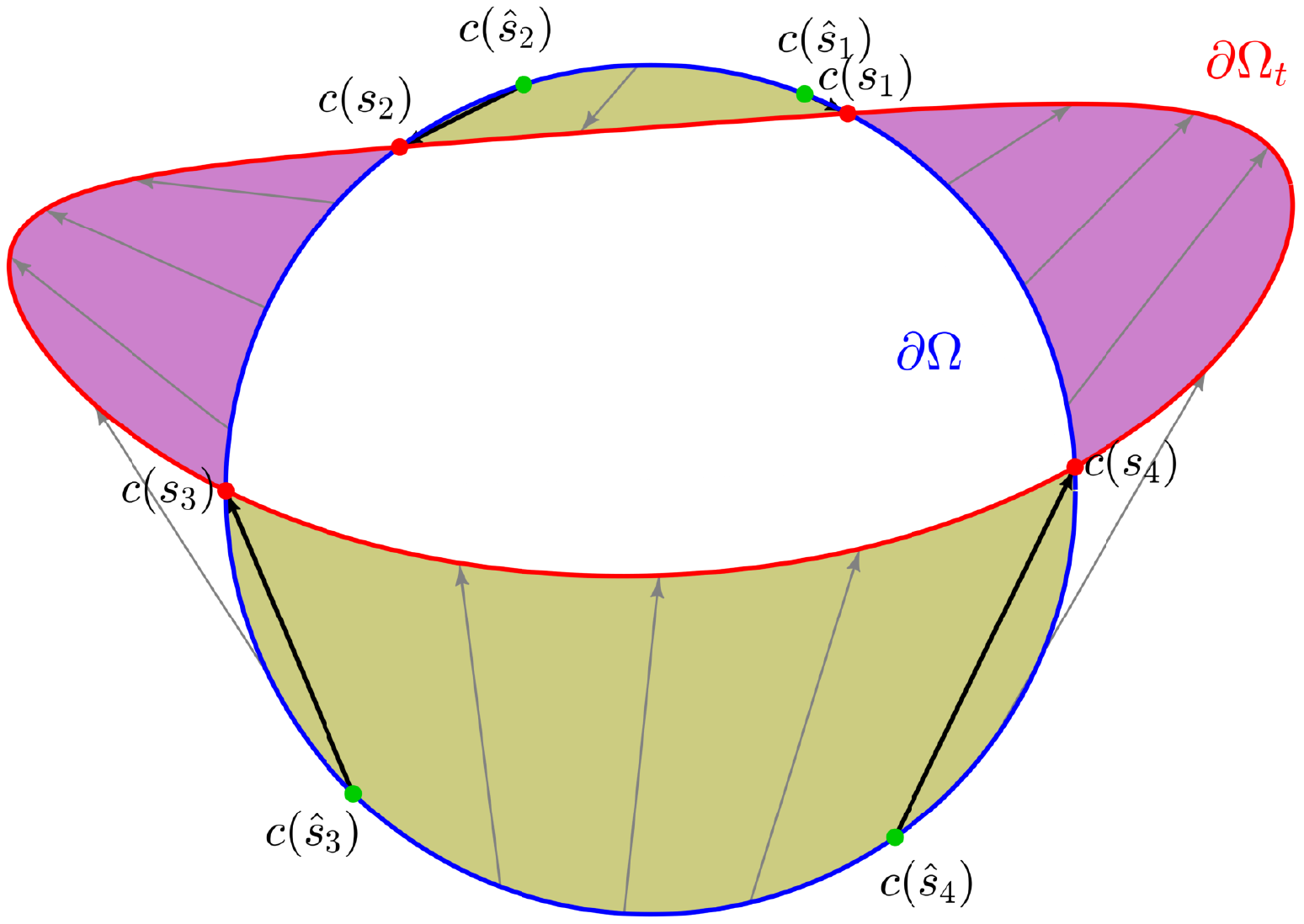}
		\caption{Illustration of the pairs $(s_i(t),\,\hat s_i(t))$. 
		 }
		\label{fig:picshapederivative}
	\end{figure}
	Let $\overline t >0$ sufficiently small such that, for all $t \in (0,\overline t)$, the number of intersection points between $\partial \Omega$ and $\partial \Omega_t$ is a fixed number $N$. To each intersection point we associate a pair of numbers $(s_i(t),\,\hat s_i(t))$ such that
	\begin{equation}\label{eq::defineIntersection}
	c(s_i(t)) = \defC(\hat s_i(t)) = c(\hat s_i(t)) + t V(c(\hat s_i(t))),
	\end{equation}
	see Figure \ref{fig:picshapederivative} for an illustration of the situation.
	The symmetric difference can now be written as $\Omega_t \symmDiff \Omega = \bigcup_{i=1}^{N} A_i(t)$, where $A_i(t)$ denotes the region between $\partial \Omega$ and $\partial \Omega_t$ bounded by the intersection points $c(s_i(t))$ and $c(s_{i+1}(t))$ (here, $A_N(t)$ is bounded by $c(s_N(t))$ and $c(s_1(t))$). More precisely, $A_i(t)$ is bounded by the two segments $\{c(s): s \in [s_{i}(t), s_{i+1}(t)]\}$ and $\{c_t(s): s \in [\hat s_{i}(t), \hat s_{i+1}(t)]\}$. Now let $i \in \{1, \dots, N\}$ fixed. The volume of $A_i(t)$ can be expressed by the divergence theorem as 
	\begin{align}
	|A_i(t)| &= \int_{A_i(t)} \frac{1}{2}\mbox{div} \begin{pmatrix} x_1 \\ x_2 \end{pmatrix} \, \mbox dx = 
	\frac{1}{2}\int_{\partial A_i(t)} x \cdot n(x) \,\mbox dS_x \\
	&= \frac{1}{2} \left(\int_{s_i(t)}^{s_{i+1}(t)} c(s)^\top  R_i \frac{\dot c(s)}{|\dot c(s)|} |\dot c(s)|\,\mbox ds + \int_{\hat s_i(t)}^{\hat s_{i+1}(t)} \defC( s)^\top  (-R_i) \frac{\dot c_t( s)}{|\dot \defC(s)|} |\dot \defC(s)|\,\mbox d s \right)
	\end{align}
	with the rotation matrix
	\begin{equation}
	R_i = d_i  \begin{bmatrix}
	0 & 1\\
	-1 & 0
	\end{bmatrix}, \quad \mbox{ where } d_i = \begin{cases}
	1, & \mbox{if } A_i(t) \subset \Omega \setminus \Omega_t, \\
    -1, &\mbox{if } A_i(t) \subset \Omega_t \setminus \Omega,
    \end{cases}
	\end{equation}
	 such that $R_i\frac{\dot c(s)}{|\dot c(s)|}$ and $(-R_i)\frac{\dot c_t(s)}{|\dot c_t(s)|}$ are the unit normal vectors pointing out of $A_i(t)$ at $c(s)$ and $c_t(s)$, respectively. Note that $R_i^\top = - R_i$. For further use we note that
	\begin{align*}
	\defC(s)^\top  R_i \,\dot \defC(s) &=  [c(s) + t V(c(s))]^\top R_i\,(I + t \, \partial V(c(s)))\,\dot c(s) \\
	&= c(s)^\top R_i \, \dot c(s) + t\left(V(c(s))^\top R_i\, \dot c(s) + c(s)^\top R_i\cdot\partial V(c(s))\,\dot c(s) \right) + t^2 V(s)^\top R_i\,\partial V(c(s))\,\dot c(s).
	\end{align*}
	Furthermore, let
	\begin{equation}
	\bar s_i := \mylim s_i(t), 
	\end{equation}
	and we observe from \eqref{eq::defineIntersection} that $ \mylim \hat s_i(t) = \bar s_i$.
	Moreover, since we assumed the inverse of $c$ to be smooth (in particular Lipschitz continuous with constant $L$), we have that the limit
	\begin{align} \label{eq_lim_sisit}
        \mylim \frac{1}{t} |s_i(t) - \hat s_i(t)| \leq \mylim  \frac{L}{t} |c(s_i(t)) - c(\hat s_i(t))| = L |V(c(\bar s_i))|,  
	\end{align}
	exists. Here we used \eqref{eq::defineIntersection} and the continuity of $V$ and $c$.
	
	With the abbreviation $f_c(s) := c(s)^\top  R_i\, \dot c(s)$ we have
	\begin{align}
	|A_i(t)|&=\frac 1 2 \Bigg( \int_{\hat s_i(t)}^{s_{i}(t)} f_c(s) \,ds + \int_{s_i(t)}^{s_{i+1}(t)} f_c(s) \,ds +\int_{s_{i+1}(t)}^{\hat s_{i+1}(t)} f_c(s) \,ds- \int_{\hat s_i(t)}^{\hat s_{i+1}(t)} f_{\defC}(s) \,d s  \nonumber \\
	& \qquad - \int_{\hat s_i(t)}^{s_{i}(t)} f_c(s) \,ds - \int_{s_{i+1}(t)}^{\hat s_{i+1}(t)} f_c(s) \,ds \Bigg) \nonumber \\
	&= \frac 1 2\Bigg( \int_{\hat s_i(t)}^{\hat s_{i+1}(t)} \big[f_c(s)-f_{\defC}(s) \big] \,d s - \int_{\hat s_i(t)}^{s_{i}(t)} f_c(s) \,ds - \int_{s_{i+1}(t)}^{\hat s_{i+1}(t)} f_c(s) \,ds \Bigg).
	\end{align}
	Thus, 
	$\frac{1}{t}|A_i(t)| 	= \frac{1}{2}\big(  B_i(t) - C_i(t) + C_{i+1}(t)\big)$
	with
	\begin{align}
	B_i(t)&=-\int_{\hat s_i(t)}^{\hat s_{i+1}(t)} \big[ V(c(s))^\top R_i \, \dot c(s) + c(s)^\top R_i \,\partial V(c(s)) \,\dot c(s) \big] \,\mbox ds + \mathcal O(t), \\
	C_i(t)&=\frac{1}{t}\int_{\hat s_i(t)}^{s_i(t)} f_c(s) \, \mbox ds.
	\end{align}
	In order to compute $B_i(t)$ we note that
	\begin{align}
	\frac{d}{ds}(V(c(s))^\top R_i \, c(s)) &= V(c(s))^\top R_i \, \dot c(s) + (\partial V(c(s)) \dot c(s))^\top \, R_i c(s) \nonumber\\
	&= V(c(s))^\top R_i \, \dot c(s) - c(s)^\top R_i \, \partial V(c(s)) \, \dot c(s) \nonumber
	\end{align}
	
	Thus,
	\begin{align}
	B_i(t)&= \int_{\hat s_i(t)}^{\hat s_{i+1}(t)} \left[\frac{d}{ds}(V(c(s))^\top R_i \, c(s)) - 2V(c(s))^\top R_i\, \dot c(s) \right] \,ds + \mathcal O(t) \nonumber \\
	&=  V(c(\hat s_{i+1}(t)))^\top R_i \, c(\hat s_{i+1}(t)) - V(c(\hat s_{i}(t)))^\top R_i \, c(\hat s_{i}(t)) - 2 \int_{\hat s_i(t)}^{\hat s_{i+1}(t)} V(c(s))^\top  R_i \, \dot c(s) \, ds+ \mathcal O(t). \label{eq_Bi_t}
	\end{align}
	For $C_i(t)$, by the mean value theorem there exists $\tilde s_i \in [\hat s_i(t), s_i(t)]$ such that
	\begin{align}\label{eq_Ci_t}
        C_i(t) = \frac{1}{t} \int_{\hat s_i(t)}^{s_i(t)} c(s)^\top R_i \dot c(s) \, \mbox ds =  c(\tilde s_i)^\top R_i \dot c(\tilde s_i) \frac{1}{t}|s_i(t) - \hat s_i(t)|. 
	\end{align}
	
	On the other hand, we note that with \eqref{eq::defineIntersection} the following vector identity holds
	\begin{equation*}
	t V(c(\hat s_{i}(t))) = 
	c(s_{i}(t)) - c(\hat s_{i}(t)) = 
	\int_{\hat s_{i}(t)}^{s_{i}(t)} \dot c(s) \,ds =
	(s_{i}(t)-\hat s_{i}(t))\int_{0}^{1} \dot c(\hat s_{i}(t)+a(s_{i}(t)-\hat s_{i}(t))) \,da 
	\end{equation*}
	and thus,
	\begin{equation}\label{eq::proofShape_resultA}
	\dot c(\bar s_{i})\mylim \frac{1}{t}(s_{i}(t)-\hat s_{i}(t)) = \mylim \left( \frac{1}{t}(s_{i}(t)-\hat s_{i}(t))\int_{0}^{1} \dot c(\hat s_{i}(t)+a(s_{i}(t)-\hat s_{i}(t))) \,da \right) =  V(c(\bar s_{i})). 
	\end{equation}
	
	Combining \eqref{eq_Bi_t} and \eqref{eq_Ci_t} we get	
	\begin{align*}
        \mylim \frac1t |A_i(t)| =& \mylim \frac12 (B_i(t) - C_i(t)+ C_{i+1}(t))\\
        =& \mylim \frac12 \left( - 2 \int_{\hat s_i(t)}^{\hat s_{i+1}(t)}  V(c(s))^\top  R_i \, \dot c(s) \, ds+ \delta_i(t) - \delta_{i+1}(t) \right)
	\end{align*}
    with
    \begin{align*}
        \delta_i(t) &:=
         c(\hat s_{i}(t))^\top R_i V(c(\hat s_{i}(t)))  -  c(\tilde s_i)^\top R_i \dot c(\tilde s_i) \frac{1}{t}|s_i(t)- \hat s_i(t)| .
    \end{align*}
    From \eqref{eq::proofShape_resultA}, it follows that $ \mylim \delta_i(t) = 0 = \mylim \delta_{i+1}(t) $ and thus
    \begin{align*}
        \mylim \frac1t |A_i(t)| =&-  \int_{\bar s_i}^{\bar s_{i+1}}  V(c(s))^\top  R_i \, \dot c(s) \, ds = - \int_{c(\bar s_i)}^{c(\bar s_{i+1})} V(x) \cdot n_i(x) \, \mbox d S_x
    \end{align*}
    Here we used
    \begin{equation*}
        n_i(x) = \begin{cases}
                    n(x) & \mbox{ if } V(x) \cdot n(x) <0, \\
                    -n(x) & \mbox{ if } V(x) \cdot n(x) >0, 
                 \end{cases}
    \end{equation*}
    where $n$ denotes the unit normal vector pointing out of $\Omega$. Thus we have found
    \begin{equation*}
        \mylim \frac1t |\Omega_t \symmDiff \Omega| = \sum_{i=1}^N \mylim \frac1t |A_i(t)| = \int_{\partial \Omega} |V(x) \cdot n(x)| \; \mbox dS_x.
    \end{equation*}
	
%

\end{proof}

\subsection{Matrix entries for the numerical shape derivative}\label{app::shapeDerivative}
The matrix entries for $d_k\mathbf m_{l}^I = \detJl d_k\bar{\mathbf m}_{l}^I$ and $d_k\mathbf f_{l}^I = \detJl d_k\bar{\mathbf f}_{l}^I$ \eqref{eq::formulaShapeDerivative} are given by
{\tiny
	\begin{align*}
	d_k\bar m_l^{A^\pm}[1,1]&=\pm\frac{\phi_{l_1}^4\left(\phi_{l_2}^3+\phi_{l_2}^2\phi_{l_3}+\phi_{l_2}\phi_{l_3}^2+\phi_{l_3}^3\right)-4\phi_{l_1}\phi_{l_2}^3\phi_{l_3}^3+6\phi_{l_1}^2\phi_{l_2}^2\phi_{l_3}^2\left(\phi_{l_2}+\phi_{l_3}\right)-4\phi_{l_1}^3\phi_{l_2}\phi_{l_3}\left(\phi_{l_2}^2+\phi_{l_2}\phi_{l_3}+\phi_{l_3}^2\right)}{4{\left(\phi_{l_1}-\phi_{l_2}\right)}^4{\left(\phi_{l_1}-\phi_{l_3}\right)}^4},\\
d_k\bar m_l^{A^\pm}[1,2]&=\mp\frac{\phi_{l_1}^2\left(3\phi_{l_1}^2\phi_{l_2}^2+2\phi_{l_1}^2\phi_{l_2}\phi_{l_3}+\phi_{l_1}^2\phi_{l_3}^2-8\phi_{l_1}\phi_{l_2}^2\phi_{l_3}-4\phi_{l_1}\phi_{l_2}\phi_{l_3}^2+6\phi_{l_2}^2\phi_{l_3}^2\right)}{12{\left(\phi_{l_1}-\phi_{l_2}\right)}^4{\left(\phi_{l_1}-\phi_{l_3}\right)}^3}=d_k\bar m_l^{A^\pm}[2,1],\\
d_k\bar m_l^{A^\pm}[1,3]&=\mp\frac{\phi_{l_1}^2\left(\phi_{l_1}^2\phi_{l_2}^2+2\phi_{l_1}^2\phi_{l_2}\phi_{l_3}+3\phi_{l_1}^2\phi_{l_3}^2-4\phi_{l_1}\phi_{l_2}^2\phi_{l_3}-8\phi_{l_1}\phi_{l_2}\phi_{l_3}^2+6\phi_{l_2}^2\phi_{l_3}^2\right)}{12{\left(\phi_{l_1}-\phi_{l_2}\right)}^3{\left(\phi_{l_1}-\phi_{l_3}\right)}^4} = d_k\bar m_l^{A^\pm}[3,1],\\
d_k\bar m_l^{A^\pm}[2,2]&=\pm\frac{\phi_{l_1}^3\left(3\phi_{l_1}\phi_{l_2}+\phi_{l_1}\phi_{l_3}-4\phi_{l_2}\phi_{l_3}\right)}{12{\left(\phi_{l_1}-\phi_{l_2}\right)}^4{\left(\phi_{l_1}-\phi_{l_3}\right)}^2},\\
d_k\bar m_l^{A^\pm}[2,3]&=\pm\frac{\phi_{l_1}^3\left(\phi_{l_1}\phi_{l_2}+\phi_{l_1}\phi_{l_3}-2\phi_{l_2}\phi_{l_3}\right)}{12{\left(\phi_{l_1}-\phi_{l_2}\right)}^3{\left(\phi_{l_1}-\phi_{l_3}\right)}^3} = d_k\bar m_l^{A^\pm}[3,2],\\
d_k\bar m_l^{A^\pm}[3,3]&=\pm\frac{\phi_{l_1}^3\left(\phi_{l_1}\phi_{l_2}+3\phi_{l_1}\phi_{l_3}-4\phi_{l_2}\phi_{l_3}\right)}{12{\left(\phi_{l_1}-\phi_{l_2}\right)}^2{\left(\phi_{l_1}-\phi_{l_3}\right)}^4}
, \\
	d_k\bar m_l^{B^\pm}[1,1]&=\mp\frac{\phi_{l_2}^4}{4{\left(\phi_{l_1}-\phi_{l_2}\right)}^4\left(\phi_{l_2}-\phi_{l_3}\right)},\\
d_k\bar m_l^{B^\pm}[1,2]&=\pm\frac{\phi_{l_2}^3\left(3\phi_{l_1}\phi_{l_2}-4\phi_{l_1}\phi_{l_3}+\phi_{l_2}\phi_{l_3}\right)}{12{\left(\phi_{l_1}-\phi_{l_2}\right)}^4{\left(\phi_{l_2}-\phi_{l_3}\right)}^2} = d_k\bar m_l^{B^\pm}[2,1],\\
d_k\bar m_l^{B^\pm}[1,3]&=\pm\frac{\phi_{l_2}^4}{12{\left(\phi_{l_1}-\phi_{l_2}\right)}^3{\left(\phi_{l_2}-\phi_{l_3}\right)}^2}=d_k\bar m_l^{B^\pm}[3,1],\\
d_k\bar m_l^{B^\pm}[2,2]&=\mp\frac{\phi_{l_2}^2\left(3\phi_{l_1}^2\phi_{l_2}^2-8\phi_{l_1}^2\phi_{l_2}\phi_{l_3}+6\phi_{l_1}^2\phi_{l_3}^2+2\phi_{l_1}\phi_{l_2}^2\phi_{l_3}-4\phi_{l_1}\phi_{l_2}\phi_{l_3}^2+\phi_{l_2}^2\phi_{l_3}^2\right)}{12{\left(\phi_{l_1}-\phi_{l_2}\right)}^4{\left(\phi_{l_2}-\phi_{l_3}\right)}^3},\\
d_k\bar m_l^{B^\pm}[2,3]&=\mp\frac{\phi_{l_2}^3\left(\phi_{l_1}\phi_{l_2}-2\phi_{l_1}\phi_{l_3}+\phi_{l_2}\phi_{l_3}\right)}{12{\left(\phi_{l_1}-\phi_{l_2}\right)}^3{\left(\phi_{l_2}-\phi_{l_3}\right)}^3} = d_k\bar m_l^{B^\pm}[3,2],\\
d_k\bar m_l^{B^\pm}[3,3]&=\mp\frac{\phi_{l_2}^4}{12{\left(\phi_{l_1}-\phi_{l_2}\right)}^2{\left(\phi_{l_2}-\phi_{l_3}\right)}^3}
, \\
	d_k\bar m_l^{C^\pm}[1,1]&=\pm\frac{\phi_{l_3}^4}{4{\left(\phi _{1}-\phi _{3}\right)}^4\left(\phi _{2}-\phi _{3}\right)},\\
d_k\bar m_l^{C^\pm}[1,2]&=\pm\frac{\phi_{l_3}^4}{12{\left(\phi _{1}-\phi _{3}\right)}^3{\left(\phi _{2}-\phi _{3}\right)}^2} = d_k\bar m_l^{C^\pm}[2,1],\\
d_k\bar m_l^{C^\pm}[1,3]&=\pm\frac{\phi_{l_3}^3\left(3\phi _{1}\phi _{3}-4\phi _{1}\phi _{2}+\phi _{2}\phi _{3}\right)}{12{\left(\phi _{1}-\phi _{3}\right)}^4{\left(\phi _{2}-\phi _{3}\right)}^2} = d_k\bar m_l^{C^\pm}[3,1],\\
d_k\bar m_l^{C^\pm}[2,2]&=\pm\frac{\phi_{l_3}^4}{12{\left(\phi _{1}-\phi _{3}\right)}^2{\left(\phi _{2}-\phi _{3}\right)}^3},\\
d_k\bar m_l^{C^\pm}[2,3]&=\pm\frac{\phi_{l_3}^3\left(\phi _{1}\phi _{3}-2\phi _{1}\phi _{2}+\phi _{2}\phi _{3}\right)}{12{\left(\phi _{1}-\phi _{3}\right)}^3{\left(\phi _{2}-\phi _{3}\right)}^3} = d_k\bar m_l^{C^\pm}[3,2],\\
d_k\bar m_l^{C^\pm}[3,3]&=\pm\frac{\phi_{l_3}^2\left(6\phi_{l_1}^2\phi_{l_2}^2-8\phi_{l_1}^2\phi _{2}\phi _{3}+3\phi_{l_1}^2\phi_{l_3}^2-4\phi _{1}\phi_{l_2}^2\phi _{3}+2\phi _{1}\phi _{2}\phi_{l_3}^2+\phi_{l_2}^2\phi_{l_3}^2\right)}{12{\left(\phi _{1}-\phi _{3}\right)}^4{\left(\phi _{2}-\phi _{3}\right)}^3}
, \\
	d_k\bar f_l^{A^\pm}[1]&=\mp\frac{\phi_{l_1}\left(\phi_{l_1}^2\phi_{l_2}^2+\phi_{l_1}^2\phi_{l_2}\phi_{l_3}+\phi_{l_1}^2\phi_{l_3}^2-3\phi_{l_1}\phi_{l_2}^2\phi_{l_3}-3\phi_{l_1}\phi_{l_2}\phi_{l_3}^2+3\phi_{l_2}^2\phi_{l_3}^2\right)}{3{\left(\phi_{l_1}-\phi_{l_2}\right)}^3{\left(\phi_{l_1}-\phi_{l_3}\right)}^3},\\
d_k\bar f_l^{A^\pm}[2]&=\pm\frac{\phi_{l_1}^2\left(2\phi_{l_1}\phi_{l_2}+\phi_{l_1}\phi_{l_3}-3\phi_{l_2}\phi_{l_3}\right)}{6{\left(\phi_{l_1}-\phi_{l_2}\right)}^3{\left(\phi_{l_1}-\phi_{l_3}\right)}^2}, \qquad
d_k\bar f_l^{A^\pm}[3]=\pm\frac{\phi_{l_1}^2\left(\phi_{l_1}\phi_{l_2}+2\phi_{l_1}\phi_{l_3}-3\phi_{l_2}\phi_{l_3}\right)}{6{\left(\phi_{l_1}-\phi_{l_2}\right)}^2{\left(\phi_{l_1}-\phi_{l_3}\right)}^3}
, \\
	d_k\bar f_l^{B^\pm}[1]&=\pm\frac{\phi_{l_2}^3}{3{\left(\phi_{l_1}-\phi_{l_2}\right)}^3\left(\phi_{l_2}-\phi_{l_3}\right)}, \qquad
d_k\bar f_l^{B^\pm}[2]=\mp\frac{\phi_{l_2}^2\left(2\phi_{l_1}\phi_{l_2}-3\phi_{l_1}\phi_{l_3}+\phi_{l_2}\phi_{l_3}\right)}{6{\left(\phi_{l_1}-\phi_{l_2}\right)}^3{\left(\phi_{l_2}-\phi_{l_3}\right)}^2}, \qquad
d_k\bar f_l^{B^\pm}[3]=\mp\frac{\phi_{l_2}^3}{6{\left(\phi_{l_1}-\phi_{l_2}\right)}^2{\left(\phi_{l_2}-\phi_{l_3}\right)}^2}
, \\
	d_k\bar f_l^{C^\pm}[1]&=\mp\frac{\phi_{l_3}^3}{3{\left(\phi_{l_1}-\phi_{l_3}\right)}^3\left(\phi_{l_2}-\phi_{l_3}\right)}, \qquad
d_k\bar f_l^{C^\pm}[2]=\mp\frac{\phi_{l_3}^3}{6{\left(\phi_{l_1}-\phi_{l_3}\right)}^2{\left(\phi_{l_2}-\phi_{l_3}\right)}^2}, \qquad
d_k\bar f_l^{C^\pm}[3]=\mp\frac{\phi_{l_3}^2\left(2\phi_{l_1}\phi_{l_3}-3\phi_{l_1}\phi_{l_2}+\phi_{l_2}\phi_{l_3}\right)}{6{\left(\phi_{l_1}-\phi_{l_3}\right)}^3{\left(\phi_{l_2}-\phi_{l_3}\right)}^2}
.
	\end{align*}}




\end{document}